\documentclass[11pt,a4paper]{amsart}
\usepackage{verbatim}

\usepackage[toc]{appendix}
\usepackage[T1]{fontenc}

\usepackage{graphicx}
\usepackage{enumerate}
\usepackage{amsmath,amsfonts,amssymb}
\usepackage{color}
\def\loc{\operatorname{loc}}
\usepackage{cite}
\usepackage{ latexsym }
\definecolor{citation}{rgb}{0.11,0.67,0.84}
\definecolor{formula}{rgb}{0.1,0.2,0.6}
\definecolor{url}{rgb}{0.11,0.67,0.84}
\usepackage{pgf,tikz}
\usepackage{mathrsfs}
\usepackage{fancyhdr}
\usepackage{dutchcal}

\usepackage{dsfont}

\newcommand{\reqnomode}{\tagsleft@false}

\usepackage{hyperref}

\def\aax{\mathfrak{a}}

\newcommand\ccc{\mathfrak{c}}

\def\dx{\,{\rm d}x}

\def\dy{\,{\rm d}y}
\def\ttl{\textnormal{\texttt{t}}}

\def \d{\,{\rm d}}
\def \diver{\,{\rm div}}
\def\dist{\,{\rm dist}}

\def\supp{\,{\rm supp}}

\def\rrs{\rr_{\textnormal{s}}}
\def\rri{\rr_{\textnormal{i}}}
\def\rs{r_{\textnormal{s}}}
\def\ri{r_{\textnormal{i}}}
\def\kkk{{\textnormal{\texttt{k}}}}

\def\deb{\rightharpoonup}

\allowdisplaybreaks
\makeatletter
\DeclareRobustCommand*{\bfseries}{%
  \not@math@alphabet\bfseries\mathbf
  \fontseries\bfdefault\selectfont
  \boldmath
}

\makeatother

\newlength{\defbaselineskip}
\newcommand{\setlinespacing}[1]
           {\setlength{\baselineskip}{#1 \defbaselineskip}}
\newcommand{\mint}{\mathop{\int\hskip -1,05em -\, \!\!\!}\nolimits}

\newtheorem{theorem}{Theorem}
\newtheorem{corollary}{Corollary}
\newtheorem{definition}{Definition}
\newtheorem{remark}{Remark}[section]
\newtheorem{lemma}{Lemma}[section]
\newtheorem{proposition}{Proposition}[section]
\numberwithin{equation}{section}
\allowdisplaybreaks[4]

\newcommand{\tta}{s}

\newcommand{\hhh}{\textnormal{\texttt{h}}}
\newcommand{\mmm}{m}

\newcommand{\ppp}{\mathfrak p}
\newcommand{\qqq}{\mathfrak q}

\newcommand{\kk}{\kappa}

\def\en{\mathbb N}

\def\er{\mathbb R}

\newcommand{\App}[1]{\textnormal{App}{(#1)}}
\newcommand{\ti}[1]{\tilde{#1}}
\newcommand{\mf}[1]{\mathfrak{#1}}

\newcommand{\BB}{\mathcal{B}_1}

\newcommand\eps\varepsilon

\def\eqn#1$$#2$${\begin{equation}\label#1#2\end{equation}}

\newcommand{\be}{\begin{equation}}
\newcommand{\ee}{\end{equation}}

\newcommand{\rr}{\varrho}

\newcommand{\snr}[1]{\lvert #1\rvert}

\newcommand{\nr}[1]{\lVert #1 \rVert}

\newcommand{\N}{\mathbb{N}}

\def\name[#1, #2]{#1 #2}

\newcommand{\sss}{\mf{s}}

\newcommand{\rif}[1]{(\ref{#1})}
\newcommand{\trif}[1] {\textnormal{\rif{#1}}}

\newcommand{\stackleq}[1]{\stackrel{\rif{#1}}{ \leq}}

\title[The nonuniformly elliptic sharp growth rate]{The sharp growth rate in\\ nonuniformly elliptic Schauder theory}

\author[De Filippis]{Cristiana De Filippis}  \address{Cristiana De Filippis\\Dipartimento SMFI, Universit\'a di Parma\\ Parco Area delle Scienze 53/A, 43124 Parma, Italy} \email{\url{cristiana.defilippis@unipr.it}}
\author[Mingione]{Giuseppe Mingione}  \address{Giuseppe Mingione\\Dipartimento SMFI, Universit\`a di Parma, Viale delle Scienze 53/a, Campus, 43124 Parma, Italy} \email{\url{giuseppe.mingione@unipr.it}}

\begin{document}

\vspace{1cm}

\subjclass[2020]{49N60, 35J60 \vspace{1mm}} 

\keywords{Regularity, Schauder estimates, Nonuniform ellipticity\vspace{1mm}}

Duke Math. J., to appear. 

\vspace{6mm}

\maketitle
\begin{center}
\emph{To Masashi Misawa, sensei of regularity, on his 60th birthday.}
\end{center}

\begin{abstract}
Schauder estimates hold in nonuniformly elliptic problems under optimal assumptions on the growth of the  ellipticity ratio. 
 \end{abstract}
\setcounter{tocdepth}{1}
{\small \tableofcontents}

\setlinespacing{1.00}

\section{Introduction}
In this paper we finally bring to a conclusion the longstanding quest for Schauder estimates in the setting of nonuniformly elliptic problems. Indeed, in a first step \cite{piovra} we established local Schauder theory assuming a certain polynomial growth rate on the ellipticity ratio, which is the central and necessary condition in this framework. Here we prove Schauder estimates assuming the best possible growth rate on such a quantity. Note that, while Schauder estimates have always been a classical issue sitting at the roots of nonuniformly elliptic theory, the optimal growth rate on the ellipticity ratio, displayed in \rif{rate} and \rif{pq} below, was first discovered around twenty years ago in \cite{sharp, FMM}, where counterexamples to regularity were found when that rate not was met. At that time, only gradient integrability results for solutions were established. 
Here we finally prove maximal regularity, that is, local gradient H\"older continuity of solutions provided coefficients are H\"older continuous, using the above mentioned optimal conditions. To fix the ideas, let us consider the key model case
\eqn{modello}
$$
w \mapsto \mathcal G(w, \Omega ):=\int_{\Omega}\ccc(x)G(Dw) \dx \,,\qquad \Omega\subset \er^n, \ \  n \geq 2\,,
$$
with $\Omega$ being a bounded, open domain, where $G(\cdot)$ is a convex and non-negative integrand with growth conditions on the lowest and largest eigenvalues dictated by 
\eqn{sottosopra}
$$
|z|^{p-2} \mathds{I}_{\rm d} \lesssim \partial_{zz}G(z)  \lesssim |z|^{q-2} \mathds{I}_{\rm d}\,,\qquad \mbox{for $|z|\geq 1$}\,,
\qquad 1<p \leq q\,.$$
Here the main point is the H\"older regularity of coefficients, i.e., $\ccc(\cdot)$ is assumed to satisfy
\eqn{coeff}
$$
|\ccc(x_1)-\ccc(x_2)| \leq L|x_1-x_2|^{\alpha}, \quad  \alpha \in (0,1]\,, \qquad  1 \leq \ccc(\cdot) \leq L\,.
$$
The nonuniform ellipticity of the functional $\mathcal G$, and therefore of the related Euler-Lagrange equation 
\eqn{el0}
$$
-\diver\, (\ccc(x)\partial_{z}G(Du))=0\,,
$$
stems from the fact the ellipticity ratio of $\partial_{z}G(\cdot)$, which is given by
\eqn{ratioG}
$$
\mathcal R_{\partial_{z}G}(z):= \frac{\mbox{highest eigenvalue of}\ \partial_{zz}  G(z)}{\mbox{lowest eigenvalue of}\  \partial_{zz} G(z)} \,,
$$
might grow with polynomial rate
\eqn{rate}
$$
\mathcal R_{\partial_{z}G}(z) \lesssim |z|^{q-p} \,,\qquad \mbox{for $|z|\geq 1$}
$$
and therefore might not be a bounded quantity on $\{|z|\geq 1\}$ when $p<q$. The unboundedness of $\mathcal R_{\partial_{z}G}(z)$ for $|z|$ large is precisely  the meaning of nonuniform ellipticity of the equation in \rif{el0}, as described in \cite{ivanov2, LUcpam, tru1967, simon0, simon1, piovra, ciccio, UU}. We refer to \rif{assf} for the complete set of assumptions on $G(\cdot)$ and, in particular, for the precise formulation of \rif{sottosopra}. 
Those satisfying \rif{sottosopra} are also called integrands with $(p,q)$-growth (nonstandard) conditions, a terminology introduced by Marcellini in \cite{ma2}. The point is now to establish when the H\"older continuity of coefficients \rif{coeff} implies the local H\"older continuity of the gradient of minima and/or solutions to \rif{el0}, that is, to establish the validity of (local) Schauder theory. 
In the linear case $\diver\,(\texttt{A}(x)Du)=0$, with $\texttt{A}(x)$ being an elliptic matrix with H\"older continuous entries, Schauder estimates are originally a result of Hopf \cite{hopf} in the interior case, and of Caccioppoli \cite{cacc1} and Schauder \cite{js, js2} in the global one. There are nowadays several, elegant perturbative methods to achieve such results: via freezing-of-coefficients and use of proper function spaces and decay lemmas \cite{camp}, via convolution methods \cite{trusc}, or via blow-up \cite{simon2}. Nonlinear Schauder estimates have been the object of intensive investigation. In the uniformly elliptic case $p=q$ they were established in the eighties in \cite{gg2,gg3,liebe,manth1,manth2, misawa}, also in the specific framework of the Calculus of Variations. We stress that, when dealing with uniformly elliptic problems, both in the linear and in the nonlinear case, Schauder theory is a perturbation theory, in the sense that, independently of the specific method of proof employed, estimates are achieved by comparison schemes with solutions to suitable problems without coefficients (frozen problems).  

The nonuniformly elliptic case is instead a different story. In this situation the nature of Schauder estimates is not any longer perturbative. In fact, there are examples of convex functionals with $\alpha$-H\"older coefficients and $(p,q)$-growth conditions, with $p$ and $q$ that can be chosen arbitrarily close in dependence of $n$ and $ \alpha$, and possessing very wild minimizers, i.e., their discontinuities are concentrated on fractals \cite{sharp, FMM} with maximal Hausdorff dimension $n-p$ \cite{balci2, balci3}; see Remark \ref{bere} for more details. In other words, no matter coefficients are H\"older, solutions can be {\em as bad} as any other competitor or admissible test function. On top of this, note that the examples in \cite{balci2, balci3, sharp, FMM} deal with equations exhibiting only a soft form of nonuniform ellipticity as detailed in \cite{ciccio}, the general situation being potentially worse. The question of the validity of Schauder theory in the nonuniformly elliptic case remained a discussed open problem for a long time; see for instance the comments in \cite{gg3, liebe3} and \cite[page 40]{ivanov2}. We also refer to \cite[Introduction]{piovra} for a wider discussion and a historical account of the issue. As a matter of fact, in nonuniformly elliptic problems the dependence on coefficients was never allowed to be H\"older continuous and additional, often unnatural assumptions, had to be considered \cite{liebe0, ivanov0, ivanov1, LUcpam, ma2, simon0, simon1}. In \cite{piovra} we finally showed the validity of Schauder estimates under the following main assumption on the gap $q/p$:
\eqn{vecchiobound}
$$
\frac qp  < 1+  \frac{\alpha^2}{5n^2} \,.
$$
Assumptions of the type \rif{vecchiobound} are necessary and central already in the autonomous case by counterexamples \cite{gia, ma2}, and ensure that the rate of blow-up of the ellipticity ratio \rif{ratioG} is not too fast at infinity, as implied by \rif{rate} in connection to \rif{vecchiobound}. The basic question then remains about the optimality of \rif{vecchiobound}. In this paper we finally prove that local H\"older continuity of the gradient of minimizers holds provided 
\eqn{pq}
$$
\frac{q}{p}<1+\frac{\alpha}{n}\,.
$$
Condition \rif{pq} is sharp. This is shown by the aforementioned examples in \cite{sharp, FMM}. The bound \rif{pq} was already used to establish local H\"older continuity of solutions to certain elliptic equations with nonstandard growth conditions, but only under very special structure assumptions making the problems in question still uniformly elliptic. See for instance \cite{BCM, CM, HO1, HO2,HO3} and related references, and the discussion in \cite{ciccio, cimemin}. Further special structures are imposed in \cite{logdoppio} and concern the borderline case $p=1$; see also \cite{BSc,gme0, gmearma, gme1, gme2} for problems with linear and nearly linear growth. The results in this paper treat the most general case of nonuniformly elliptic problems with polynomial growth, and are described in the following sections. No special, additional structures are considered. Needless to say, \rif{pq} describes a phenomenon that is absent in the classical uniformly elliptic case, i.e., a sharp interaction between the growth of the ellipticity ratio, and the rate of regularity of coefficients. 

Let us mention that, whilst we are basically dealing with the variational case, our techniques provide similar results in the case of equations $\diver\, A(x, Du)=0$. In fact, in order to derive all our regularity estimates we basically use the Euler-Lagrange equation. Specifically, results for equations can be obtained combining our  methods with the approach of \cite[Section 2.6]{piovra}. Here we prefer to expand on variational integrals as in this case the notion of solution is more natural and comes along without ambiguities. For instance, considering the variational case dispenses us from discussing the notion of energy solution, that becomes less clear in the nonuniformly elliptic setting, where regularity results are obtained simultaneously to existence ones, and not for general solutions. We refer to \cite{BM, piovra, ma2, masurvey} for such aspects and for more discussion on the case of equations.

\subsection{Results}  For the rest of the paper we fix a set of parameters $n, p, q,\alpha$, where $n\geq 2$ is an integer, $1<p\leq q$ 
and $\alpha \in (0,1]$ are real numbers. We also denote 
\eqn{defiH}
$$\mbox{$H_{\mu}(z):=\snr{z}^{2}+\mu^{2}$, \ \  for $z\in \er^n$ \ 
and $\mu\geq0$\,.}
$$ 
More notation can be found in Section \ref{notazioni}. We consider integral functionals of the type
\eqn{ggg2}
$$
W^{1,1}(\Omega) \ni w\mapsto \mathcal{F}(w,\Omega):=\int_{\Omega}F(x,Dw) \dx\,,
$$
assuming that the Carath\'eodory-regular integrand $F\colon \Omega \times \er^n \to \er$
satisfies \begin{flalign}\label{assf}
\begin{cases}
\, \tilde{z}\mapsto F(\cdot,\tilde{z})\in C^{2}_{\loc}(\mathbb{R}^{n}\setminus \{0_{\er^n}\})\cap C^{1}_{\loc}(\mathbb{R}^{n})\vspace{2mm}\\
\, H_{\mu}(z)^{p/2}\le F(x,z)\le LH_{\mu}(z)^{q/2}+LH_{\mu}(z)^{p/2}\vspace{2mm}\\
\, H_{\mu}(z)^{(p-2)/2}\snr{\xi}^{2}\le \partial_{zz}F(x,z)\xi \cdot \xi\vspace{2mm}\\
\, \snr{\partial_{zz}F(x,z)}\le L H_{\mu}(z)^{(q-2)/2}+LH_{\mu}(z)^{(p-2)/2}\vspace{2mm}\\
\, \snr{\partial_{z} F(x_{1},z)-\partial_{z}F(x_{2},z)}\le L \snr{x_{1}-x_{2}}^{\alpha}[H_{\mu}(z)^{(q-1)/2}+H_{\mu}(z)^{(p-1)/2}]
\end{cases}
 \end{flalign}
for all $z, \xi\in \mathbb{R}^{n}$, $|z|\not=0$ in the case of \rif{assf}$_{3,4}$, $x\in \Omega$, where $0 \leq \mu\leq 1 \leq L$. 
We shall explicitly mention when less assumptions than the full \rif{assf} will be needed on $F(\cdot)$. Conditions \rif{assf}$_{3,4}$ imply the following growth rate of the ellipticity ratio of $\partial_{z} F(\cdot)$:
\eqn{ratio}
$$
\mathcal R_{\partial_{z}F(x, \cdot)}(z):= \frac{\mbox{highest eigenvalue of}\ \partial_{zz}  F(x,z)}{\mbox{lowest eigenvalue of}\  \partial_{zz} F(x,z)}  \lesssim |z|^{q-p} \,,\qquad \mbox{for $|z|\geq 1$}
$$
as in \rif{rate}. 
The case of the functional \rif{modello} therefore follows taking $F(x, z)\equiv \ccc(x)G(z)$. 
\begin{definition}\label{defi-min} Under assumption \eqref{assf}$_2$, a function $u \in W^{1,p}_{\loc}(\Omega)$ is a \emph{(local) minimizer} of the functional $\mathcal F$ in \eqref{ggg2} if, for every ball ${\rm B}\Subset \Omega$, we have $F(\cdot, Du) \in L^1({\rm B})$ and $\mathcal F(u,{\rm B})\leq \mathcal F(w,{\rm B})$ holds for every $w \in u + W^{1,1}_0({\rm B})$. 
\end{definition}
\vspace{1mm}
For nonautonomous functionals of the type in \rif{ggg2} the analysis of regularity of minima unavoidably passes through suitable relaxed functionals. This is suggested by the potential occurrence of the so-called Lavrentiev phenomenon. In fact, in the present setting it may happen that 
\eqn{l.0}
$$
\inf_{w\in u_{0}+W^{1,p}_{0}({\rm B})}\mathcal{F}(w,{\rm B})<\inf_{w\in u_{0}+W^{1,p}_{0}({\rm B})\cap 
W^{1,q} ({\rm B})}\mathcal{F}(w,{\rm B})
$$
when $p<q$, as for instance shown in \cite{sharp, Z0, Z1}. This is an obvious obstruction to regularity of minima and under assumptions \rif{assf} never occurs when the integrand $F(\cdot)$ is independent of $x$, that is, in the autonomous case \cite{sharp}. The idea is to pass to a relaxed form of \rif{ggg2}, which is able to overcome the strict inequality in \rif{l.0} and to catch the minimality features at the level of a priori estimates. In this setting, this idea was first introduced by Marcellini in \cite{ma0, madg} to describe cavitation phenomena in Nonlinear Elasticity, connected to analogous, earlier constructions devised by Lebesgue \cite{lebesgue} and Serrin \cite{serrin}. Such functionals have been since then extensively treated in the literature; see for instance \cite{BFM, ma1, sharp, piovra, camel, FMa, FMal, gme3,  thomas}. 
\vspace{1mm}
\begin{definition}[Relaxation]\label{LSMdef}
The Lebesgue-Serrin-Marcellini (LSM) relaxation of the functional $\mathcal F$ in \eqref{ggg2} under assumption \eqref{assf}$_2$ is defined as 
\eqn{LSM}
$$
\bar{\mathcal F}(w,U)  \, := \inf_{\{w_{j}\}\subset W^{1,q} (U)}  \left\{ \liminf_j \mathcal F(w_{j},U)  \, \colon \,  w_{j} \deb w\  \mbox{in} \ W^{1,p}(U)  \right\} \,,
$$
whenever $w \in W^{1,p}(U)$ and $U\Subset \Omega$ is an open subset.
\end{definition}
\vspace{1mm}
See also Remark \ref{alternativa} for variants and comments on the relaxed functional $\bar{\mathcal F}$. Note that $\mathcal F(w,{\rm B})< \infty$ for every ball ${\rm B}\Subset \Omega$ automatically implies $w \in W^{1,p}_{\loc}(\Omega)$. 
\begin{definition}[Relaxed Minima]\label{defi-min2} Under assumption \eqref{assf}$_2$, a  function $u \in W^{1,p}_{\loc}(\Omega)$ is a \emph{(local) minimizer} of the functional $\bar{\mathcal F}$ in \eqref{LSM} if, for every ball ${\rm B}\Subset \Omega$, we have $\bar{\mathcal F}(u,{\rm B})<\infty $ and $\bar{\mathcal F}(u,{\rm B})\leq \bar{\mathcal F}(w,{\rm B})$ holds for every $w \in u + W^{1,p}_0({\rm B})$. 
\end{definition}
Assuming \eqref{assf} implies convexity of $z \mapsto F(\cdot, z)$. It follows that $\mathcal F$ turns out to be weakly $W^{1,p}$-lower semicontinuous \cite[Theorem 4.5]{giu} and therefore $\mathcal F \leq \bar{\mathcal F}$ on $W^{1,p}$ and $\bar{\mathcal F}=\mathcal F$ on $W^{1,q}$ (see Lemma \ref{bf.0}). In this case we define the Lavrentiev gap functional $\mathcal L_{\mathcal F}$ as 
$$
0 \leq \mathcal{L}_{\mathcal F}(w,{\rm B}):=\begin{cases}
\, \bar{\mathcal{F}}(w,{\rm B})-\mathcal{F}(w,{\rm B})\quad &\mbox{if} \ \ \mathcal{F}(w,{\rm B})<\infty\\
\, 0\quad &\mbox{if} \ \ \mathcal{F}(w,{\rm B})\equiv \infty
\end{cases}
$$
whenever ${\rm B} \Subset \Omega$ is a ball and $w\in W^{1,p}(\rm B)$. This is bound to provide a quantitative measure of the occurrence of \rif{l.0} and in fact $\mathcal{L}_{\mathcal F}(w,{\rm B})=0$ when $w \in W^{1,q}({\rm B})$. 
An approach that became standard since \cite{sharp, ma1} is then to first prove regularity of (local) minimizers $u$ of  $\bar{\mathcal F}$ in order to skip the obstruction given by \rif{l.0}. Then, when the Lavrentiev gap functional vanishes, that is, when
\eqn{vanish}
$$\mbox{$\mathcal{L}_{\mathcal F}(u,{\rm B})=0$ \ \ holds for all balls ${\rm B}\Subset \Omega$}\,, $$ 
it then happens that $u$ is also a minimizer of the LSM-relaxation $\bar{\mathcal{F}}$, indeed
\eqn{newminima}
$$
\bar{\mathcal{F}}(u, {\rm B}) = \mathcal{F}(u,{\rm B})+\mathcal{L}_{\mathcal F}(u,{\rm B}) = \mathcal{F}(u,{\rm B}) \leq \mathcal{F}(w,{\rm B})\leq \bar{\mathcal{F}}(w, {\rm B}) 
$$
holds whenever $w\in u+W^{1,p}_0({\rm B})$. The outcome is that in this case regularity results for the original functional follow from those for the relaxed one. 
\begin{theorem}\label{main}
Let $u\in W^{1,p}_{\loc}(\Omega)$ be a minimizer of the LSM-relaxation $\bar{\mathcal{F}}$ of $\mathcal{F}$, under assumptions \eqref{pq} and \eqref{assf}. Then $u$ is also a minimizer of $\mathcal F$ and, with $B_{\rr}\Subset \Omega$ being a ball, $\rr \in (0,1]$, and $\theta\in (0,1)$, 
\begin{itemize} 
\item The local Lipschitz estimate
\eqn{m.1}
$$
\nr{Du}_{L^{\infty}(B_{\rr/2})} \le c\rr^{-n\mf{b}}\bar{\mathcal{F}}(u,B_{\rr})^{\mf{b}}+c
$$
holds with $c\equiv c(n,p,q,\alpha,L)\geq1$, $\mf{b}\equiv \mf{b}(n,p,q,\alpha)\geq 1$. 
\item (Degenerate Schauder) The local H\"older estimate
\eqn{m.2}
$$
[Du]_{0,\alpha_*;B_{\theta \rr}}\le c
$$
holds with both $\alpha_*\in (0,1)$ and $c\geq 1$ depending on $n,p,q,\alpha,L,\theta, \rr$ and $\bar{\mathcal{F}}(u,B_{\rr})$.
\item (Schauder) If, in addition, $\mu >0$ and $\partial_{zz}F(\cdot)$ is continuous, then $u\in C^{1,\alpha}_{\loc}(\Omega)$ when $\alpha <1$. Specifically, 
\eqn{m.3}
$$
[Du]_{0,\alpha;B_{\theta \rr}}\le c_{*}
$$
holds as in \eqref{m.2}, where $c_{*}$ depends also on the modulus of continuity of $\partial_{zz}F(\cdot)$ and on $\mu$\footnote{Specifically, since our results are local in nature, we shall use the following uniform continuity: for every $H\geq 1$, there is a modulus of continuity $\omega_{H}(\cdot)$ such that $\snr{\partial_{zz}F(x_1,z_1)-\partial_{zz}F(x_2,z_2)}\leq \omega_H(\snr{x_1-x_2}+ 
\snr{z_1-z_2})$ whenever $x_1, x_2\in \Omega$ and $\snr{z_1}, \snr{z_2} \leq H$. Notice that the continuity assumption on $\partial_{zz}F(\cdot)$ can actually be weakened; for this see Remark \ref{remarkfinale}.}.  
\end{itemize}\end{theorem}
\begin{corollary}\label{mfmf} 
Theorem \ref{main} extends to any minimizer $u$ of the original functional $\mathcal{F}$ provided \trif{vanish} holds. \end{corollary}
Theorem \ref{main} and Corollary \ref{mfmf} indicate that the relaxed functional $\bar{\mathcal F}$ provides a selection principle capable of identifying/producing regular minimizers of the original functional  $\mathcal F$, while excluding those that are irregular due to the appearance of the Lavrentiev phenomenon \rif{l.0}. As for the occurrence of \rif{vanish}, let us remark that this is satisfied by large families of functionals \cite[Section 5]{sharp} and can easily be verified by adding rather mild growth conditions of the integrand, for instance. Indeed, this is the case when conditions of the type 
\eqn{implica}
$$
\ccc(x)\mathcal H(z) \lesssim  F(x, z) \lesssim \ccc(x)\mathcal  H(z) +1$$
hold, where $\mathcal H(\cdot)$ is a non-negative convex integrand and $\ccc(\cdot)$ is measurable and bounded away from zero and infinity; see \cite[Section 5]{sharp}. This follows by a simple convolution argument and does not even need neither \rif{pq} or \rif{assf}, and not even convexity of $z \mapsto F(\cdot, z)$. As a consequence, Corollary \ref{mfmf} covers the functional $\mathcal G$ in \eqref{modello}, provided $G(\cdot)$ satisfies \eqref{assf}$_{1-4}$, $\ccc(\cdot)$ is as in \rif{coeff} and \rif{pq} is assumed. It often happens that the very same assumptions, like for instance \rif{pq}, guaranteeing a priori regularity estimates \rif{m.1}, also imply that the Lavrentiev gap functional vanishes, so that the circle closes \cite{piovra, sharp}. This happens for instance if we assume 
\eqn{implica-2}
$$
\ccc(x)\left(\snr{z}^p+\aax(x)\snr{z}^q\right) \lesssim  F(x, z) \lesssim \ccc(x)\left(\snr{z}^p+\aax(x)\snr{z}^q\right) +1
$$
where $0\leq \aax(\cdot)\in C^{0,\alpha}(\Omega)$ and \rif{pq} holds \cite[Section 5]{sharp}, \cite[Theorem 4]{BCM}; again at this stage convexity of $z \mapsto F(\cdot, z)$ is not required. This is the case of double phase integrands \cite{BCM, CM} and shows a concrete example for which \rif{pq} simultaneously work as both a regularity condition and as a condition ensuring approximation in energy \rif{vanish}. Many more situations in which \rif{vanish} holds can be found in \cite[Section 5]{sharp} and in related literature on the Lavrentiev phenomenon; see for instance the recent \cite{iwo, koch} and related references. On the other hand, counterexamples to Schauder estimates, as those initially considered in \cite{sharp, FMM}, occur in connection to the Lavrentiev phenomenon \rif{l.0}, as pointed out in the pioneering papers of Zhikov \cite{Z0, Z1}. For more on these aspects we refer to Remark \ref{bere} at the end of this paper. Comments on the gap bound \rif{pq} can be found in Section \ref{proof6}. In the third point of Theorem \ref{main} the assumption $\mu>0$, making the equation non-degenerate, is necessary in order to reach the optimal H\"older exponent $\alpha$. This is known by classical counterexamples working already in the uniformly elliptic case of the  $p$-Laplacean equation $\diver\, (|Du|^{p-2}Du)=0$ \cite{lewis, krol}, which is indeed covered here taking $\mu=0$ and $p=q$. 
\subsection{A proof road map}\label{prevsec} The path leading to Theorem \ref{main} is technically not elementary and establishes a technique that we believe might in the future become standard in a number of problems with lack of ellipticity and/or regularity of external ingredients. The main steps can be summarized as follows, and we confine ourselves to describe how a priori estimates are obtained. 
\begin{itemize}
\item A preliminary higher integrability result. This is in Proposition \ref{scat}. We basically prove that condition \rif{pq} implies that the gradient of minimizers locally belongs to $L^{\qqq}$, for any finite $\qqq\geq 1$. In fact, we prove something more. Via a delicate use of interpolation inequalities we show that assuming a relaxed bound on $q/p$ - see \rif{pqb} - implies that any (locally) {\em bounded} minimizer has gradient which is locally integrable with arbitrary large power. In this respect, conditions \rif{pqb} are again optimal by counterexamples working for bounded minimizers \cite{sharp, FMM}. This extends a series of classical results, starting by the work of Ladyzhenskaya \& Ural'tseva \cite{LUcpam}, see also \cite{ch, UU}, where non-dimensional gap bounds are used in presence of bounded minimizers. The outcome is separately stated in Theorem \ref{main3} below. Proposition \ref{scat} can now be used in the proof of Theorem \ref{main} as, assuming \rif{pq}, implies both \rif{pqb} and that minima are (locally) bounded. 
\item The higher integrability estimates of the previous point are achieved taking as a blueprint a few Besov spaces techniques that are sometimes employed in uniformly elliptic problems with a certain lack of ellipticity. Examples are elliptic equations in the Heisenberg group \cite{d}, or, and especially, in the nonlocal setting \cite{bl, bls, gl}. The intuition is that a certain number of approaches, bound to compensate the lack of strong form of uniform ellipticity, might also be forced to compensate the presence of nonuniform ellipticity. This is very delicate to handle and requires extreme care and a sharp numerology in order to keep the optimal conditions \rif{pqb}; see in particular Section \ref{fm} and \rif{numerology}. Specifically, we prove that
\eqn{derive}
$$
\left\|\frac{\tau_{h}^{2}u}{\snr{h}^{s_i}}\right\|_{L^{t_{i}}} \leq c_i
$$
holds for $h \in \er^n$ and $|h|$ suitably small and with sequences
$
t_{i} \to \infty$, $1< s_i  \to 1$ and $c_i\to \infty$. Here $\tau_{h}$ denotes the standard finite difference operator and $\tau_{h}^{2}$ its iteration - see \rif{basilari}. Using basic embedding properties of Besov spaces, from \rif{derive} we deduce that $Du \in L^{t_{i}}_{\loc}$ (see Lemma \ref{bllema}). In \rif{derive} it is 
\eqn{moserlineare}
$$
t_{i+1} = t_{i} + \textnormal{\texttt{s}}(p+\gamma -q) 
$$
for some $\textnormal{\texttt{s}} >0$, where $\gamma$ appears in \rif{pqb}. Notice that it is precisely the occurrence of \rif{pqb} to guarantee that $\{t_{i}\}$ diverges, therefore implying that $Du$ is locally integrable with any power. Estimates leading to \rif{derive} implement a kind of  fractional Moser type iteration but do not lead directly to $Du \in L^{\infty}_{\loc}$, essentially because the growth in the gradient integrability exponent in \rif{moserlineare} is linear rather than geometric. 
\item A fractional Caccioppoli inequality. This is in Lemma \ref{fracacci} and it is aimed at implementing a fractional version of the Bernstein technique, upgrading the methods from \cite{piovra} with new estimates. In its standard formulation, for autonomous problems the classical Bernstein method prescribes to differentiate the Euler-Lagrange equation
$
\diver\, \partial_{z}F(Du)=0
$
in order to show that suitable convex functions of $|Du|$, as for example $|Du|^p$, are subsolutions to linear elliptic equations. See for instance the classical work of Uhlenbeck \cite{uh}. This implies the local boundedness of $Du$ as subsolutions satisfy level sets energy inequalities
\eqn{classica}
$$
\rr^{2}\mint_{B_{\rr/2}}\snr{D(|Du|^p-\kk)_{+}}^{2}\dx\lesssim  \mint_{B_{\rr}}(|Du|^p-\kk)_{+}^{2}\dx, \quad \forall \ \kk\geq 0
$$
so that one can run De Giorgi's iteration scheme; here it is $(|Du|^p-\kk)_{+}:=\max\{|Du|^p-\kk,0\}$. This method cannot be implemented in presence of H\"older coefficients as the Euler-Lagrange equation 
$\diver\, \partial_{z}F(x,Du)=0$ is this time non-differentiable. Therefore the idea is to read H\"older continuity of coefficients as fractional differentiability of the Euler-Lagrange equation and to get a fractional energy inequality of the type 
\eqn{fraziona}
$$
 \rr^{2\beta-n} [(|Du|^p-\kk)_{+}]_{\beta,2;B_{\rr/2}}^2\lesssim M^{\textnormal{\texttt{b}}} \mint_{B_{\rr}}(|Du|^p-\kk)_{+}^{2}\dx + M^{\textnormal{\texttt{b}}} \mathfrak{T}(B_{\rr})  $$
provided $ \|Du\|_{L^\infty(B_{\rr})}\leq M$. The main differences between the classical \rif{classica} and the new \rif{fraziona} are three. Fractional derivatives in the form of Gagliardo norms of $(|Du|^p-\kk)_+$, here in \rif{gaglia} below, occur in the left-hand side of \rif{fraziona}, instead of full derivatives as in \rif{classica}. The additional terms $\mathfrak{T}(B_{\rr})$; these encode the presence of coefficients, and therefore, in this setting, are the ones to handle with greater  care. Finally, the presence of $M$ in the right-hand side, that allows  \rif{fraziona} to look {\em homogeneous} with respect to $|Du|^p$ and therefore to be of the type occurring in the uniformly elliptic setting as \rif{classica}. This is why we call \rif{fraziona} a Renomalized Caccioppoli inequality. The price we pay for this is indeed the appearance of $M\approx \|Du\|_{L^\infty}$ in the right-hand side of \rif{fraziona}. In this respect the exponent $\textnormal{\texttt{b}}\equiv \textnormal{\texttt{b}}(n,p,q,\alpha)>0$ encodes the polynomial growth of the ellipticity ratio \rif{ratio} and $\beta\equiv \beta(\alpha)\in (0,1)$ the rate of H\"older continuity of coefficients. Connecting these aspects is delicate and the ultimate point is that \rif{pq} allows a sharp interplay leading to $L^{\infty}$-bounds for $Du$. This is in turn achieved via a fractional De Giorgi type iteration relying on the use of nonlinear potentials of the type pioneered by Havin \& Maz'ya \cite{HM} and that are intensively used in Nonlinear Potential Theory; see \rif{defi-P}. Local Lipschitz estimates are, as well-known, the focal point of regularity in nonuniformly elliptic problems \cite{LUcpam, simon0, simon1, ivanov1, ivanov2, ma1,ma2, masurvey}. After this, we can modify more standard perturbation methods to conclude with local H\"older continuity. 

\item The optimal form of \rif{fraziona}, necessary for the sharp bound \rif{pq}. The proof of \rif{fraziona} is based on a Littlewood-Paley type argument resembling the dyadic decomposition of Besov functions \cite{triebel}, but implemented at a nonlinear level; we refer to \cite{piovra, logdoppio, cimemin, KM, Mp} for more details in this direction. The key point is that in our setting atoms are replaced by solutions to certain nonlinear problems without coefficients. At this stage, a crucial role is played by certain improved forms of a priori estimates for such problems recently obtained by Bella \& Sch\"affner \cite{BS0, BS01, BS1}. These allow to deduce the optimal exponent $\textnormal{\texttt{b}}$ in \rif{fraziona}, eventually leading to the sharp bound \rif{pq}. 
\end{itemize}

\subsection{Locally bounded minima}
As mentioned in the previous section, when dealing with bounded minimizers bounds on $q/p$, guaranteeing regularity are independent of $n$. 
\begin{theorem}[Non-dimensional gap bound]\label{main3} Let $u\in W^{1,p}_{\loc}(\Omega)$ be an $L^{\infty}_{\loc}(\Omega)$-regular minimizer of the LSM-relaxation $\bar{\mathcal{F}}$ of $\mathcal{F}$, under assumptions \eqref{assf} with $p\leq n$ and 
\eqn{pqbb}
$$
 q<p+\alpha\min\left\{1,p/2\right\}\,.
$$
Then $Du\in L^{\qqq}_{\loc}(\Omega;\er^n)$ for all $\qqq\in [1,\infty)$, and
\begin{flalign}\label{m.4}
\nr{Du}_{L^{\qqq}(B_{\rr/2})}^{\qqq}\le  c\left(\frac{\nr{u- (u)_{B_{\rr}}}_{L^{\infty}(B_{\rr})}}{\rr} +1\right)^{\mf{b}_{\qqq}}\left[\bar{\mathcal{F}}(u,B_{\rr})+\rr^n\right]
\end{flalign}
holds for every ball $B_{\rr}\Subset \Omega$, $\rr \leq 1$, with $c\equiv c(n,p,q,\alpha,L)$, $\mf{b}_{\qqq}\equiv \mf{b}_{\qqq}(n,p,q,\alpha, \qqq)$.
If \eqref{vanish} is satisfied, then the  same result holds for $ L^{\infty}_{\loc}(\Omega)$-regular minimizers $u$ of $\mathcal F$ in \trif{ggg2}.
\end{theorem}
When $p\geq 2$ condition \rif{pqbb} is sharp, again by the counterexamples in \cite{sharp, FMM}, as described in Remark \ref{bere}. Note that boundedness of minima typically derives from boundary conditions via maximum principles. Note also that in Theorem \ref{main3} we confine ourselves to the case $p\leq n$, otherwise $u \in L^{\infty}_{\loc}$ is automatic by Morrey's embedding. On the other hand, when $p> n$, condition \rif{pqbb} is more restrictive than \rif{pq}. 

\vspace{2mm}

{\bf Acknowledgements} We would like to thank the referees for their numerous comments and for the attention they gave to the original manuscript. The present version of the paper has benefited from their insights. We also thank Carlo Alberto Antonini for suggestions given on a preliminary form of the manuscript. This paper is supported by the University of Parma via the project ``Nonuniform ellipticity in nonlinear problems". 
\section{Preliminaries}
\subsection{Notation}\label{notazioni} 
We denote by $c$ a general constant such that $c\geq 1$. Different occurrences will be still denoted by $c$. Special occurrences will be denoted by $c_*,  \tilde c$ or likewise. Relevant dependencies on parameters will be as usual emphasized by putting them in parentheses.  By writing $\texttt{a} \lesssim \texttt{b}$, with $\texttt{a},\texttt{b}$ being two non-negative real numbers, we mean there exists a constant $c$ such that $\texttt{a}\leq c\, \texttt{b}$. In case the constant $c$ depends on parameters as, to say, $\gamma, \sigma$, that we wish to emphasize, we shall denote  $\texttt{a} \lesssim_{\gamma, \sigma} \texttt{b}$. Finally, we write $\texttt{a} \approx_{\gamma, \sigma} \texttt{b}$ when both $\texttt{a} \lesssim_{\gamma, \sigma} \texttt{b}$ and $\texttt{b} \lesssim_{\gamma, \sigma} \texttt{a}$ occur. With $x_0 \in \er^n$, we denote 
\eqn{zerocenter}
$$
B_r(x_0):= \{x \in \er^n  :   |x-x_0|< r\}, \quad \mathcal B_{r} \equiv B_{r}(0_{\er^n}):= \{x \in \er^n  :   |x|< r\}\,.
$$
When no ambiguity arises we omit denoting the center, i.e., we abbreviate $B_r \equiv B_r(x_0)$; this especially happens when various balls in the same context share the same center. We shall often denote by ${\rm B}$ a generic ball in $\er^n$, when neither the center nor the radius will particularly matter. Finally, with $\rm B$ being a given ball with radius $r$ and $\gamma$ being a positive number, we denote by $\gamma \rm B$ the concentric ball with radius $\gamma r$ and by $\rm B/\gamma \equiv (1/\gamma){\rm B}$. Given $x \equiv (x_i)_{1\leq i\leq n}\in \er^n$, we denote $|x|_{\infty}=\max_i \, |x_i|$ and 
\eqn{cubozzo}
$$
Q_r(x_0):= \{x \in \er^n  :   |x-x_0|_{\infty}< r\}
$$
the ball with respect to the usual Chebyshev metric (hypercube centred at $x_0$ and sidelength $2r$). 
We denote by 
\eqn{inner}
$$Q_{\textnormal{inn}}(B_r(x_0)):= \{x \in \er^n  :   |x-x_0|_{\infty}< r/\sqrt{n}\}\subset B_r(x_0)$$ the inner hypercube of the ball $B_r(x_0)$, that is, the largest hypercube concentric to $B_r(x_0)$, with sides parallel to the coordinate axes, which is contained in $B_r(x_0)$. We shall also need to use a standard lattice of cubes of mesh $r>0$, and for this we shall canonically use
\eqn{lattice}
$$
\mathcal L_r := \{Q_{r}(y) \subset \er^n \, \colon \, y\in (2r) \mathbb Z^n\}\,.
$$
These are mutually disjoint cubes whose closures cover $\er^n$. 
In denoting several function spaces like $L^{\mmm}(\Omega), W^{1,m}(\Omega)$, we shall denote the vector valued version by $L^{\mmm}(\Omega;\er^k), W^{1,\mmm}(\Omega;\er^k)$ in the case the maps considered take values in $\er^k$, $k\in \mathbb{N}$; when confusion will not arise, we shall abbreviate $L^{\mmm}(\Omega;\er^k)\equiv L^{\mmm}(\Omega)$ and the like. With $\mathcal  A \subset \er^{n}$ being a measurable subset such that  $0<|\mathcal A|<\infty$, and $g \colon \mathcal  A \to \er^{k}$, $k\geq 1$, being a measurable map, we denote $$(g)_{\mathcal  A}:=\frac{1}{\snr{\mathcal  A}}\int_{\mathcal  A}g(x)\dx$$ its integral average. With $\alpha \in (0,1]$ we denote 
$$
[g]_{0,\alpha;\mathcal  A}:= \sup_{x,y\in \mathcal  A; x\not=y}\frac{\snr{g(x)-g(y)}}{\snr{x-y}^\alpha}
$$
the usual $\alpha$-H\"older seminorm of $g$ in $\mathcal  A$. Finally, given a number $v \in \er$, we denote $
v_{+}:=\max\{v,0\}$.  
\subsection{Basic tools} We shall use the auxiliary vector field $V_{\mu}\colon \er^{n} \to  \er^{n}$ defined by
\eqn{vpvqm}
$$
V_{\mu}(z):= (|z|^{2}+\mu^{2})^{(p-2)/4}z, \qquad \mu\ge 0\,.
$$
Note that 
\eqn{elmy0}
$$\snr{V_{\mu}(z)}^2\leq H_{\mu}(z)^{p/2}\lesssim_{p} \snr{V_{\mu}(z)}^2 + \mu^p$$
holds for every $z \in \er^n$. A couple of related inequalities  is
\eqn{Vm}
$$
\begin{cases}
\, \snr{V_{\mu}(z_{1})-V_{\mu}(z_{2})}\approx_{p} (\snr{z_{1}}^{2}+\snr{z_{2}}^{2}+\mu^{2})^{(p-2)/4}\snr{z_{1}-z_{2}} \vspace{2mm}\\
\, \snr{z_{1}-z_{2}}\lesssim_{p} \snr{V_{\mu}(z_{1})-V_{\mu}(z_{2})}^{2/p} +  \mathds{1}_{\{p<2\}}\snr{V_{\mu}(z_{1})-V_{\mu}(z_{2})}(\snr{z_1}^{2}+\mu^{2})^{(2-p)/4}\,,
\end{cases}
$$
valid for all $z_{1},z_{2}\in \mathbb{R}^{n}$, with $\mathds{1}_{\{p<2\}}=1$ if $p<2$ and zero otherwise; see \cite[Lemma 2.1]{ha} and \cite[Lemma 2]{km1}, respectively. Another related couple of estimates we shall need is
\eqn{mon}
$$
\begin{cases}
\displaystyle
\, \snr{t_{1}-t_{2}}^{\mathfrak{p}} \lesssim_{\mathfrak{p}} \snr{\snr{t_{1}}^{\mathfrak{p}-1}t_{1}-\, \snr{t_{2}}^{\mathfrak{p}-1}t_{2}},\qquad t_{1},t_{2}\in \mathbb{R}, \ \mathfrak{p}\geq 1\vspace{2mm}\\
\displaystyle
 (\snr{z_{1}}^{2}+\snr{z_{2}}^{2}+\mu^{2})^{\mf{t}/2} \approx_{\mf{t}} \int_{0}^{1}(\snr{z_{1}+\ttl(z_{2}-z_{1})}^{2}+\mu^{2})^{\mf{t}/2}\d\ttl,\quad \quad \mf{t}>-1
\end{cases}
$$
for which we refer to \cite[Lemma A.3]{bls} and \cite{ha}, respectively. Finally, with $\mu$ as in \rif{vpvqm} and $z\in \mathbb{R}^{n}$, we further set
\eqn{defiE}
$$E_{\mu}(z):=\frac 1p (H_{\mu}(z)^{p/2}-\mu^{p})=
\frac 1p [(|z|^2+\mu^2)^{p/2}-\mu^{p}]=\int_0^{|z|}(\ttl^2+\mu^2)^{(p-2)/2}\ttl \d \ttl$$
which is an auxiliary function that plays an important role in the formulation of a priori estimates for nonlinear problems \cite{BM}. As in \cite[(3.15)]{piovra}, and actually as an easy consequence of \rif{Vm}$_1$, of the definition in \rif{defiE} and of Mean Value Theorem, we have that 
\eqn{diffH}
$$
\snr{E_{\mu}(z_2)-E_{\mu}(z_1)} \lesssim_p
(\snr{z_1}^{2}+\snr{z_2}^{2}+\mu^2)^{p/4}\snr{V_{\mu}(z_1)-V_{\mu}(z_2)}
$$
holds for every $z_1, z_2 \in \er^n$.
\begin{lemma}\label{perparti}
Let $\rr, \mathcal{h}_{0}>0$ 
and $h\in \mathbb{R}^{n}$ be such that $\snr{h}\in (0,\mathcal{h}_{0}/4)$, let $B_{\rr}\equiv B_{\rr}(x_{0})\subset \mathbb{R}^{n}$ be a ball, $\mathcal V\in L^{\infty}(B_{\rr+\mathcal{h}_{0}}(x_{0});\mathbb{R}^{n})$, $\mathcal W\in W^{1,\infty}_{0}(B_{\rr}(x_{0});\er^n)$, and $\mathcal K\colon B_{\rr+\mathcal{h}_{0}}(x_{0})\times \er^n \mapsto \er^n$ a continuous vector field which is bounded on $B_{\rr+\mathcal{h}_{0}}(x_{0})\times A$ for every bounded set $A\subset \er^n$. Then
\eqn{af}
$$
\int_{B_{\rr}} \tau_{h}[\mathcal K (\cdot ,\mathcal V)]\cdot \mathcal W \dx=-\snr{h}\int_{B_{\rr}} \int_{0}^{1}\mathcal K (x+\ttl h,\mathcal V(x+\ttl h))\d\ttl\cdot \partial_{h/\snr{h}}\mathcal W\dx\,.
$$
\end{lemma}
\begin{proof} We take a standard, non-negative mollifier $\phi\in C^{\infty}_{0}(\BB)$, $ \nr{\phi}_{L^{1}(\mathbb{R}^n)}=1$, and define, for  $0 < \delta < \mathcal{h}_{0}/4$ and $x \in B_{\rr}$
$$
\mathcal K_{\delta}(x,z):=\int_{\mathcal B_1}\int_{\mathcal B_1}\mathcal K (x+\delta \lambda,z+\delta y)\phi(\lambda)\phi(y)\d\lambda\dy,\quad \ \mathcal V_{\delta}(x):=\int_{\mathcal B_1}\mathcal V(x+\delta \lambda)\phi(\lambda)\d\lambda\,.
$$
We then have
\begin{flalign*}
& \int_{B_{\rr}}  \tau_{h}[\mathcal K_{\delta}(\cdot,\mathcal V_{\delta})]\cdot \mathcal W\dx\\
&\qquad =\int_{B_{\rr}}\int_{0}^{1}\frac{\d}{\d\ttl}\mathcal K_{\delta}(x+\ttl h,\mathcal V_{\delta}(x+\ttl h))\d\ttl\cdot \mathcal W\dx\nonumber \\
&\qquad =\sum_{i,m=1}^{n}\int_{B_{\rr}}\int_{0}^{1}\partial_{x_{i}}\mathcal K_{\delta}^{m}(x+\ttl h,\mathcal V_{\delta}(x+\ttl h))\mathcal W^{m}h_{i}\d\ttl\dx\nonumber \\
&\qquad\qquad +\sum_{i,j,m=1}^{n}\int_{B_{\rr}}\int_{0}^{1}\partial_{z_{j}}\mathcal K^{m}_{\delta}(x+\ttl h,\mathcal V_{\delta}(x+\ttl h))\partial_{x_{i}}\mathcal V_{\delta}^{j}(x+\ttl h)\mathcal W^{m}h_{i}\d\ttl\dx\nonumber \\
&\qquad=\sum_{i,m=1}^{n}\int_{B_{\rr}}\int_{0}^{1}D_{i}[\mathcal K^{m}_{\delta}(x+\ttl h,\mathcal V_{\delta}(x+\ttl h))]\mathcal W^{m}h_{i}\d\ttl\dx\nonumber \\
&\qquad=-\sum_{i,m=1}^{n}\int_{B_{\rr}}\int_{0}^{1}\mathcal K^{m}_{\delta}(x+\ttl h, \mathcal V_{\delta}(x+\ttl h))D_{i}\mathcal W^{m}h_{i}\d\ttl\dx\,.
\end{flalign*}
It therefore follows that
\eqn{af.0}
$$
\int_{B_{\rr}}  \tau_{h}[\mathcal K_{\delta}(\cdot,\mathcal V_{\delta})]\cdot \mathcal W\dx=
-\snr{h}\int_{B_{\rr}}\int_{0}^{1} \mathcal K_{\delta}(x+\ttl h,\mathcal V_{\delta}(x+\ttl h))\d\ttl\cdot \partial_{h/\snr{h}}\mathcal  W\d\ttl\dx\,.
$$
Since for $\ttl \in (0,1]$ we have $\nr{\mathcal V_{\delta}}_{L^{\infty}(B_{\rr})}+
\nr{\mathcal V_{\delta}(\cdot+\ttl h)}_{L^{\infty}(B_{\rr})}
\le 2\nr{\mathcal V}_{L^{\infty}(B_{\rr+\mathcal{h}_{0}}(x_{0}))}=:V_{\infty}$, it follows
\begin{flalign*}
\sup_{x\in B_{\rr}}\snr{\tau_{h}[\mathcal K_{\delta}(\cdot,\mathcal V_{\delta})]}&+\sup_{x\in B_{\rr}}\int_{0}^{1}\snr{\mathcal K_{\delta}(x+\ttl h,\mathcal V_{\delta}(x+\ttl h))}\d\ttl\nonumber \\
\le& \sup_{(x,z)\in B_{\rr+\mathcal{h}_{0}}(x_{0})\times \{z\in \mathbb{R}^{n}\colon \snr{z}\le V_{\infty}\}}3\snr{\mathcal K (x,z)}\le: 3K_{\infty}<\infty,
\end{flalign*}
and $K_{\infty}$ does not depend on $\delta$. Letting $\delta \to 0$ in \eqref{af.0} we obtain \eqref{af} by dominated convergence.
\end{proof} 
Finally, a classical iteration result \cite[Lemma 6.1]{giu}.
\begin{lemma}\label{iterlem}
Let $\hhh\colon [r_1,r_2]\to \mathbb{R}$ be a non-negative and bounded function, and let $a,b, \gamma$ be non-negative numbers. Assume that the inequality 
$
\hhh(\tau_1)\le  \hhh(\tau_2)/2+a(\tau_2-\tau_1)^{-\gamma}+b
$
holds whenever $r_1\le \tau_1<\tau_2\le r_2$. Then 
$
\hhh(r_1)\lesssim_{\gamma} a(r_2-r_1)^{-\gamma}+b.
$
\end{lemma}
\subsection{Additional H\"older continuity}\label{addhol}  In Theorems \ref{main}-\ref{main3} we can always assume that 
\eqn{assf2}
$$
\snr{F(x_1,z)-F(x_2,z)} \leq 
 L \snr{x_{1}-x_{2}}^{\alpha}[H_{\mu}(z)^{q/2}+H_{\mu}(z)^{p/2}]
$$
holds whenever $x_1, x_2\in \Omega$ and $z\in \er^n$. In particular, $F$ is locally H\"older continuous in both variables. For this, next to \rif{ggg2}, define  $\mathcal{F}_1(w,\Omega):=\int_{\Omega}\mathbb F(x,Dw) \dx$, where $\mathbb F(x,z):= F(x,z)-F(x,0_{\er^n})+L\mu^q+L\mu^p$. 
Note that $\mathbb F(\cdot)$ satisfies \rif{assf} (with $2L$ instead of $L$, which is inessential; this actually happens only in \rif{assf}$_2$) and in addition \rif{assf2}. Indeed, by writing $\mathbb F(x,z)=\int_0^1 \partial_zF(x,\ttl z)\d \ttl \cdot z+L\mu^q+L\mu^p$, and using \rif{assf}$_5$, it follows
\begin{flalign*}
\snr{\mathbb F(x_1,z)-\mathbb F(x_2,z)} & \leq \int_0^1 \snr{\partial_z F(x_1, \ttl z)-\partial_z F(x_2, \ttl z)}\d \ttl \snr{z}\\
& \leq  L \snr{x_{1}-x_{2}}^{\alpha}[H_{\mu}(z)^{q/2}+H_{\mu}(z)^{p/2}]\,.
\end{flalign*}
Now, note that any minimizer of $\mathcal{F}$ is a minimizer of $\mathcal{F}_1$ and vice versa.  The same happens for minima of $\bar{\mathcal{F}}$ with respect to $\bar{\mathcal{F}}_1$ since 
$\bar{\mathcal{F}}_1(\cdot, {\rm B})= \bar{\mathcal{F}}(\cdot, {\rm B})- \int_{{\rm B}}[F(x, 0_{\er^n})-L\mu^q-L\mu^p]\dx$ for every ball ${\rm B}\Subset \Omega$. Therefore, up to passing to $\mathcal{F}_1$, we can always assume that \rif{assf2} holds whenever \rif{assf} are verified. 

\section{Minimal fractional and interpolation tools} For $w \colon \Omega \to \mathbb{R}^{k}$, $k\ge 1$, $\texttt{t}>0$ and $h \in \mathbb{R}^n$, we set $\Omega_{\texttt{t}\snr{h}}:=\left\{x\in \Omega\colon \dist(x,\partial \Omega)>\texttt{t}\snr{h}\right\}$, and introduce the finite difference operators $\tau_{h}\colon L^{1}(\Omega;\mathbb{R}^{k})\to L^{1}(\Omega_{|h|};\mathbb{R}^{k})$, $\tau^{2}_{h}\colon L^{1}(\Omega;\mathbb{R}^{k})\to L^{1}(\Omega_{2|h|};\mathbb{R}^{k})$, pointwise defined as
\eqn{basilari}
$$
\begin{cases}
\, \tau_{h}w(x):=w(x+h)-w(x)\vspace{1.5mm}\\
\, \tau_{h}^{2}w(x):=\tau_{h}(\tau_{h}w)(x)\equiv w(x+2h)-2w(x+h)+w(x)\,.
 \end{cases}
$$
Given another map $v\colon \Omega\to \mathbb{R}^{k}$, the discrete Leibniz rule reads as
\eqn{prod}
$$
\tau_{h}(vw)(x)=w(x+h)\tau_{h}v(x)+v(x)\tau_{h}w(x)\,.
$$
Moreover, if $B_{\rr}\Subset B_{r}$ are concentric balls and $w\in W^{1,m}(B_r;\mathbb{R}^{k})$, $\mmm\ge 1$ and $\snr{h}\leq r-\rr$, then
\eqn{gh}
$$
\nr{\tau_{h}w}_{L^{\mmm}(B_{\rr})}\le \snr{h}\nr{Dw}_{L^{\mmm}(B_{r})}\,.
$$
\begin{definition}\label{fra1def}
Let  $\mmm \in [1, \infty)$, $s \in (0,1)$.
\begin{itemize}
\item With $\Omega \subset \er^n$ being an open subset, $w\colon \Omega\to \mathbb{R}^{k}$ belongs to the Sobolev-Slobodeckij space $W^{\tta,\mmm}(\Omega;\er^k )$ iff
\begin{flalign}
\notag
\| w \|_{W^{\tta,\mmm}(\Omega)} & := \|w\|_{L^{\mmm}(\Omega)}+ \left(\int_{\Omega} \int_{\Omega}  
\frac{|w(x)
- w(y) |^{\mmm}}{|x-y|^{n+\tta \mmm}} \dx \dy \right)^{1/\mmm}\\
&=: \|w\|_{L^{\mmm}(\Omega)} + [w]_{\tta,\mmm;\Omega} < \infty\,.\label{gaglia}
\end{flalign}
\item $w\colon \er^n\to \mathbb{R}^{k}$ belongs to the Nikol'skii space $N^{s,\mmm}(\er^n;\er^k)$ iff 
$$\| w \|_{N^{\tta,\mmm}(\er^n;\er^k )} :=\|w\|_{L^{\mmm}(\er^n)} + \left(\sup_{|h|\not=0}\, \int_{\er^n} \left|\frac{\tau_{h}w}{\snr{h}^{\tta}}\right|^{\mmm}
 \dx  \right)^{1/\mmm}<\infty\,.$$
\end{itemize}
\end{definition}
The embedding inequalities 
\eqn{33}
$$
\|w\|_{W^{\beta, \mmm}(\er^n)}\lesssim_{n,s,\beta,m}\|w\|_{N^{s, \mmm}(\er^n)} 
\lesssim_{n,s,\mmm} \|w\|_{W^{s, \mmm}(\er^n)}
$$
hold whenever $0< \beta < s < 1$ and $m\geq 1$; see for instance \cite[Propositions 2.6, 2.7]{bl}. A local, quantified version of the left-hand side inequality in \eqref{33} is in \begin{lemma}\, \hspace{-2.5mm}\label{l4}  
Let $B_{\varrho} \Subset B_{r}\subset \er^n$ be concentric balls with $r\leq 1$, $w\in L^{\mmm}(B_{r};\mathbb{R}^{k})$, $m\geq 1$ and assume that for $\tta \in (0,1]$ there holds
\eqn{cru1}
$$
\nr{\tau_{h}w}_{L^{\mmm}(B_{\rr})}\le  \snr{h}^{\tta }
$$
for every $h\in \mathbb{R}^{n}$ with $0<\snr{h}\le (r-\rr)/K$, where $K \geq 1$.
Then
\eqn{cru2}
$$
\nr{w}_{W^{\beta,\mmm}(B_{\rr})}\lesssim_{n }\frac{1}{(\tta -\beta)^{1/\mmm}}
\left(\frac{r-\rr}{K}\right)^{\tta -\beta} +\left(\frac{K}{r-\rr}\right)^{n/\mmm+\beta} \nr{w}_{L^{\mmm}(B_{\rr})}
$$
holds for all $\beta\in(0,\tta)$. 
\end{lemma}
See \cite[Lemma 1]{amp} for a proof.  
For the next result we refer to \cite[Proposition 2.4]{bl} and \cite[Lemma 2.4]{gl}; see also \cite[Theorem 1]{d} for related inequalities. It basically says that gradient integrability can be retrieved controlling double difference quotients with exponents larger than one; these are in turn relevant in the definition of  certain Besov spaces \cite[Section 2.2]{bl}. 
\begin{lemma}\label{bllema} Let $w\in L^{\mmm}(B_{\rr+7\mathcal{h}_{0}})$, $\mmm\in (1,\infty)$, where $B_{\rr+7\mathcal{h}_{0}}\subset \er^n$ is a ball and $\rr, \mathcal{h}_{0}\in (0,1]$. Then
\eqn{immersione2}
$$
\nr{Dw}_{L^{\mmm}(B_{\rr})}\lesssim_{n,m}  \frac{1}{(\tta-1)(2-\tta)} \left[\sup_{0<\snr{h}<\mathcal{h}_{0}}\left\|\frac{\tau_{h}^{2}w}{\snr{h}^{s}}\right\|_{L^{\mmm}(B_{\rr+5\mathcal{h}_{0}})}+\frac{\nr{w}_{L^{\mmm}(B_{\rr+6\mathcal{h}_{0}})}}{\mathcal{h}_{0}^{\tta}} \right]
$$
holds provided $\tta\in (1,2)$. 
\end{lemma}
We proceed with two lemmas of interpolative nature; the former is  \cite[Corollary 3.2,(c)]{bm}, while the latter is instead a localization in Nikol'skii spaces for which we provide a proof. 
\begin{lemma}\, \hspace{-2.5mm}  With $\tta \in (0,1]$, $\mmm\in (1,\infty)$, and $w\in W^{\tta,\mmm}(\mathbb{R}^{n};\mathbb{R}^{k})\cap L^{\infty}(\mathbb{R}^{n};\mathbb{R}^{k})$, 
\eqn{bm.1}
$$
\nr{w}_{W^{\sigma\tta,\mmm/\sigma}(\mathbb{R}^{n})}\lesssim_{n,\tta,\mmm,\sigma}\nr{w}_{L^{\infty}(\mathbb{R}^{n})}^{1-\sigma}\nr{w}_{W^{\tta,\mmm}(\mathbb{R}^{n})}^{\sigma}
$$
holds for every $\sigma \in (0,1)$. 
\end{lemma}
\begin{lemma}\label{ls} 
Let $B_{\theta\rr} \Subset B_{\rr}\subset \mathbb{R}^{n}$ be concentric balls with $\theta\in (0,1)$ and let $w\in W^{1,m}(B_{\rr})\cap L^{\infty}(B_{\rr})$. Then 
\begin{flalign}\label{fs.0}
\theta^{ns/m}\sup_{0<\snr{h}<(1-\theta)\rr/4}\left(\mint_{B_{\theta\rr}}\left|\frac{\tau_{h}w}{\snr{h}^{s}}\right|^{m/s}\dx\right)^{s/m}&\lesssim_{n,s,m} \frac{\nr{w}_{L^{\infty}(B_{\rr})}^{1-s}}{(1-\theta)^{s}\rr^{s}}\left(\mint_{B_{\rr} }\snr{w}^{m}\dx\right)^{s/m}\nonumber \\
&\qquad \ \ +\nr{w}_{L^{\infty}(B_{\rr})}^{1-s}\left(\mint_{B_{\rr} }\snr{Dw}^{m}\dx\right)^{s/m}
\end{flalign}
holds whenever $s\in (0,1)$, $\mmm\in (1,\infty)$.
\end{lemma}
\begin{proof}
We rescale $w$ on $B_{\rr}\equiv B_{\rr}(x_{0})$ defining $w_{\rr}(x):=w(x_{0}+\rr x)$; this  belongs to $W^{1,m}(\mathcal B_1)\cap L^{\infty}(\mathcal B_1)$, recall the notation in \rif{zerocenter}. Then we choose a cut-off function $\eta\in C^{1}_{0}(\mathcal B_1)$ such that $\mathds{1}_{\mathcal B_{(2\theta+1)/3}}\le \eta\le \mathds{1}_{\mathcal B_{(\theta+1)/2}}$ and $\snr{D\eta}\lesssim 1/(1-\theta)$. The function $\ti{w}:=\eta w_{\rr}$ obviously belongs to $W^{1,m}(\mathbb{R}^{n})\cap L^{\infty}(\mathbb{R}^{n})$ with
\begin{flalign}\label{fs.1}
\nr{\ti{w}}_{L^{\infty}(\mathbb{R}^{n})}\le \nr{w_{\rr}}_{L^{\infty}(\mathcal B_1)}\qquad \mbox{and}\qquad \nr{D\ti{w}}_{L^{m}(\mathbb{R}^{n})}\lesssim \frac{\nr{w_{\rr}}_{L^{m}(\mathcal B_1)}}{(1-\theta)}+\nr{Dw_{\rr}}_{L^{m}(\mathcal B_1)}\,.
\end{flalign}
Note that $\snr{h}<(1-\theta)/4$ implies $\eta(x)=\eta(x+h)=1$ for $x \in \mathcal B_{\theta}$, so that
\begin{eqnarray*}
\sup_{0<\snr{h}<(1-\theta)/4}\left(\int_{\mathcal B_{\theta}}\left|\frac{\tau_{h}w_{\rr}}{\snr{h}^{s}}\right|^{m/s}\dx\right)^{s/m}&=&\sup_{0<\snr{h}<(1-\theta)/4}\left(\int_{\mathcal B_{\theta}}\left|\frac{\tau_{h}\ti{w}}{\snr{h}^{s}}\right|^{m/s}\dx\right)^{s/m}\\
&\le &\nr{\ti{w}}_{N^{s,m/s}(\mathbb{R}^{n})}\nonumber \\
&\stackrel{\eqref{33}}{\le}&c\nr{\ti{w}}_{W^{s,m/s}(\mathbb{R}^{n})}\\
&\stackrel{\eqref{bm.1}}{\le}& c\nr{\ti{w}}^{1-s}_{L^{\infty}(\mathbb{R}^{n})}\nr{\ti{w}}^{s}_{W^{1,m}(\mathbb{R}^{n})}\nonumber \\
&\stackrel{\eqref{fs.1}}{\le}&\frac{c\nr{w_{\rr}}_{L^{\infty}(\mathcal B_1)}^{1-s}\nr{w_{\rr}}_{L^{m}(\mathcal B_1)}^{s}}{(1-\theta)^{s}} \notag \\
&& \  +c\nr{w_{\rr}}_{L^{\infty}(\mathcal B_1)}^{1-s}\nr{Dw_{\rr}}_{L^{\mmm}(\mathcal B_1)}^{s},
\end{eqnarray*}
with $c\equiv c(n,s,m)$. Estimate \rif{fs.0} follows by scaling back to $B_{\rr}$ the inequality on the above display. 
\end{proof}
\section{Nonlinear potentials} Here we recall a few basic tools and results from \cite{piovra}. These essentially deal with a class of nonlinear potentials of the type originally introduced by Havin \& Maz'ya \cite{HM}; see also \cite{BM,camel,ciccio,logdoppio} for applications in the setting of nonuniformly elliptic problems. With $\sigma>0$ and $\vartheta\geq 0$ being fixed parameters, and $f\in L^{1}(B_{r}(x_{0}))$, where $B_{r}(x_0)\subset \er^n$ is a ball, the nonlinear Havin-Mazya-Wolff type potential ${\bf P}_{\sigma}^{\theta}(f;\cdot)$ is defined by 
\eqn{defi-P} 
$$
{\bf P}_{\sigma}^{\vartheta}(f;x_0,r) := \int_0^r \varrho^{\sigma} \left(  \mint_{B_{\varrho}(x_0)} \snr{f} \dx \right)^{\vartheta} \frac{\d\varrho}{\varrho} \,.
$$
Such potentials generalize classical Riesz and Wolff potentials for suitable choices of $\vartheta, \sigma$ \cite{piovra}. For instance, by choosing $\vartheta=\sigma=1$ we obtain the classical (truncated) Riesz potential ${\bf I}_{1}(f)$, while taking $\vartheta =1/(p-1)$ and $\sigma= p\alpha/(p-1)$ we obtain the classical Wolff potential ${\bf W}_{\alpha,p}^f$; see for instance \cite{cimemin} the precise definitions. The next result is a special case of \cite[Lemma 4.1]{piovra}. 
\begin{lemma}\label{crit} 
Let $B_{\tau}\Subset B_{\tau+r}\subset \mathbb{R}^{n}$ be two concentric balls with $\tau, r\leq 1$, $f\in L^{1}(B_{\tau+r})$ and let $\sigma,\vartheta>0$ be such that $ n\vartheta>\sigma$. Then
\eqn{stimazza}
$$
\nr{{\bf P}_{\sigma}^{\vartheta}(f;\cdot,r)}_{L^{\infty}(B_{\tau})} \lesssim_{n,\vartheta,\sigma,m} c\|f\|_{L^{m}(B_{\tau+r})}^{\vartheta} $$
holds whenever $m > n\vartheta/\sigma>1$. 
\end{lemma}
The next lemma is a version of \cite[Lemma 4.2]{piovra}, see also \cite[Lemma 3.1]{BM}, and takes information from various results from Nonlinear Potential Theory starting by the seminal paper \cite{kilp}. It is essentially a nonlinear potential theoretic version of De Giorgi's iteration in that a kind of reverse H\"older's inequality with a remainder term \rif{revva} is shown to imply a pointwise bound \rif{siapplica}, while the presence of the remainder term is encoded by a potential of the type in \rif{defi-P}. The key point here is the tracking of the explicit dependence on the constants $M_0, M_*$ from \rif{revva} to \rif{siapplica}. 
\begin{lemma}\label{revlem}
Let $B_{r_{0}}(x_{0})\subset \mathbb{R}^{n}$ be a ball, $n\ge 2$, $f \in L^1(B_{2r_0}(x_{0}))$, and constants $\chi >1$, $\sigma, \vartheta,\ti{c},M_{0}>0$ and $\kappa_0, M_{*}\geq 0$. Assume that $v \in L^2(B_{r_0}(x_0))$ is such that for all $\kk\ge \kk_{0}$, and for every concentric ball $B_{\rr}(x_{0})\subseteq B_{r_{0}}(x_{0})$, the inequality
\begin{flalign}
\left(\mint_{B_{\rr/2}(x_{0})}(v-\kk)_{+}^{2\chi}  \dx\right)^{\frac1{2\chi}}  &\le \ti{c}M_{0}\left(\mint_{B_{\rr}(x_{0})}(v-\kk)_{+}^{2}  \dx\right)^{\frac{1}{2}}+\ti{c} M_{*}\rr^{\sigma}\left(\mint_{B_{\rr}(x_{0})}\snr{f}  \dx\right)^{\vartheta}
 \label{revva}
\end{flalign}
holds. If $x_{0}$ is a Lebesgue point of $v$ in the sense that 
$$
v(x_0) = \lim_{r\to 0} (v)_{B_{r}(x_0)}\,,
$$ then
\eqn{siapplica}
$$
 v(x_{0})  \le\kk_{0}+cM_{0}^{\frac{\chi}{\chi-1}}\left(\mint_{B_{r_{0}}(x_{0})}(v-\kk_{0})_{+}^{2}  \dx\right)^{1/2}
+cM_{0}^{\frac{1}{\chi-1}} M_{*}\mathbf{P}^{\vartheta}_{\sigma}(f;x_{0},2r_{0})
$$
holds with $c\equiv c(n,\chi,\sigma,\vartheta,\ti{c})$.  
\end{lemma}
\section{Conditional nonlinear Schauder}
In this section we enucleate in a separate statement, and extend, some purely perturbation results obtained in \cite[Section 10]{piovra}. We consider a general elliptic equation of the type
\eqn{equazioni}
$$
-\diver\, A (x,Du)=0\quad \mbox{in}\ \Omega
$$ 
where $\Omega$ is a bounded open subset of $\er^n$ and $A \colon \Omega \times \er^n \to \er^n$ satisfies 
\begin{flalign}\label{assAA}
\begin{cases}
\, \tilde{z}\mapsto A(\cdot,\tilde{z})\in C^{1}_{\loc}(\mathbb{R}^{n})\vspace{2mm}\\
 \, |A(x,  z)|+ H_{\mu}(z)^{1/2}\snr{\partial_z A(x,  z)} \le \ti{L}H_{\mu}(z)^{(q-1)/2}+\ti{L}H_{\mu}(z)^{(p-1)/2}\vspace{2mm}\\
\,  \ti{L}^{-1}H_{\mu}(z)^{(p-2)/2}\snr{\xi}^{2}\le \partial_{z}A(x,z)\xi\cdot \xi\vspace{2mm} \\
\,  \snr{A(x_{1},z)- A(x_{2}, z)}  \le \ti{L}  \snr{x_{1}-x_{2}}^{\alpha}[H_{\mu}(z)^{(q-1)/2}+H_{\mu}(z)^{(p-1)/2}]
\end{cases}
\end{flalign} 
whenever $x,x_1, x_2 \in \Omega$ and $z, \xi\in \er^n $, where $\ti{L}\geq 1$, $0< \mu \leq 2$. The main point of the next Proposition \ref{piovrapert} is that in the nondegenerate case $\mu>0$, if we know that $Du$ is H\"older continuous with some exponent, then we can improve such exponent up to the optimal one dictated by coefficients, thereby recovering in the nonlinear case the full statement of the classical linear Schauder theory. The case $p=q=2$, which is automatically non-degenerate, is a by now classical result of Giaquinta \& Giusti \cite{gg3}; in fact, the approach given here and in \cite{piovra} also gives a fast track to the interior estimates of \cite{gg3}. Assumption $\mu>0$, which is obviously linked to the continuity of $\partial_{z}A(\cdot)$, is necessary in view of the counterexamples in \cite{lewis, krol} that already hold in the uniformly elliptic, but degenerate case. 
\begin{proposition}\label{piovrapert} Let $u \in W^{1,\infty}(\Omega)$ be a weak solution to \eqref{equazioni}, under  assumptions \eqref{assAA} with $\alpha<1$ and 
\eqn{beep0}
$$
\frac{q}{p} < 1 +\frac 1n\,.
$$
Moreover, assume that $u \in C^{1, \beta}(\Omega)$ for some $\beta \in (0,\alpha)$ and that for every $H>0$ there exists a modulus of continuity $ \omega_{H}(\cdot)$ such that
\eqn{modulus}
$$
\snr{\partial_{z}A(x,z_1)-\partial_{z}A(x,z_2)} \leq \omega_{H}(\snr{z_1-z_2})\,,
$$
holds whenever $x\in \Omega$, $z_1, z_2\in \er^n$ with $|z_1|, |z_2|\leq H$. Then $u \in C^{1, \alpha}_{\loc}(\Omega)$. Moreover, with 
\eqn{grandeM}
$$
M:= \|Du\|_{L^\infty(\Omega)}+[Du]_{0, \beta;\Omega}+1\,,
$$
for every open subset $\Omega_0\Subset \Omega$ there exists a constant $c\equiv c(M, \dist(\Omega_0, \Omega))$, depending also on $n,p,q, \alpha, \beta,\ti{L},\mu$ and $\omega_{M}(\cdot)$, such that 
$
[Du]_{0, \alpha; \Omega_0} \leq c.$ 
\end{proposition}
\begin{proof} In the case $p\geq 2$ the proof is in \cite[Section 10.1]{piovra}. There we treated the variational case, i.e., when \rif{equazioni} is the Euler-Lagrange equation of a functional of the type \rif{ggg2}, but in fact the proof is written in a way that works directly for general equations of the type \rif{equazioni}. Here we describe how to reach the missing case $1<p<2$. For the rest of the proof we shall denote by $c$ a general constant depending at most on $n,p,q,\ti{L}$ and by $c_M$ a constant of the same type, but depending also on $M$. Following the arguments in \cite[Section 10.1]{piovra}, it is sufficient to show how to arrive at a suitable analogue of \cite[(10.2)-(10.3)]{piovra} when $p<2$. Specifically, 
we take a ball $B_{r}\equiv B_{r}(x_{\rm c})\Subset \Omega$ such that $r\leq 1$, and denote 
$A_{r}(z):=A(x_{\rm c},z)+r^{\alpha}H_{\mu}(z)^{(q-2)/2}z.$
We define $v$ as the unique solution to  
$\diver\, A_{r}(Dv) =0$ in $B_{r}$ such that $v-u\in W^{1,q}_0(B_{r})$ and our aim is to prove, for $p<2$, that 
\eqn{campcomp0}
$$
\mint_{B_{r/2}}\snr{Du-Dv}^{2}\dx+\mint_{B_{r}}\snr{V_{\mu}(Du)-V_{\mu}(Dv)}^{2}\dx\leq c_M r^{2\alpha}
$$
holds. This incorporates \cite[(10.2)]{piovra} and \cite[(10.3)]{piovra}, this last one with $r/2$ instead of $r$ (which is inessential). For this, thanks to \rif{beep0} we can use \cite[Lemma 5.3]{piovra}, that gives
\eqn{supb}
$$
\|Dv\|_{L^{\infty}(B_{r/2})} \leq c \left(\mint_{B_r}(\snr{Dv}+1)^p\dx\right)^{ \frac{1+\eps}{(n+1)p-nq}} 
$$
for every $\eps>0$ when $n=2$ and $\eps=0$ when $n\geq 3$. 
We then recall the classical monotonicity inequality
\eqn{beep}
$$
\snr{V_{\mu}(z_{1})-V_{\mu}(z_{2})}^{2} \lesssim_{p, \ti{L}}   (A_r(z_{1})-A_r(z_{2}))\cdot (z_{1}-z_{2})
 $$
 that holds for every choice of $z_1, z_2\in \er^n$. 
This is a standard consequence of \rif{Vm}$_1$, \rif{mon}$_2$ and \rif{assAA}$_3$. To proceed we estimate
\begin{eqnarray*}
&& \notag \int_{B_{r}}\snr{V_{\mu}(Du)-V_{\mu}(Dv)}^{2}\dx\\ && \qquad \stackleq{beep}  c \int_{B_{r}}\left(A_{r}(Du)-A_{r}(Dv)\right) \cdot \left(Du-Dv\right)\dx\nonumber \\
&&\qquad\  = c\int_{B_{r}}\left(A_{r}(Du)-A(x, Du)\right)\cdot \left(Du-Dv\right)\dx\nonumber\\
&&\qquad \ =c\int_{B_{r}}\left(A(x_{\rm c},Du)-A(x, Du)\right)\cdot \left(Du-Dv\right)\dx\nonumber\\
&&\qquad\qquad    + cr^{\alpha}\int_{B_{r}}H_{\mu}(Du)^{(q-2)/2}Du\cdot \left(Du-Dv\right)\dx\nonumber\\
&&  \qquad\stackrel{\eqref{assAA}_4}{\leq } cr^{\alpha}\int_{B_{r}}\left(H_{\mu}(Du)^{(q-1)/2}+H_{\mu}(Du)^{(p-1)/2}\right)\snr{Du-Dv} \dx  \\
& &\qquad\qquad   + cr^{\alpha}\int_{B_{r}}H_{\mu}(Du)^{(q-1)/2}\snr{Du-Dv} \dx\\
&&\qquad \stackleq{grandeM} c (M^{q-1}+M^{p-1})r^{\alpha}\int_{B_{r}}\snr{Du-Dv} \dx \\
&&\qquad \stackrel{\rif{Vm}_2}{\leq} c_Mr^\alpha \int_{B_{r}}\snr{V_{\mu}(Du)-V_{\mu}(Dv)}^{2/p} \dx\\
& &\qquad \qquad+
  c_Mr^\alpha \int_{B_{r}}\snr{V_{\mu}(Du)-V_{\mu}(Dv)}(\snr{Du}^{2}+\mu^{2})^{(2-p)/4}\dx\\
  &&\qquad \stackleq{grandeM}  c_Mr^\alpha \int_{B_{r}}\left(\snr{V_{\mu}(Du)-V_{\mu}(Dv)}^{2/p} +\snr{V_{\mu}(Du)-V_{\mu}(Dv)}\right) \dx\\
    &&\qquad \stackrel{\textnormal{Young}}{\leq}  \frac12 \int_{B_{r}}\snr{V_{\mu}(Du)-V_{\mu}(Dv)}^2 \dx + c_Mr^{n+2\alpha}+c_Mr^{n+\alpha p/(p-1)}\,.
\end{eqnarray*}
Recalling that $p/(p-1)\geq 2$ we obtain
 \eqn{campcomp0pre}
$$
\mint_{B_{r}}\snr{V_{\mu}(Du)-V_{\mu}(Dv)}^{2}\dx\leq c_M r^{2\alpha}
$$ 
from which we have, also using \rif{elmy0}
$$
\mint_{B_{r}}\snr{Dv}^{p}\dx \lesssim \mint_{B_{r}}(\snr{V_{\mu}(Dv)}^{2}+1)\dx\lesssim \mint_{B_{r}}\snr{V_{\mu}(Du)-V_{\mu}(Dv)}^{2}\dx+ M^p\leq c_M \,.
$$
This and \rif{supb} give
\eqn{nellapiovra}
$$\nr{Dv}_{L^{\infty}(B_{r/2})}\leq H $$
for some $H \equiv H(n,p,q,\ti{L}, M)\geq 1$.  Finally, we have 
\begin{eqnarray*}
 \mint_{B_{r/2}}\snr{Du-Dv}^2 \dx
& \stackrel{\rif{Vm}_1}{\leq}& c\mint_{B_{r/2}}(\snr{Du}^2+\snr{Dv}^2+1)^{(2-p)/2}\snr{V_{\mu}(Du)-V_{\mu}(Dv)}^2
\dx
\\
& \stackrel{\rif{grandeM},\rif{nellapiovra}}{\leq} &c_M\mint_{B_{r/2}}\snr{V_{\mu}(Du)-V_{\mu}(Dv)}^2 
\dx\stackleq{campcomp0pre} c_M r^{2\alpha}\,.
\end{eqnarray*}
This, together with \rif{campcomp0pre}, completes the proof of \rif{campcomp0}. As mentioned, the rest of the proof or Proposition \ref{piovrapert} now closely follows that in \cite[Section 10]{piovra}. Summarizing, each gradient component $\mathfrak v:=D_sv$, $s \in \{1, \ldots, n\}$, solves the linear(ized) elliptic equation $\diver\, (\mathbb{A}(x)D\mathfrak{v})=0$, where $[\mathbb{A}(x)]_{ij}:= \partial_{z_j}A_{r}^i(x_{{\rm c}},Dv(x))$. By \rif{assAA}$_{2,3}$ and \rif{nellapiovra}, this satisfies 
$
H^{p-2} \mathds{I}_{\rm d} \lesssim \mathbb{A}(x)$ and $\snr{\mathbb{A}(x)} \lesssim (\mu^{p-2} +H^{q-2}) 
$ on $B_{r/2}$,  
where $H$ is defined in \rif{nellapiovra} and $\mathds{I}_{\rm d}=[\delta_{ij}]_{ij}$. Then $[Dv]_{0, \beta_0; B_{r/4}} $ is finite for some $\beta_0>0$ by De Giorgi-Nash-Moser theory and we can determine a modulus of continuity as in \rif{modulus} for $\partial_{z}A_{r}(\cdot)$ with $H$ being exactly the one in \rif{nellapiovra}. This leads to the continuity of  $\mathbb{A}(\cdot)$ on $B_{r/4}$. The assertion finally follows using a variant of Campanato's perturbation theory. 
 \end{proof}
\begin{remark} {\em In the case $\partial_z A(\cdot)$ is a symmetric $n \times n$ matrix, condition \rif{beep0} can be relaxed in $q/p<1+2/n$; in particular, this happens when considering the Euler-Lagrange equations of functionals. This follows using \cite[(5.25)]{piovra} instead of \cite[(5.24)]{piovra} to get an analogue of \rif{supb}. Anyway, the current version of Proposition \ref{piovrapert} is sufficient for the purposes of this paper as \rif{beep0} is obviously implied by our central assumption \rif{pq}. Further weakening of \rif{beep0} could be obtained using the results of Bella \& Schäffner \cite{BS01, BS0, BS1} described in Section \ref{bellasec} below, again in the case $\partial_z A(\cdot)$ is symmetric.}\end{remark}
\section{The Lebesgue-Serrin-Marcellini relaxation} Here we briefly collect a few basic properties of the LSM-relaxation $\bar{\mathcal F}$ defined in \rif{LSM}. In this section ${\rm B}\subset \er^n$ denotes a fixed ball. 
The following lemma is an easy consequence of the arguments developed in \cite[Section 6]{ma1} and \cite[Section 6]{sharp}. See also \cite{camel, piovra, gme3, ma0}. 
\begin{lemma}\label{bf.0} 
Let $F\colon {\rm B}\times \mathbb{R}^{n}\to \mathbb{R}$ be a Carath\'eodory-regular integrand satisfying \eqref{assf}$_{2}$. Then
\begin{itemize}
\item  If $w\in W^{1,p}({\rm B})$ is such that $\bar{\mathcal{F}}(w,{\rm B})<\infty$, then there exists a sequence $\{w_{j}\}\subset W^{1,q}({\rm B})$ such that
\eqn{l.1.1}
$$
\begin{cases}
\, w_{j}\rightharpoonup w \ \   \mbox{weakly in}  \ W^{1,p}({\rm B})\\
\, w_{j}\to w \ \   \mbox{strongly in}  \ L^{p}({\rm B})\\
\,    \mathcal{F}(w_{j},{\rm B})\to \bar{\mathcal{F}}(w,{\rm B})\,.
\end{cases}
$$
\item  If $z \mapsto F(\cdot, z)$ is convex and $\mathcal{L}_{\mathcal F}(w,{\rm B})=0$, then \eqref{l.1.1} holds with $\mathcal{F}(w,{\rm B})$ instead of $\bar{\mathcal{F}}(w,{\rm B})$. Moreover, if $w \in W^{1,q}(B)$, then $\bar{\mathcal{F}}(w,{\rm B})=\mathcal{F}(w,{\rm B})$. 
\item Assume that also \eqref{assf}$_{1,3}$ hold. If $w_{1},w_{2}\in W^{1,p}({\rm B})$ are such that $\bar{\mathcal{F}}(w_{1},{\rm B})+\bar{\mathcal{F}}(w_{2},{\rm B})\linebreak <\infty$ and $Dw_{1}\not =Dw_{2}$ on a subset of positive measure of the ball ${\rm B}$, then it holds that $$\bar{\mathcal{F}}\left(\frac{w_{1}+w_{2}}{2},{\rm B}\right)<\frac{\bar{\mathcal{F}}(w_{1},{\rm B})}{2}+\frac{\bar{\mathcal{F}}(w_{2},{\rm B})}{2}\,.$$
\end{itemize}
\end{lemma}
In the next lemma we upgrade \rif{l.1.1} to the case of bounded $w$. 
\begin{lemma}\label{bf.1} Let $F\colon {\rm B}\times \mathbb{R}^{n}\to \mathbb{R}$ be a Carath\'eodory-regular integrand satisfying \eqref{assf}$_2$, $\varrho>0$, and $w\in W^{1,p}({\rm B})\cap L^{\infty}( {\rm B})$ be such that $\bar{\mathcal{F}}(w,{\rm B})<\infty$. There exists a sequence $\{w_{j}\}\subset W^{1,q}({\rm B})$ such that \eqref{l.1.1} holds with $\nr{w_{j}}_{L^{\infty}({\rm B})}\le (1+\rr)\nr{w}_{L^{\infty}({\rm B})}+\varrho$ for every $j \in \en$. 
\end{lemma}
\begin{proof} We prove that
\eqn{oppo}
$$
\bar{\mathcal{F}}(w,{\rm B})=\bar{\mathcal{F}}_{\infty}(w,{\rm B}):=\inf\left\{\liminf_{j\to \infty}\mathcal{F}(w_{j},{\rm B})\colon \{w_{j}\}\in \mathcal{C}_{\infty}(w,{\rm B})\right\}
$$ 
where
\begin{flalign*}
&\mathcal{C}_{\infty}(w,{\rm B}):=\Big\{\{w_{j}\}\subset W^{1,q} ({\rm B})\colon w_{j}\rightharpoonup w \ \mbox{weakly in} \ W^{1,p}({\rm B}) 
\\ & \hspace{25mm} \ \mbox{and $\nr{w_{j}}_{L^{\infty}({\rm B})}\le (1+\rr)\nr{w}_{L^{\infty}({\rm B})}+\varrho$ \, for every $j \in \en$} \Big\}\,.
\end{flalign*} 
As $\bar{\mathcal{F}}(w,{\rm B})\le \bar{\mathcal{F}}_{\infty}(w,{\rm B})$ holds by the above definitions, we need the opposite inequality to prove \rif{oppo}. With $\{w_{j}\}\subset W^{1,q}({\rm B})$ being as in \rif{l.1.1}, we define the following truncation of $w_j$:
 $$\bar{w}_{j}:=\min\left\{\max\left\{w_{j},-(1+\rr)\nr{w}_{L^{\infty}({\rm B})}-\varrho\right\},(1+\rr)\nr{w}_{L^{\infty}({\rm B})}+\varrho\right\}\subset W^{1,q}({\rm B})\cap L^{\infty}({\rm B})$$ 
so that $\nr{\bar{w}_{j}}_{L^{\infty}({\rm B})}\le (1+\rr)\nr{w}_{L^{\infty}({\rm B})}+\varrho$.  
Set ${\rm B}_{j,+}:= \{ \snr{w_{j}}>(1+\rr)\nr{w}_{L^{\infty}({\rm B})}+\varrho\} $, ${\rm B}_{j,-}:= {\rm B}\setminus {\rm B}_{j,+}$ and $ {\rm B}_{j,*}:=\{\snr{w_{j}-w}>\rr\nr{w}_{L^{\infty}({\rm B})}+\rr\}$.  By \rif{l.1.1}$_2$ we have $|{\rm B}_{j,*}|\to 0$, and, as ${\rm B}_{j,+}\subset {\rm B}_{j,*}$, it also follows that 
\eqn{convy}
$$\snr{{\rm B}_{j,+}}\to 0 \ \  \mbox{and therefore $\nr{\bar{w}_{j}-w}_{L^p({\rm B})}\to 0$}\,.$$ This implies $\bar{w}_{j}\rightharpoonup w$ in $W^{1,p}({\rm B})$. Indeed, 
given any vector field $W\in L^{p/(p-1)}({\rm B};\mathbb{R}^{n})$ we define $W_{j}:=W\mathds{1}_{\{{\rm B}_{j,-}\}}$, so that $W_{j}\to W$ strongly in $L^{p/(p-1)}({\rm B})$ by \rif{convy}. Using this we find
\begin{flalign}\label{bf.4}
\int_{{\rm B}}  D\bar{w}_{j}\cdot W \dx=\int_{{\rm B}} Dw_{j}\cdot W\mathds{1}_{\{{\rm B}_{j,-}\}}\dx=\int_{{\rm B}} Dw_{j}\cdot W_{j} \dx\to \int_{{\rm B}} Dw\cdot W\dx\,,
\end{flalign}
from which the claimed weak convergence follows together with $\{\bar{w}_{j}\}\in \mathcal{C}_{\infty}(w,{\rm B})$. We then have
\begin{flalign*}
\mathcal{F}(\bar{w}_{j},{\rm B})=\mathcal{F}(w_{j},{\rm B}_{j,-})+\int_{{\rm B}_{j,+}} F(x, 0_{\er^n})\dx \leq \mathcal{F}(w_{j},{\rm B})+2L\mu^{p}\snr{{\rm B}_{j,+}}
\end{flalign*}
that yields, thanks to \rif{l.1.1} and \rif{convy} 
$$
\bar{\mathcal{F}}_{\infty}(w,{\rm B})\le \liminf_{j\to \infty}\mathcal{F}(\bar{w}_{j},{\rm B})\le\lim_{j\to \infty}\mathcal{F}(w_{j},{\rm B})=\bar{\mathcal{F}}(w,{\rm B})
$$
so that \rif{oppo} is proved. The conclusion now follows from \rif{oppo} and observing that by the very definition of $\bar{\mathcal{F}}_{\infty}$, whenever  $\bar{\mathcal{F}}_{\infty}(w, {\rm B})$ is finite there exists a sequence $\{w_{j}\}\subset \mathcal{C}_{\infty}(w,{\rm B})$ such that $\mathcal{F}(w_{j},{\rm B})\to \bar{\mathcal{F}}_{\infty}(w,{\rm B})=\bar{\mathcal{F}} (w,{\rm B})$, analogously to \rif{l.1.1}.
\end{proof}
\begin{remark}\label{alternativa} {\em The literature contains several different notions of relaxed functionals. These are alternative but yet related to the one in \rif{LSM}. For instance one can define
\eqn{LSM2}
$$
\bar{\mathcal F}_{\loc}(w,U)  \, := \inf_{\{w_{j}\}\subset W^{1,q}_{\loc} (U)}  \left\{ \liminf_j \mathcal F(w_{j},U)  \, \colon \,  w_{j} \deb w\  \mbox{in} \ W^{1,p}(U)  \right\} 
$$
whenever $U \subset \Omega$ is an open subset. 
It is not difficult to see that our techniques apply to this case with minor modifications; see also \cite[Section 6]{sharp} to have examples on how to treat the functional in \rif{LSM2} from the point of view of regularity.  
The main difference between $\bar{\mathcal F}_{\loc}(w,U)$ and $\bar{\mathcal F}(w,U)$ defined in \rif{LSM} relies in that, while a measure representation is always possible for the former, this is not the case for the latter and only a weaker type of measure representation holds \cite[Section 3]{FMal}. Problems occur at the boundary in the last case. Anyway, as we are interested in local, interior regularity estimates, such differences are not really important for our purposes. Further discussion on such different definitions can be found in \cite{gme3, thomas}. Relaxed functionals as $\bar{\mathcal F}$ are considered to give a natural extension to $W^{1,p}$ of integral functionals as $\mathcal F$ in \rif{ggg2} under $(p,q)$-growth conditions \rif{assf}$_2$ in the following sense. The idea is that since the original functional $\mathcal F$ coincides with $\bar {\mathcal F}$ on $W^{1,q}$ by Lemma \ref{bf.0}, and it is finite on $W^{1,q}$, one is led to extend it on $W^{1,p}$ by considering its lower semicontinuous envelope with respect to the natural weak topology of $W^{1,p}$ \cite{ma0}. For this reason, and especially in those problems arising from modelling of singularities, sometimes $\bar{\mathcal F}$ is thought to be a more natural functional that the original one $\mathcal F$ to work with. The connection is of course given by \rif{vanish} and the Lavrentiev phenomenon. This can be in fact related to the emergence of cavitation in Nonlinear Elasticity models \cite{madg}. We shall come back to such points in the final Remark \ref{bere}}.
\end{remark} 

\section{Gradient higher integrability}\label{fmi}
Here we consider the functional $\mathcal F$ in \rif{ggg2} with an integrand $F\colon \Omega\times \mathbb{R}^{n}\to \mathbb{R}$ which is this time $C^2$-regular, and satisfies the following strengthened form of \rif{assf}:
\begin{flalign}\label{assfr}
\begin{cases}
\,L_*^{-1}H_{\mu}(z)^{q/2}+\ti{L}^{-1}H_{\mu}(z)^{p/2}
\le F(x,z)\le \ti{L}H_{\mu}(z)^{q/2}+\ti{L}H_{\mu}(z)^{p/2}\vspace{2mm}\\
\, L_*^{-1}H_{\mu}(z)^{(q-2)/2}\snr{\xi}^{2}+\ti{L}^{-1}H_{\mu}(z)^{(p-2)/2}\snr{\xi}^{2}\le \partial_{zz}F(x,z)\xi \cdot \xi\vspace{2mm}\\
\, \snr{\partial_{zz}F(x,z)}\le \ti{L} H_{\mu}(z)^{(q-2)/2}+\ti{L}H_{\mu}(z)^{(p-2)/2}\vspace{2mm}\\
\, \snr{\partial_{z} F(x_{1},z)-\partial_{z} F(x_{2},z)}\le \ti{L}\snr{x_{1}-x_{2}}^{\alpha}[H_{\mu}(z)^{(q-1)/2}+H_{\mu}(z)^{(p-1)/2}]
\end{cases}
 \end{flalign}
for all $z,\xi\in \mathbb{R}^{n}$, $x,x_{1},x_{2}\in \Omega$, where $\ti{L},L_{*}\geq 1$, $0 < \mu \leq 2$. 
It follows
\eqn{1d}
$$
\begin{cases}
\displaystyle 
\, \snr{\partial_{z} F(x,z)}\le c H_{\mu}(z)^{(q-1)/2}+cH_{\mu}(z)^{(p-1)/2}\vspace{2mm}\\\displaystyle
\,  \snr{V_{\mu}(z_{1})-V_{\mu}(z_{2})}^{2}\leq  c\, (\partial_{z} F(x,z_{1})-\partial_{z} F(x,z_{2}))\cdot (z_{1}-z_{2})
\end{cases}
$$
for all $z,z_{1},z_{2}\in \mathbb{R}^{n}$, $x\in \Omega$, with $c\equiv c(n,p,q,\ti{L})$. Indeed, see \cite[Lemma 2.1]{ma1} and \cite[Lemma 3.4]{piovra} for \rif{1d}$_1$; inequality \rif{1d}$_2$ follows in a standard way by \rif{Vm}$_1$, \rif{mon}$_2$ and \rif{assfr}$_2$. In this section we shall deal with minimizers $u \in W^{1,q}_{\loc}(\Omega)$ of the functional in \rif{ggg2} under assumptions \rif{assfr}. 
In this respect, standard regularity theory \cite{manth1, manth2} applies and yields 
 \eqn{areg}
 $$u\in C^{1,\gamma_{0}}_{\loc}(\Omega)\qquad \mbox{for some} \ \  \gamma_{0}\equiv \gamma_{0}(n,p,q,\ti{L},L_{*}, \mu)\in (0,1).$$
 \begin{proposition}\label{scat}
Let $u\in W^{1,q}_{\loc}(\Omega)$ be a minimizer of the functional  $\mathcal F$ in  \eqref{ggg2}, under assumptions \trif{pq} and \eqref{assfr}. Then 
\eqn{hhh}
$$
\mint_{B_{\rr/2}}\snr{Du}^{\qqq}\dx \le c\left(\mint_{B_{2\rr}}\snr{Du}^{p}\dx+1\right)^{\mf{b}_{\qqq}}
$$
holds for every $\qqq\in [1,\infty)$, whenever $B_{2\rr}\Subset \Omega$, $\rr \leq 1$, is a ball, where $c\equiv c(n,p,q,\alpha,\ti{L},\qqq)\geq 1$ and $\mf{b}_{\qqq}\equiv \mf{b}_{\qqq}(n,p,q,\alpha, \qqq)\geq 1$.
\end{proposition}
\subsection{Proof of Proposition \ref{scat}}\label{lb} We take a ball $ B_{2\rr}\equiv B_{2\rr}(x_{0})\Subset  \Omega$ as in the statement of Proposition \ref{scat}, and blow-up $u$ by defining
 \eqn{riscalamento}
 $$\ti{u}(x):=\frac{u(x_{0}+\rr x)}{\rr} \quad \mbox{for every $x \in \mathcal B_2$}, \qquad \ti{u} \stackrel{\rif{areg}}{\in} C^{1, \gamma_0}(\mathcal B_{2})\,,$$ which is a minimizer of the functional
 \eqn{funr}
 $$
 \begin{cases}
 \, \displaystyle W^{1,q}(\mathcal B_2)\ni w\mapsto  \int_{\mathcal B_2}\ti{F}(x,Dw)\dx\vspace{1mm}\\
\,  \ti{F}(x,z):=F(x_{0}+\rr x,z), \quad x \in  \mathcal B_2,\ z \in \er^n\,.
 \end{cases}
$$
It is not difficult to see that $\ti{F}(\cdot)$ still verifies \eqref{assfr}$_{1,2,3}$ and that it satisfies $\eqref{assfr}_{4}$ with $\ti{L}$ replaced by $\ti{L}\rr^{\alpha}$, i.e., 
\eqn{riscalone}
$$
\snr{\partial_{z} \ti{F}(x_{1},z)-\partial_{z} \ti{F}(x_{2},z)}\le \ti{L} \rr^{\alpha}\snr{x_{1}-x_{2}}^{\alpha}[H_{\mu}(z)^{(q-1)/2}+H_{\mu}(z)^{(p-1)/2}]
$$
holds whenever $x_1, x_2 \in \mathcal B_2$ and $z\in \er^n$. 
By \rif{riscalamento}$_2$ and minimality $\ti{u}$ solves 
 \begin{flalign}\label{els}
 \int_{\mathcal B_2} \partial_{z} \ti{F}(x,D\ti{u})\cdot D\varphi\dx=0\qquad \mbox{for all} \ \varphi\in W^{1,1}_{0}(\mathcal B_2)\,.
 \end{flalign}
We shall consider the following conditions on $p$ and $q$:
\eqn{pqb}
$$
q < p+ \gamma\,, \qquad \mbox{where}  \ \gamma :=
\begin{cases}
\,  p\alpha  /2& \mbox{if $p<2$}\vspace{1mm}\\
\,  \alpha & \mbox{if $2 \leq p \leq n$}\vspace{1mm}\\
\,  p\alpha/n & \mbox{if $n < p$}\,.
\end{cases}
 $$Moreover, we define
\eqn{eett0}
$$
A:=\begin{cases}
\, \nr{\ti{u}}_{L^{\infty}(\mathcal B_{1})}+1\quad &\mbox{if} \ \ p\le n\vspace{1mm}\\
\, \nr{\ti{u}}_{L^{\infty}(\mathcal B_{1})}+ [\ti{u}]_{0,1-n/p;\mathcal B_{1}}+1\quad &\mbox{if} \ \ p>n\,.
\end{cases}
$$
 \begin{proposition}\label{tfs} Let $\ti{u}$ be defined in \eqref{riscalamento}. Under assumptions \eqref{assfr} and \eqref{pqb} the a priori estimate
\eqn{15}
$$
\nr{D\ti{u}}_{L^{\qqq}(\mathcal B_{r_1})}^{\qqq}\le \frac{cA^{\mf{b}}}{(r_2-r_1)^{\mf{b}}}\left(\nr{D\ti{u}}_{L^{p}(\mathcal B_{r_2})}^p +1\right)\,,
$$
holds whenever $\qqq \geq 1$, $1/2 \leq r_1 < r_2 \leq 1$, with $c\equiv c(n,p,q,\alpha,\ti{L},\qqq)\geq 1$ and $\mf{b}\equiv \mf{b}(n,p,q,\alpha,\qqq)\linebreak \geq 1$.
\end{proposition}
The proof of Proposition \ref{tfs} will be given in Sections \ref{fm}-\ref{ultima}. We now show how to use Proposition \ref{tfs} to derive Proposition \ref{scat}; consequently, here  we employ the assumptions of this last result. Essentially, we have to find an integral estimate for $A$ in \rif{eett0}. When $p>n$ 
 this is obviously accomplished by Sobolev-Morrey embedding theorem, that is 
 \eqn{em}
$$
\nr{\ti{u}}_{L^{\infty}(\mathcal B_{1})}+ [\ti{u}]_{0,1-n/p;\mathcal B_{1}}\le c(n,p)\left(\mint_{\mathcal B_{1}}(\snr{D\ti{u}}^{p} + \snr{\ti{u}}^{p})\dx\right)^{1/p}\,.
$$
In the case $p\leq n$ we notice that the functional $\mathcal{F}$ is of the type considered in \cite{hisc} and that the bound in \eqref{pq} implies \cite[(6), page 2]{hisc}; see Remark \ref{HSre} below and \cite[Theorem 2.3]{CMM}. Therefore the assumptions of \cite[Theorem 1.1]{hisc} are satisfied and $\tilde u\in L^{\infty}_{\loc}(\mathcal B_2)$; moreover 
\eqn{bb}
$$
\nr{\ti{u}}_{L^{\infty}(\mathcal B_{1})}\le c\left(\mint_{\mathcal B_{2}}(\snr{D\ti{u}}^{p}+\snr{\ti{u}}^{p} +1)\dx\right)^{\mf{e}/p}\,,
$$
holds with $c\equiv c(n,p,q,\ti{L})\geq 1$ and $\mf{e}\equiv \mf{e}(n,p,q)\geq 1$ which is defined in \cite[(18)-(20)]{hisc} (here it is $g(\cdot)\equiv 1$ in the language of \cite{hisc}). No  dependence of the constants on $L_*$ occurs in \rif{bb}. 
Plugging \rif{em} and \rif{bb} in \rif{15} with $r_1=1/2$, $r_2=1$ leads to 
$$
\mint_{\mathcal B_{1/2}}\snr{D\ti{u}}^{\qqq}\dx\le c\left(\mint_{\mathcal B_2}(\snr{D\ti{u}}^{p}+\snr{\ti{u}}^{p}+1)\dx\right)^{\mf{b}_{\qqq}} 
$$ holds with a suitable choice of $c\equiv c(n,p,q,\alpha,\ti{L})\geq 1$ and $\mf{b}_{\qqq}\equiv \mf{b}_{\qqq}(n,p,q,\alpha)\geq 1$.  
Obviously, we can assume that $(\ti{u})_{\mathcal B_2}=0$ as $\ti{u}-(\ti{u})_{\mathcal B_2}$ is still a minimizer, so that Poincarè's inequality implies 
$$
\mint_{\mathcal B_{1/2}}\snr{D\ti{u}}^{\qqq}\dx\le c\left(\mint_{\mathcal B_{2}}(\snr{D\ti{u}}^{p}+1)\dx\right)^{\mf{b}_{\qqq}} 
$$
and finally rescaling back to $u$ leads to \rif{hhh}. 
\begin{remark}\label{HSre}
{\em In \cite{hisc} the integrand $F(\cdot)$ is assumed to verify a doubling condition with respect to the gradient variable \cite[(5)]{hisc}, i.e.,  $F(x, 2z)\lesssim F(x, z)+1$; this is not the case here. However, a quick inspection of the proofs in \cite{hisc} shows that this is not needed in our setting. Specifically, in \cite[Proof of Theorem 3.1, Step 1]{hisc}  the authors deal with terms of the type $\nr{F(\cdot, 2Du)}_{L^1(\{u\geq M\}\cap {\rm B})}$. In \cite{hisc}, thanks to the doubling property and $F(\cdot, Du)\in L^1$, such terms disappear after letting $M\to \infty$, eventually leading to \cite[(24)]{hisc}. The same result follows knowing that $Du \in L^q$ in view of \rif{assfr}$_1$, as in fact it is here. See \cite{cianchi} for related results.}
\end{remark}
\subsection{Numerology for Proposition \ref{tfs}}\label{fm}
We set 
\eqn{defitante}
$$
\begin{cases}
\, d:=\max\{p,2\}\vspace{1mm} \\
\, t_{*}:=2q+\max\{2p,p^{2}/n\}+4 \vspace{1mm}\\
\, \kk_{0}:=(1-n/p)_{+}\,.
\end{cases}
$$
Recalling that $\gamma$ has been introduced in \rif{pqb}, we further define 
\eqn{illambda}
$$\lambda:=2(\gamma-\kk)\kk_{0}, \quad \mbox{where $\kk$ is taken such that $0 < \kk  < \min\{p+\gamma-q,\alpha\}$}/2< \gamma\,.
$$ 
With $t\ge t_{*}$ as in \rif{defitante}$_2$, we introduce 
\eqn{ts}
 $$
 \sigma\equiv \sigma(t):=t-2q+p+1+2\gamma-2\kk,\quad\quad  \vartheta\equiv \vartheta(t):=\frac{\sigma-2-2\gamma+2\kk+\alpha+\lambda}{\sigma}\,.
 $$
When $p\in (1,2)$ we shall also use the exponent 
 \eqn{ts.1}
 $$
 \theta\equiv \theta(t):= \frac{p(1+\vartheta\sigma+\alpha)}{2}+\frac{(2-p)(\sigma-1-2\gamma+2\kk)}{2}= t-2q+p(1+\alpha)\,.
 $$
 Finally, we set
     \eqn{ss12}
$$
\sigma_{1}\equiv \sigma_{1}(t):=\begin{cases}
\, \sigma-1+p\quad &\mbox{if} \ \ p\le n\vspace{.5mm} \\
\, p^{2}/n-2\kk\quad &\mbox{if} \ \ p>n\,,
\end{cases} 
$$
and 
     \eqn{ss122}
$$
\sigma_{2}\equiv \sigma_{2}(t):=\begin{cases}
\ 1+\vartheta\sigma+\alpha \quad &\mbox{if} \ \   p\le n\vspace{.5mm}\\
\ 0\quad &\mbox{if} \ \ p>n\,.
\end{cases}
$$
Note that 
 \eqn{foll}
$$
\begin{cases}
\, \sigma-1-2\gamma+2\kk=t-2q+p=1+ \vartheta\sigma -\alpha-\lambda\vspace{.6mm}\\
\, \mbox{$\lambda >0$ if and only if $p>n$, and $\lambda =0$ for $p\leq  n$}\vspace{.6mm}\\
\,   \sigma> p+1+2(\gamma-\kk)\vspace{.6mm}\\
\, 0 < \vartheta <1\vspace{.6mm}\\
\,   0 < \theta < \sigma -1+p\,.
\end{cases}
$$
\begin{remark}\label{sigmat}{\em In the following computations, any dependence of the constants generically denoted by $c$ on the parameter $\sigma$ in \rif{ts} will be indicated as a dependence on $t$. Indeed, we shall eventually choose $t$ large during the final iteration of Section \ref{ultima}, and, as consequence, this will influence the size of $\sigma$ which is a function of $t$.}
\end{remark}
 \subsection{Basic estimates for Proposition \ref{tfs}}\label{fm2}
We introduce radii $$1/2\leq r_{1}\le \tau_{1}<\tau_{2}\leq r_{2}\le 1\,,$$and, with  $K\geq 2$, set \eqn{raggimain}
$$\rri:=\tau_{1}+\frac{\tau_{2}-\tau_{1}}{2K} < \tau_{1}+\frac{\tau_{2}-\tau_{1}}{K}=:\rrs\,.$$ 
Accordingly, we define 
\eqn{iraggipre}
$$
  \ri:=\frac{3\rri+\rrs}{4}, \qquad \ri^*:=\frac{2\rri+\rrs}{3}, \qquad  
\rs:=\frac{\rri+\rrs}{2} \,.
$$ Note that 
\eqn{iraggi}
$$1/2\leq r_1\leq \tau_{1}< \rri < \ri< \ri^*< \rs< \rrs< \tau_2 \leq r_2\leq 1
$$
and 
\eqn{iraggiK}
$$
 \frac{\tau_{2}-\tau_{1}}{24K}\leq \ri- \rri ,  \ri^*-\ri , 
\rs- \ri^*,   \rrs-\rs \leq \frac{\tau_{2}-\tau_{1}}{2K}\,.
$$
We then pick $h\in \mathbb{R}^{n}$ and a cut-off function $\eta\in C^{2}_{0}(\mathcal B_1)$ such that 
\eqn{vettori}
$$
\begin{cases}
\, \displaystyle 0< \snr{h} < \mathcal{h}_{0}:=\frac{\tau_{2}-\tau_{1}}{10^4K}\vspace{1.5mm}\\
\, \mathds{1}_{\mathcal B_{\ri^*}}\le \eta\le \mathds{1}_{\mathcal B_{\rs}}, \quad 
 \displaystyle \snr{D\eta}^2+ \snr{D^{2}\eta}\lesssim  \frac1{(\rs-\ri^*)^{2}}\stackrel{\rif{iraggiK}}{\lesssim}\frac {K^2}{(\tau_2-\tau_1)^{2}}\,.
\end{cases}
$$
 In \eqref{els} we take 
 $$\varphi:=\frac{\tau_{-h}(\eta^{d}\snr{\tau_{h}\ti{u}}^{\sigma-1}\tau_{h}\ti{u})}{\snr{h}^{1+\vartheta\sigma+\alpha}}\,,
 $$ 
 which is admissible by \rif{riscalamento}$_2$ and it is compactly supported in $\mathcal B_{\rrs}$. Integration-by-parts for finite differences yields
  \begin{eqnarray}\label{recalling}
0&=&\frac{1}{\snr{h}^{1+\vartheta\sigma+\alpha}}\int_{\mathcal B_{1}}\tau_{h}[\partial_{z} \ti{F}(\cdot,D\ti{u})]\cdot D (\eta^{d}\snr{\tau_{h}\ti{u}}^{\sigma-1}\tau_{h}\ti{u})\dx\nonumber \\
&=&\frac{1}{\snr{h}^{1+\vartheta\sigma+\alpha}}\int_{\mathcal B_{1}} [\partial_{z} \ti{F}(x+h,D\ti{u}(x+h))-\partial_{z} \ti{F}(x,D\ti{u}(x+h))]\cdot D(\eta^{d}\snr{\tau_{h}\ti{u}}^{\sigma-1}\tau_{h}\ti{u} )\dx\nonumber \\
&&\ +\frac{1}{\snr{h}^{1+\vartheta\sigma+\alpha}}\int_{\mathcal B_{1}}[\partial_{z} \ti{F}(x,D\ti{u}(x+h))-\partial_{z} \ti{F}(x,D\ti{u}(x))]\cdot D(\eta^{d}\snr{\tau_{h}\ti{u}}^{\sigma-1}\tau_{h}\ti{u})\dx\nonumber \\
&=&\frac{\sigma}{\snr{h}^{1+\vartheta\sigma+\alpha}}\int_{\mathcal B_{1}}\eta^{d}\snr{\tau_{h}\ti{u}}^{\sigma-1}[ \partial_{z} \ti{F}(x+h,D\ti{u}(x+h))-\partial_{z} \ti{F}(x,D\ti{u}(x+h))]\cdot \tau_{h}D\ti{u}\dx\nonumber \\
&& \ +\frac{d}{\snr{h}^{1+\vartheta\sigma+\alpha}}\int_{\mathcal B_{1}}\eta^{d-1} \snr{\tau_{h}\ti{u}}^{\sigma-1}\tau_{h}\ti{u}[\partial_{z} \ti{F}(x+h,D\ti{u}(x+h))-\partial_{z} \ti{F}(x,D\ti{u}(x+h))]\cdot D\eta\dx\nonumber \\
&&\ +\frac{\sigma}{\snr{h}^{1+\vartheta\sigma+\alpha}}\int_{\mathcal B_{1}}\eta^{d}\snr{\tau_{h}\ti{u}}^{\sigma-1}[\partial_{z} \ti{F}(x,D\ti{u}(x+h))-\partial_{z} \ti{F}(x,D\ti{u}(x))]\cdot \tau_{h}D\ti{u}\dx\nonumber \\
&&\ +\frac{d}{\snr{h}^{1+\vartheta\sigma+\alpha}}\int_{\mathcal B_{1}}\eta^{d-1}\snr{\tau_{h}\ti{u}}^{\sigma-1}\tau_{h}\ti{u}[ \partial_{z} \ti{F}(x,D\ti{u}(x+h))-\partial_{z} \ti{F}(x,D\ti{u}(x))]\cdot D\eta\dx\nonumber \\
&& =:\mbox{(I)}+\mbox{(II)}+\mbox{(III)}+\mbox{(IV)}.
\end{eqnarray}
In the following we shall abbreviate as
\eqn{eett}
$$
\begin{cases}
\,  
\displaystyle \mathcal{G}(h)\equiv \mathcal{G}(x,h):=\int_{0}^{1}\partial_{zz}\ti{F}(x,D\ti{u}(x)+\ttl \tau_{h}D\ti{u}(x))\d\ttl \vspace{2mm} \\
\, \displaystyle \mathcal{D}(h)\equiv \mathcal{D}(x,h):= \snr{D\ti{u}(x)}^{2}+\snr{D\ti{u}(x+h)}^{2}+\mu^{2}\vspace{2mm}\\
\, \displaystyle H:=H_{1}(D\tilde u)=\snr{D\ti{u}}^{2}+1 \vspace{2mm}\\
\, \displaystyle  \mathcal{N}_{\infty}^2:=\nr{D\eta}_{L^{\infty}}^{2}+\nr{D^{2}\eta}_{L^{\infty}}+1\lesssim \frac{K^2}{(\tau_2-\tau_1)^{2}}\,. 
\end{cases}
$$
Note that both $ \mathcal{G}(x,h)$ and $\mathcal{D}(x,h)$ are defined for $x \in \mathcal B_{2-\snr{h}}$. 
By first using \rif{assfr}$_2$ and \rif{mon}$_2$ with $\mf{t}=p-2$, it follows
\begin{flalign}
\nonumber \mbox{(III)}&=\frac{\sigma}{\snr{h}^{1+\vartheta\sigma+\alpha}}\int_{\mathcal B_{1}}\eta^{d}\snr{\tau_{h}\ti{u}}^{\sigma-1}  \mathcal{G}(h)\tau_{h}D\ti{u}\cdot \tau_{h}D\ti{u}\dx
\\
& \geq \frac{\sigma}{\ti{c}\snr{h}^{1+\vartheta\sigma+\alpha}}\int_{\mathcal B_{1}}\eta^{d}\snr{\tau_{h}\ti{u}}^{\sigma-1}\mathcal{D}(h)^{(p-2)/2}\snr{\tau_{h}D\ti{u}}^{2}\dx=:\frac{1}{\ti{c}}\mbox{(V)} \label{notache}
\end{flalign}
with $\ti{c}\equiv \ti{c} (n,p)$, so that 
\eqn{notache2}
$$
\mbox{(III)} \geq \frac 12 \mbox{(III)} + \frac{1}{2\ti{c}}\mbox{(V)}\,.
$$
We use the definition of $A$ in \rif{eett0} and of $\kappa_0, \lambda$ in \rif{defitante}$_3$, \rif{illambda} to estimate, for every $x\in  \mathcal B_{1-|h|}$
\eqn{issa}
$$
\begin{cases}
\, \snr{\tau_{h}\ti{u}(x)}^{\sigma-1}\lesssim  A^{2(\gamma-\kk)}\snr{\tau_{h}\ti{u}(x)}^{\sigma-1-2\gamma+2\kk}\snr{h}^{\lambda}\vspace{1.5mm}\\
\, \snr{\tau_{h}\ti{u}(x)}^{\sigma}\lesssim_{\delta_1} A^{\delta_1}\snr{\tau_{h}\ti{u}(x)}^{\sigma-\delta_1}
\snr{h}^{\delta_1\kappa_0}, \quad 0\leq  \delta_1 < \sigma\vspace{1.5mm}\\
\, \snr{\tau_{h}\ti{u}(x)}^{\sigma+1}\lesssim_{\delta_2} A^{\delta_2}\snr{\tau_{h}\ti{u}(x)}^{\sigma+1-\delta_2}
\snr{h}^{\delta_2\kappa_0}, \quad 0\leq  \delta_2 < \sigma+1\vspace{1.5mm}\\
\, \snr{\tau_{h}\ti{u}(x)}^{\sigma-1+p}\lesssim_{\delta_3} A^{\delta_3}\snr{\tau_{h}\ti{u}(x)}^{\sigma-1+p-\delta_3}
\snr{h}^{\delta_3\kappa_0}, \quad 0\leq  \delta_3 < \sigma\,.
\end{cases}
$$ 
Recalling \rif{illambda}, via H\"older's and Young's inequalities we now have
\begin{eqnarray}\label{t1.1}
\snr{\mbox{(I)}}&\stackrel{\eqref{riscalone}}{\le}&\frac{\sigma \ti{L} \rr^{\alpha}\snr{h}^{\alpha}}{\snr{h}^{1+\vartheta\sigma+\alpha}}\int_{\mathcal B_{1}}\eta^{d}\snr{\tau_{h}\ti{u}}^{\sigma-1}\left(\mathcal{D}(h)^{(q-1)/2}+\mathcal{D}(h)^{(p-1)/2}\right)\snr{\tau_{h}D\ti{u}}\dx\nonumber \\
&\le&\frac 1{4\ti{c}}\mbox{(V)}+\frac{c\rr^{2\alpha}}{\snr{h}^{1+\vartheta\sigma-\alpha}}\int_{\mathcal B_{1}}\eta^{d}\snr{\tau_{h}\ti{u}}^{\sigma-1}\left(\mathcal{D}(h)^{q-p/2}+\mathcal{D}(h)^{p/2}\right)\dx\nonumber \\
&\stackrel{\rif{issa}_1}{\leq}&\frac 1{4\ti{c}}\mbox{(V)}+c A^{2(\gamma-\kk)}\int_{\mathcal B_{\rs}}\left|\frac{\tau_{h}\ti{u}}{\snr{h}^{\frac{1+\vartheta\sigma-\alpha-\lambda}{\sigma-1-2\gamma+2\kk}}}\right|^{\sigma-1-2\gamma+2\kk}[\mathcal{D}(h)+1]^{q-p/2}\dx\nonumber \\
&\le&\frac 1{4\ti{c}}\mbox{(V)}+c A^{2(\gamma-\kk)}\left(\int_{\mathcal B_{\rs}}\left|\frac{\tau_{h}\ti{u}}{\snr{h}^{\frac{1+\vartheta\sigma-\alpha-\lambda}{\sigma-1-2\gamma+2\kk}}}\right|^{\frac{t(\sigma-1-2\gamma+2\kk)}{t-2q+p}}\dx\right)^{\frac{t-2q+p}{t}}\notag\\
&& \qquad \qquad\qquad \qquad\quad  \cdot\left(\int_{\mathcal B_{\rrs}}H^{t/2}\dx\right)^{\frac{2q-p}{t}}\nonumber \\
&\stackrel{\eqref{foll}}{=}&\frac 1{4\ti{c}}\mbox{(V)}+c A^{2(\gamma-\kk)}\left(\int_{\mathcal B_{\rs}}\left|\frac{\tau_{h}\ti{u}}{\snr{h}}\right|^{t}\dx\right)^{\frac{t-2q+p}{t}}\left(\int_{\mathcal B_{\rrs}}H^{t/2}\dx\right)^{\frac{2q-p}{t}}\nonumber \\
&\stackrel{\eqref{gh}}{\le}&\frac 1{4\ti{c}}\mbox{(V)}+\underbrace{c A^{2(\gamma-\kk)}\int_{\mathcal B_{\rrs}}H^{t/2}\dx}_{\displaystyle =: c\mbox{(VI)}_{1}}\label{saisai}\,.
\end{eqnarray}
Here we have also used the elementary estimate 
\eqn{elmy}
$$
\int_{\mathcal B_{\rs}}[\mathcal{D}(h)+1]^{t/2}\dx \lesssim_{t} \int_{\mathcal B_{\rrs}}H^{t/2}\dx
$$
that follows via simple change-of-variables and \rif{iraggiK}. We shall use inequalities like this throughout all the rest of the proof. Using \rif{recalling}, \rif{notache2} and \rif{saisai} we conclude with 
\eqn{siparte}
$$
\frac12\mbox{(III)}+ \frac{1}{4\ti{c}}\mbox{(V)}\leq c\mbox{(VI)}_{1} + \snr{\mbox{(II)}+ \mbox{(IV)}}
$$
 where $c\equiv c(n,p,q,\alpha,\ti{L},t)$. 
 \subsection{Estimates for the case $p\ge 2$}
Recalling the definition of $\mbox{(V)}$ in \rif{notache} we obtain
$$
\mbox{(VII)}:= \frac{\sigma}{\snr{h}^{1+\vartheta\sigma+\alpha}}\int_{\mathcal B_{1}}\eta^{d}\snr{\tau_{h}\ti{u}}^{\sigma-1}\snr{\tau_{h}D\ti{u}}^{p}\dx \lesssim_{n,p} \mbox{(V)}\,.
$$
This and \rif{siparte} yields 
\eqn{siparte2}
$$
\mbox{(III)}
+ \mbox{(VII)}
\leq c\mbox{(VI)}_{1} + c\snr{\mbox{(II)}}+ c \snr{\mbox{(IV)}}
$$
for an increased constant $c\equiv c(n,p,q,\alpha,\ti{L},t)$. 
 Next, for $\delta_{1}\in [0,\sigma)$ to be fixed later on, we bound, via \eqref{riscalone} and \rif{issa}, as follows:
\begin{flalign}\label{t1.12}
\snr{\mbox{(II)}}& \leq \frac{c\rr^{\alpha}}{\snr{h}^{1+\vartheta\sigma}}\int_{\mathcal B_{1}}\eta^{d-1}\snr{D\eta}\snr{\tau_{h}\ti{u}}^{\sigma}\left(\mathcal{D}(h)^{(q-1)/2}+\mathcal{D}(h)^{(p-1)/2}\right)\dx\nonumber \\
& \leq \frac{c\rr^{\alpha}A^{\delta_{1}}}{\snr{h}^{1+\vartheta\sigma-\delta_{1}\kk_{0}}}\int_{\mathcal B_{1}}\eta^{d-1}\snr{D\eta}\snr{\tau_{h}\ti{u}}^{\sigma-\delta_{1}}\left[\mathcal{D}(h)+1\right]^{(q-1)/2}\dx\nonumber \\
&\le cA^{\delta_{1}}\int_{\mathcal B_{\rs}}\snr{D\eta}^{2}\left|\frac{\tau_{h}\ti{u}}{\snr{h}^{\frac{1+\vartheta\sigma-\delta_{1}\kk_{0}}{\sigma-\delta_{1}}}}\right|^{\sigma-\delta_{1}}\left[\mathcal{D}(h)+1\right]^{(p-1)/2}\dx\nonumber \\
&\qquad +cA^{\delta_{1}}\rr^{2\alpha}\int_{\mathcal B_{\rs}}\eta^{2(d-1)}\left|\frac{\tau_{h}\ti{u}}{\snr{h}^{\frac{1+\vartheta\sigma-\delta_{1}\kk_{0}}{\sigma-\delta_{1}}}}\right|^{\sigma-\delta_{1}}[\mathcal{D}(h)+1]^{(2q-p-1)/2}\dx\nonumber \\
&\le cA^{\delta_{1}}\left(\int_{\mathcal B_{\rs}}\snr{D\eta}^{\frac{2t}{t-p+1}}\left|\frac{\tau_{h}\ti{u}}{\snr{h}^{\frac{1+\vartheta\sigma-\delta_{1}\kk_{0}}{\sigma-\delta_{1}}}}\right|^{\frac{(\sigma-\delta_{1}) t}{t-p+1}}\dx\right)^{\frac{t-p+1}{t}}\left(\int_{\mathcal B_{\rrs}}H^{t/2}\dx\right)^{\frac{p-1}{t}}\nonumber \\
&\qquad +cA^{\delta_{1}}\rr^{2\alpha}\left(\int_{\mathcal B_{\rs}}\eta^{2(d-1)}\left|\frac{\tau_{h}\ti{u}}{\snr{h}^{\frac{1+\vartheta\sigma-\delta_{1}\kk_{0}}{\sigma-\delta_{1}}}}\right|^{\frac{(\sigma-\delta_{1}) t}{t-2q+p+1}}\dx\right)^{\frac{t-2q+p+1}{t}}\left(\int_{\mathcal B_{\rrs}}H^{t/2}\dx\right)^{\frac{2q-p-1}{t}}\nonumber \\
&\leq  \underbrace{c \mathcal{N}_{\infty}^2A^{\delta_{1}}\left(\int_{\mathcal B_{\rs}} \left|\frac{\tau_{h}\ti{u}}{\snr{h}^{\frac{1+\vartheta\sigma-\delta_{1}\kk_{0}}{\sigma-\delta_{1}}}}\right|^{\frac{(\sigma-\delta_{1}) t}{t-2q+p+1}}\dx\right)^{\frac{t-2q+p+1}{t}}\left(\int_{\mathcal B_{\rrs}}H^{t/2}\dx\right)^{\frac{2q-p-1}{t}}}_{\displaystyle =:c\mbox{(VI)}_{2} }
\end{flalign}
for $c\equiv c(n,p,q,\alpha,\ti{L},t)$; we have used \rif{elmy} and that $H\geq 1$. 
For $\mbox{(IV)}$, we preliminary observe that \rif{mon}$_2$, \eqref{assfr}$_{3}$ and $p\geq 2$ imply
$$
| \mathcal{G}(h)| \lesssim_{n,p,q, \ti{L}} \mathcal{D}(h)^{(q-2)/2}+ \mathcal{D}(h)^{(p-2)/2} \leq 2
\left[\mathcal{D}(h)+1\right]^{(q-2)/2}
$$
then, also using \rif{issa}$_3$ and \rif{elmy},  when $q>2$ Cauchy-Schwarz and Young's inequalities yield
\begin{flalign}
\snr{\mbox{(IV)}}&=\frac{d}{\snr{h}^{1+\vartheta\sigma+\alpha}}\left|\int_{\mathcal B_{1}}\eta^{d-1}\snr{\tau_{h}\ti{u}}^{\sigma-1}\tau_{h}\ti{u}   \mathcal{G}(h)\tau_{h}D\ti{u}\cdot D\eta \dx\right|\nonumber \\
&\le \frac{d}{\snr{h}^{1+\vartheta\sigma+\alpha}}\left(\int_{\mathcal B_{1}}\eta^{d}\snr{\tau_{h}\ti{u}}^{\sigma-1} \mathcal{G}(h)\tau_{h}D\ti{u}\cdot \tau_{h}D\ti{u}\dx\right)^{1/2}\nonumber \\
& \qquad\qquad\qquad\quad \cdot \left(\int_{\mathcal B_{1}}\eta^{d-2}\snr{D\eta}^{2}\snr{\mathcal{G}(h)}\snr{\tau_{h}\ti{u}}^{\sigma+1}\dx\right)^{1/2}\nonumber \\
&\leq \eps\mbox{(III)}+c_{\eps}A^{\delta_{2}}\int_{\mathcal B_{\rs}}\eta^{d-2}\snr{D\eta}^{2}\left|\frac{\tau_{h}\ti{u}}{\snr{h}^{\frac{1+\vartheta\sigma+\alpha-\delta_{2}\kk_{0}}{\sigma+1-\delta_{2}}}}\right|^{\sigma+1-\delta_{2}}[\mathcal{D}(h)+1]^{(q-2)/2}\dx\nonumber \\
&\le\eps\mbox{(III)}+\underbrace{ c_{\eps}\mathcal{N}_{\infty}^2A^{\delta_{2}}\left(\int_{\mathcal B_{\rs}}\left|\frac{\tau_{h}\ti{u}}{\snr{h}^{\frac{1+\vartheta\sigma+\alpha-\delta_{2}\kk_{0}}{\sigma+1-\delta_{2}}}}\right|^{\frac{(\sigma+1-\delta_{2})t}{t-q+2}}\dx\right)^{\frac{t-q+2}{t}}\left(\int_{\mathcal B_{\rrs}}H^{t/2}\dx\right)^{\frac{q-2}{t}}}_{\displaystyle =:c_{\eps} \mbox{(VI)}_{3}}\label{pip2}
\end{flalign}
with $c_{\eps}\equiv c_{\eps}(n,p,q,\ti{L},t, \eps)$ and $\eps \in (0,1)$. When $q=2$ the previous estimate remains unchanged without using H\"older's inequality as in the last line of the previous display. Using \rif{t1.12} and \rif{pip2} in \rif{siparte2} and choosing $\eps\equiv \eps(n,p,q,\alpha,\ti{L},t)>0$ small enough in order to reabsorb the term involving $\mbox{(III)}$ appearing on the right-hand side, we obtain
\eqn{1}
$$
\mbox{(VII)}\le c \mbox{(VI)}_{1}+c\mbox{(VI)}_{2}+c\mbox{(VI)}_{3}
$$
for $c\equiv c(n,p,q,\alpha,\ti{L},t)$. Next, we provide a lower bound for $\mbox{(VII)}$. Recalling that now it is $d=p$ and also using \rif{issa}$_4$ we have
\begin{eqnarray}\label{11.5}
\mbox{(VII)}\notag &\ge&\frac{1}{c\snr{h}^{1+\vartheta\sigma+\alpha}}\int_{\mathcal B_{1}}\left|D\left(\eta^{d/p}\snr{\tau_{h}\ti{u}}^{(\sigma-1)/p}\tau_{h}\ti{u}\right)\right|^{p}\dx\\ 
&& \qquad -cA^{\delta_{3}}\int_{\mathcal B_{1}}\snr{D\eta}^{p}\left|\frac{\tau_{h}\ti{u}}{\snr{h}^{\frac{1+\vartheta\sigma+\alpha-\delta_{3}\kk_{0}}{\sigma-1+p-\delta_{3}}}}\right|^{\sigma-1+p-\delta_{3}}\dx\nonumber \\
&\stackrel{\eqref{gh}}{\ge}&\frac{1}{c\snr{h}^{1+\vartheta\sigma+\alpha+p}}\int_{\mathcal B_{\rs}}\left|\tau_{h}\left(\eta^{d/p}\snr{\tau_{h}\ti{u}}^{(\sigma-1)/p}\tau_{h}\ti{u}\right)\right|^{p}\dx\notag \\ 
&& \qquad -cA^{\delta_{3}}\int_{\mathcal B_{\rs}}\snr{D\eta}^{p}\left|\frac{\tau_{h}\ti{u}}{\snr{h}^{\frac{1+\vartheta\sigma+\alpha-\delta_{3}\kk_{0}}{\sigma-1+p-\delta_{3}}}}\right|^{\sigma-1+p-\delta_{3}}\dx\nonumber \\
&\stackrel{\eqref{prod}}{\ge}&\frac{1}{c\snr{h}^{1+\vartheta\sigma+\alpha+p}}\int_{\mathcal B_{\rs}}\eta(x+h)^{d}\left|\tau_{h}\left(\snr{\tau_{h}\ti{u}}^{(\sigma-1)/p}\tau_{h}\ti{u}\right)\right|^{p}\dx\nonumber \\
&&\qquad -cA^{\delta_{3}}\int_{\mathcal B_{\rs}}\snr{D\eta}^{p}\left|\frac{\tau_{h}\ti{u}}{\snr{h}^{\frac{1+\vartheta\sigma+\alpha-\delta_{3}\kk_{0}}{\sigma-1+p-\delta_{3}}}}\right|^{\sigma-1+p-\delta_{3}}\dx\notag \\ && \qquad -\frac{1}{\snr{h}^{1+\vartheta\sigma+\alpha+p}}\int_{\mathcal B_{\rs}}\snr{\tau_{h}\eta}^{p}\snr{\tau_{h}\ti{u}}^{\sigma-1+p}\dx\nonumber \\
&\geq &\frac{1}{c\snr{h}^{1+\vartheta\sigma+\alpha+p}}\int_{\mathcal B_{\rs}}\eta(x+h)^{d}\left|\tau_{h}\left(\snr{\tau_{h}\ti{u}}^{(\sigma-1)/p}\tau_{h}\ti{u}\right)\right|^{p}\dx\nonumber \\
&& \qquad -\underbrace{c\mathcal{N}_{\infty}^{p}A^{\delta_{3}}\int_{\mathcal B_{\rs}}\left|\frac{\tau_{h}\ti{u}}{\snr{h}^{\frac{1+\vartheta\sigma+\alpha-\delta_{3}\kk_{0}}{\sigma-1+p-\delta_{3}}}}\right|^{\sigma-1+p-\delta_{3}}\dx}_{\displaystyle =: c \mbox{(VI)}_{4}}\nonumber \\
& \ge&\frac{1}{c\snr{h}^{1+\vartheta\sigma+\alpha+p}}\int_{\mathcal  B_{  \ri}}\left|\tau_{h}\left(\snr{\tau_{h}\ti{u}}^{(\sigma-1)/p}\tau_{h}\ti{u}\right)\right|^{p}\dx-c \mbox{(VI)}_{4},
\end{eqnarray}
where $c\equiv c(n,p,q,\ti{L},t)$ and, according to Remark \ref{sigmat}, we have incorporated in $t$ the dependence on $\sigma$. 
Note that here we have used the standard mean value theorem for smooth functions to infer
\eqn{filu}	
$$
\notag \frac{1}{\snr{h}^{1+\vartheta\sigma+\alpha+p}}\int_{\mathcal B_{\rs}}\snr{\tau_{h}\eta}^{p}\snr{\tau_{h}\ti{u}}^{\sigma-1+p}\dx \leq \frac{\mathcal{N}_{\infty}^{p}}{\snr{h}^{1+\vartheta\sigma+\alpha}}\int_{\mathcal B_{\rs}}\snr{\tau_{h}\ti{u}}^{\sigma-1+p}\dx\leq \mbox{(VI)}_{4}\,.
$$
We have also used the fact that $\eta(x+h)=1$ when $x \in \mathcal B_{\ri}$; see \rif{iraggiK}-\rif{vettori}. 
From \rif{11.5}, using \rif{mon}$_1$ with $\mathfrak{p}\equiv (\sigma-1+p)/p>1$, we conclude with
\eqn{2}
$$
\mbox{(VII)}\ge \frac 1c\int_{\mathcal  B_{\ri}}\left|\frac{\tau_{h}^{2}\ti{u}}{\snr{h}^{\frac{1+\vartheta\sigma+\alpha+p}{\sigma-1+p}}}\right|^{\sigma-1+p}\dx-c \mbox{(VI)}_{4}\,.
$$
Merging \eqref{1} and \eqref{2} yields
\eqn{3}
$$
\int_{\mathcal  B_{\ri}}\left|\frac{\tau_{h}^{2}\ti{u}}{\snr{h}^{\frac{1+\vartheta\sigma+\alpha+p}{\sigma-1+p}}}\right|^{\sigma-1+p}\dx\le c \mbox{(VI)}_{1}+c\mbox{(VI)}_{2}+c\mbox{(VI)}_{3}+c\mbox{(VI)}_{4}\,,
$$
for $c\equiv c(n,p,q,\alpha,\ti{L},t)$. To control terms $\mbox{(VI)}_{2}$-$ \mbox{(VI)}_{4}$, we treat separately cases $2\le p\le n$ and $p>n$. Note that these terms contain the free parameters $\delta_1, \delta_2, \delta_3\ge 0$ that are still to be chosen. 
\subsubsection{The case $2\le p\le n$.} In this situation, by \rif{pqb} and \rif{defitante}-\rif{illambda} we have $\kk_0=\lambda =0$, $\gamma=\alpha$, and take $\delta_{1}=\delta_{2}= \delta_{3}\equiv0$. By \eqref{ts} and \eqref{foll} we infer
\eqn{foll2}
$$
\left\{
\begin{array}{c}
\displaystyle 
\, \max\left\{\frac{1+\vartheta\sigma}{\sigma},\frac{1+\vartheta\sigma+\alpha}{\sigma+1},\frac{1+\vartheta\sigma+\alpha}{\sigma-1+p}\right\}<1\\ [15pt]\displaystyle
\, \min\left\{\frac{1+\vartheta\sigma}{\sigma},\frac{1+\vartheta\sigma+\alpha}{\sigma+1},\frac{1+\vartheta\sigma+\alpha}{\sigma-1+p}\right\}>0 \\ [15pt]\displaystyle
\,\max\left\{\frac{t(1+\vartheta\sigma)}{t-2q+p+1},\frac{t(1+\vartheta\sigma+\alpha)}{t-q+2}\right\}\le t \\ [15pt]\displaystyle
\, \min\left\{\frac{t(1+\vartheta\sigma)}{t-2q+p+1},\frac{t(1+\vartheta\sigma+\alpha)}{t-q+2}\right\}> 1,
\end{array}
\right.
$$ 
so we can apply Lemma \ref{ls} to bound terms $\mbox{(VI)}_{2}$-$ \mbox{(VI)}_{4}$. We first look at $\mbox{(VI)}_{2}$. In Lemma \ref{ls} we take $s\equiv (1+\vartheta\sigma)/\sigma$ and $m\equiv t(1+\vartheta\sigma)/(t-2q+ p+1)$. We get, using \rif{foll2} repeatedly
\begin{flalign}
\notag  \mbox{(VI)}_{2}&\le \frac{c \mathcal{N}_{\infty}^{2}A^{\sigma(1-\vartheta)-1}}{(\rrs-\rs)^{1+\vartheta\sigma}}\left(\int_{\mathcal B_{\rrs}}\snr{\ti{u}}^{\frac{t(1+\vartheta\sigma)}{t-2q+p+1}}\dx\right)^{\frac{t-2q+p+1}{t}}\left(\int_{\mathcal B_{\rrs}}H^{t/2}\dx\right)^{\frac{2q-p-1}{t}}\nonumber \\
\notag &\qquad + c \mathcal{N}_{\infty}^{2}A^{\sigma(1-\vartheta)-1}\left(\int_{\mathcal B_{\rrs}}\snr{D\ti{u}}^{\frac{t(1+\vartheta\sigma)}{t-2q+p+1}}\dx\right)^{\frac{t-2q+p+1}{t}}\left(\int_{\mathcal B_{\rrs}}H^{t/2}\dx\right)^{\frac{2q-p-1}{t}}\nonumber \\
&\le\frac{c  \mathcal{N}_{\infty}^{2}A^{\sigma}}{(\rrs-\rs)^{1+\vartheta\sigma}}\left(\int_{\mathcal B_{\rrs}}H^{t/2}\dx\right)^{\frac{2q-p-1}{t}}+c \mathcal{N}_{\infty}^{2}A^{\sigma(1-\vartheta)-1} \int_{\mathcal B_{\rrs}}H^{t/2}\dx\nonumber \\
&\le\frac{c \mathcal{N}_{\infty}^{2} K^{1+\vartheta\sigma}A^{\sigma}}{(\tau_2-\tau_1)^{1+\vartheta\sigma}} \int_{\mathcal B_{\rrs}}H^{t/2}\dx\,.\label{caso3}
\end{flalign}
The constant $c$ depends on  $n,p,q,\alpha,\ti{L},t$ and in the last line of the above display we have used \rif{iraggiK}.  
For $\mbox{(VI)}_{3}$ we again use \eqref{fs.0} with $s\equiv (1+\vartheta\sigma+\alpha)/(\sigma+1)$, $m\equiv t(1+\vartheta\sigma+\alpha)/(t-q+2)$, and \rif{foll2} repeatedly, to get
\begin{flalign}\label{caso4}
 \mbox{(VI)}_{3}&\le\frac{c\mathcal{N}_{\infty}^{2}A^{\sigma(1-\vartheta)-\alpha}}{(\rrs-\rs)^{1+\vartheta\sigma+\alpha}}\left(\int_{\mathcal B_{\rrs}}\snr{\ti{u}}^{\frac{t(1+\vartheta\sigma+\alpha)}{t-q+2}}\dx\right)^{\frac{t-q+2}{t}}\left(\int_{\mathcal B_{\rrs}}H^{t/2}\dx\right)^{\frac{q-2}{t}}\nonumber \\
  \notag   &\quad +c\mathcal{N}_{\infty}^{2}A^{\sigma(1-\vartheta)-\alpha}\left(\int_{\mathcal B_{\rrs}}\snr{D\ti{u}}^{\frac{t(1+\vartheta\sigma+\alpha)}{t-q+2}}\dx\right)^{\frac{t-q+2}{t}}\left(\int_{\mathcal B_{\rrs}}H^{t/2}\dx\right)^{\frac{q-2}{t}}\nonumber \\
    &\le \frac{c\mathcal{N}_{\infty}^{2}A^{\sigma+1}}{(\rrs-\rs)^{1+\vartheta\sigma+\alpha}}\left(\int_{\mathcal B_{\rrs}}H^{t/2}\dx\right)^{\frac{q-2}{t}}+c\mathcal{N}_{\infty}^{2}A^{\sigma(1-\vartheta)-\alpha} \int_{\mathcal B_{\rrs}}H^{t/2}\dx\nonumber \\
        &\le \frac{c\mathcal{N}_{\infty}^{2}K^{1+\vartheta\sigma+\alpha}A^{\sigma+1}}{(\tau_2-\tau_1)^{1+\vartheta\sigma+\alpha}} \int_{\mathcal B_{\rrs}}H^{t/2}\dx\,,
\end{flalign}
for $c\equiv c(n,p,q,\alpha,\ti{L},t)$. Finally, on $ \mbox{(VI)}_{4}$ we use \eqref{fs.0} with $s\equiv (1+\vartheta\sigma+\alpha)/(\sigma-1+p)$, $m\equiv 1+\vartheta\sigma+\alpha\leq t$ (as we are in the case $p\geq 2$) to obtain
\begin{flalign}\label{caso5}
\notag  \mbox{(VI)}_{4}&\le \frac{c\mathcal{N}_{\infty}^{p}A^{\sigma(1-\vartheta)+p-2-\alpha}}{(\rrs-\rs)^{1+\vartheta\sigma+\alpha}}\int_{\mathcal B_{\rrs}}\snr{\ti{u}}^{1+\vartheta\sigma+\alpha}\dx \nonumber +c\mathcal{N}_{\infty}^{p}A^{\sigma(1-\vartheta)+p-2-\alpha} \int_{\mathcal B_{\rrs}}\snr{D\ti{u}}^{1+\vartheta\sigma+\alpha}\dx \nonumber \\ 
&\le c\mathcal{N}_{\infty}^{p}A^{\sigma(1-\vartheta)+p-2-\alpha}\left[\frac{A^{1+\vartheta\sigma+\alpha}}{(\rrs-\rs)^{1+\vartheta\sigma+\alpha}}+ \int_{\mathcal B_{\rrs}}H^{t/2}\dx\right]\nonumber \\
&\le
 \frac{c\mathcal{N}_{\infty}^{p}K^{1+\vartheta\sigma+\alpha}A^{\sigma-1+p}}{(\tau_2-\tau_1)^{1+\vartheta\sigma+\alpha}} \int_{\mathcal B_{\rrs}}H^{t/2}\dx\
\end{flalign}
with $c\equiv c(n,p,q,\alpha,\ti{L},t)$. Using \rif{caso3}-\rif{caso5} in \rif{3} and recalling the definition of $\mbox{(VI)}_{1}$ in \rif{t1.1}, all in all we have proved that 
\eqn{casso}
$$
\int_{\mathcal  B_{\ri}}\left|\frac{\tau_{h}^{2}\ti{u}}{\snr{h}^{\frac{1+\vartheta\sigma+\alpha+p}{\sigma-1+p}}}\right|^{\sigma-1+p}\dx\le  \frac{c\mathcal{N}_{\infty}^{p}K^{1+\vartheta\sigma+\alpha}A^{\sigma-1+p}}{(\tau_2-\tau_1)^{1+\vartheta\sigma+\alpha}} \int_{\mathcal B_{\rrs}}H^{t/2}\dx
$$
holds with $c\equiv c(n,p,q,\alpha,\ti{L},t)$.
\subsubsection{The case $p>n$.} We pick
\eqn{savy}
$$ 
\begin{cases}
\, \delta_{1}:=p(\alpha+1)/n-2\kk\in (0,\sigma)\vspace{1mm}\\
\, \delta_{2}:=2p/n-2\kk\in (0,\sigma)\vspace{1mm}\\
\,  \delta_{3}:=p^{2}/n-2\kk\in (0,\sigma)\,.
 \end{cases}
$$
Note that the relations in \rif{savy} follow from $p>n$ and \rif{illambda}-\rif{ts}. Note also that in the case $p>n$ we have $\gamma=p\alpha/n$ and $\kk_{0}:=(1-n/p)>0$. 
Recalling that $t\ge t_{*}$, that has been defined in \rif{defitante}$_2$, the choices in Section \ref{fm} and \rif{savy} imply
\eqn{tante}
$$
\left\{
\begin{array}{c}
\displaystyle 
\, \frac{1+\vartheta\sigma-\delta_{1}\kk_{0}}{\sigma-\delta_{1}}=1,\quad \quad\quad 1\le \frac{(\sigma-\delta_{1})t}{t-2q+p+1}\le t\\ [15pt]\displaystyle
\, \frac{1+\vartheta\sigma+\alpha-\delta_{2}\kk_{0}}{\sigma+1-\delta_{2}}=1,\quad\quad\quad  1 \leq \frac{(\sigma+1-\delta_{2})t}{t-q+2}\le t\\[15pt]\displaystyle
\, \frac{1+\vartheta\sigma+\alpha-\delta_{3}\kk_{0}}{\sigma-1+p-\delta_{3}}=1,\quad\quad 1 \leq \sigma-1+p-\delta_{3}\le t\,.
\end{array}
\right.
$$
Using \rif{gh} and \rif{tante} we then have 
\begin{flalign}
\mbox{(VI)}_{2}+\mbox{(VI)}_{3}&\le c  \mathcal{N}_{\infty}^{2}A^{\delta_{1}}\left(\int_{\mathcal B_{\rrs}}\snr{D\ti{u}}^{\frac{(\sigma-\delta_{1})t}{t-2q+p+1}}\dx\right)^{\frac{t-2q+p+1}{t}}\left(\int_{\mathcal B_{\rrs}}H^{t/2}\dx\right)^{\frac{2q-p-1}{t}}\nonumber \\
&\quad + c  \mathcal{N}_{\infty}^{2}A^{\delta_{2}}\left(\int_{\mathcal B_{\rrs}}\snr{D\ti{u}}^{\frac{(\sigma+1-\delta_{2})t}{t-q+2}}\dx\right)^{\frac{t-q+2}{t}}\left(\int_{\mathcal B_{\rrs}}H^{t/2}\dx\right)^{\frac{q-2}{t}}\nonumber \\
&\le c  \mathcal{N}_{\infty}^{2}A^{\delta_{2}}\int_{\mathcal B_{\rrs}}H^{t/2}\dx \label{caso6}
\end{flalign}
with $c\equiv c(n,p,q,\alpha,\ti{L},t)$. 
Finally, we have, again via \rif{gh}
\eqn{caso7}
$$
 \mbox{(VI)}_{4}\leq c\mathcal{N}_{\infty}^{p}A^{\delta_{3}}\int_{\mathcal B_{\rrs}}\snr{D\ti{u}}^{\sigma-1+p-\delta_{3}}\dx\le c\mathcal{N}_{\infty}^{p}A^{\delta_{3}} \int_{\mathcal B_{\rrs}}H^{t/2}\dx
$$
for $c\equiv c(n,p,q,\alpha,\ti{L},t)$. 
Using \rif{caso6}-\rif{caso7} in \rif{3}, and recalling the way $\mbox{(VI)}_{1} $ has been defined in \rif{saisai}, we get
\eqn{casso2}
$$
\int_{\mathcal  B_{\ri}}\left|\frac{\tau_{h}^{2}\ti{u}}{\snr{h}^{\frac{1+\vartheta\sigma+\alpha+p}{\sigma-1+p}}}\right|^{\sigma-1+p}\dx\le c\mathcal{N}_{\infty}^{p}A^{p^{2}/n-2\kk} \int_{\mathcal B_{\rrs}}H^{t/2}\dx
$$
holds with $c\equiv c(n,p,q,\alpha,\ti{L},t)$.
\subsubsection{The case $p\geq 2$ completed.} Collecting the estimates in \rif{casso} and \rif{casso2}  and recalling the definitions of $\sigma_1, \sigma_2$ given in \rif{ss12}-\rif{ss122}, we finally come up with
\eqn{es.1}
$$
\int_{\mathcal  B_{\ri}}\left|\frac{\tau_{h}^{2}\ti{u}}{\snr{h}^{\frac{1+\vartheta\sigma+\alpha+p}{\sigma-1+p}}}\right|^{\sigma-1+p}\dx\le \frac{c\mathcal{N}_{\infty}^{p}K^{\sigma_{2}}A^{\sigma_{1}}}{(\tau_2-\tau_1)^{\sigma_{2}}}\int_{\mathcal B_{\rrs}}H^{t/2}\dx\,,
$$
again for $c\equiv c(n,p,q,\alpha,\ti{L},t)$. This finally concludes the preliminary estimates in the case $p\geq 2$. 
\subsection{The case $1<p<2$.} With reference to \rif{defitante}, here it is $\lambda=\kk_{0}=0$ and $d=2$. With $\theta$ as in \rif{ts.1} we have
    \begin{flalign*}\label{9}
    \mbox{(VIII)}&:=\frac{1}{\snr{h}^{\theta}}\int_{\mathcal B_{1}}\eta^{d}\snr{\tau_{h}\ti{u}}^{\sigma-1}\snr{\tau_{h}D\ti{u}}^{p}\dx\nonumber \\ \nonumber 
      & = \int_{\mathcal B_{1}}\left(\frac{\eta^{d}\snr{\tau_{h}\ti{u}}^{\sigma-1}}{\snr{h}^{\sigma-1-2\gamma+2\kk}}\right)^{\frac{2-p}{2}}\mathcal{D}(h)^{\frac{(2-p)p}{4}}\left(\frac{\eta^{d}\snr{\tau_{h}\ti{u}}^{\sigma-1}}{\snr{h}^{1+\vartheta\sigma+\alpha}}\right)^{\frac{p}{2}}\mathcal{D}(h)^{\frac{(p-2)p}{4}}\snr{\tau_{h}D\ti{u}}^{p}\dx\,.
    \end{flalign*}
   Then, recalling the definition of $\mbox{(V)}$ in \rif{notache} and using Young's inequality, we obtain
$$
    \mbox{(VIII)} \le \mbox{(V)}+
    cA^{2(\gamma-\kk)}\int_{\mathcal B_{1}}\eta^{d}\mathcal{D}(h)^{p/2}\left|\frac{\tau_{h}\ti{u}}{\snr{h}}\right|^{\sigma-1-2\gamma+2\kk}\dx\,.
$$
To proceed, note that 
$$
1\leq \frac{t(\sigma-1-2\gamma+2\kk)}{t-p}\leq t
$$
by \rif{defitante}$_2$ and \rif{foll}$_1$, so that, thanks to \rif{gh}, we have 
   \begin{flalign*}
\notag & A^{2(\gamma-\kk)}\int_{\mathcal B_{1}}\eta^{d}\mathcal{D}(h)^{p/2}\left|\frac{\tau_{h}\ti{u}}{\snr{h}}\right|^{\sigma-1-2\gamma+2\kk}\dx\\ 
\notag &\quad \le cA^{2(\gamma-\kk)}\left(\int_{\mathcal B_{\rs}}\left|\frac{\tau_{h}\ti{u}}{\snr{h}}\right|^{\frac{t(\sigma-1-2\gamma+2\kk)}{t-p}}\dx\right)^{1-p/t}\left(\int_{\mathcal B_{\rs}}\mathcal{D}(h)^{t/2}\dx\right)^{p/t}\\
&\quad \le cA^{2(\gamma-\kk)}\left(\int_{\mathcal B_{\rrs}}\snr{D\ti{u}}^{\frac{t(\sigma-1-2\gamma+2\kk)}{t-p}}\dx\right)^{1-p/t}\left(\int_{\mathcal B_{\rrs}}H^{t/2}\dx\right)^{p/t}\le  c\mbox{(VI)}_{1}
   \end{flalign*}
    with $c\equiv c(n,p,q,\alpha, t)$ and where $\mbox{(VI)}_{1}$ is defined in \rif{t1.1}. Connecting the content of the last three displays and recalling \rif{siparte}, we arrive at 
\eqn{9.1}
   $$
    \mbox{(VIII)} +\mbox{(V)}  \leq c\mbox{(VI)}_{1}   +c |\mbox{(II)}+\mbox{(IV)}|
   $$
   with $c\equiv c(n,p,q,\alpha,\ti{L}, t)$. 
In order to estimate $\mbox{(VIII)}$ from below we proceed as in \eqref{11.5}-\eqref{2}. Indeed, using \rif{prod}-\rif{gh} we find 
   \begin{flalign*}
\mbox{(VIII)}&\ge\frac{1}{c\snr{h}^{\theta}}\int_{\mathcal B_{1}}\left|D\left(\eta^{d/p}\snr{\tau_{h}\ti{u}}^{(\sigma-1)/p}\tau_{h}\ti{u}\right)\right|^{p}\dx-c\int_{\mathcal B_{1}}\snr{D\eta}^{p}\left|\frac{\tau_{h}\ti{u}}{\snr{h}^{\frac{\theta}{\sigma-1+p}}}\right|^{\sigma-1+p}\dx\nonumber \\
&\geq \frac{1}{c\snr{h}^{\theta+p}}\int_{\mathcal B_{\rs}}\left|\tau_{h}\left(\eta^{d/p}\snr{\tau_{h}\ti{u}}^{(\sigma-1)/p}\tau_{h}\ti{u}\right)\right|^{p}\dx-c\int_{\mathcal B_{\rs}}\snr{D\eta}^{p}\left|\frac{\tau_{h}\ti{u}}{\snr{h}^{\frac{\theta}{\sigma-1+p}}}\right|^{\sigma-1+p}\dx\nonumber \\
&\ge\frac{1}{c\snr{h}^{\theta+ p}}\int_{\mathcal B_{\rs}}\eta(x+h)^{d}\left|\tau_{h}\left(\snr{\tau_{h}\ti{u}}^{(\sigma-1)/p}\tau_{h}\ti{u}\right)\right|^{p}\dx\nonumber \\
&\quad -c\int_{\mathcal B_{\rs}}\snr{D\eta}^{p}\left|\frac{\tau_{h}\ti{u}}{\snr{h}^{\frac{\theta}{\sigma-1+p}}}\right|^{\sigma-1+p}\dx-\frac{1}{\snr{h}^{\theta+p}}\int_{\mathcal B_{\rs}}\snr{\tau_{h}\eta^{d/p}}^{p}\snr{\tau_{h}\ti{u}}^{\sigma-1+p}\dx\nonumber \\
&\ge\frac{1}{c\snr{h}^{\theta+p}}\int_{\mathcal  B_{\ri}}\left|\tau_{h}\left(\snr{\tau_{h}\ti{u}}^{(\sigma-1)/p}\tau_{h}\ti{u}\right)\right|^{p}\dx-c\mathcal{N}_{\infty}^{p}\int_{\mathcal B_{\rs}}\left|\frac{\tau_{h}\ti{u}}{\snr{h}^{\frac{\theta}{\sigma-1+p}}}\right|^{\sigma-1+p}\dx
\end{flalign*}
so that, using \rif{mon}$_1$, we come to 
\eqn{cometo3} 
$$
\mbox{(VIII)} \ge \underbrace{\frac 1c\int_{\mathcal  B_{\ri}}\left|\frac{\tau_{h}^{2}\ti{u}}{\snr{h}^{\frac{\theta+p}{\sigma-1+p}}}\right|^{\sigma-1+p}\dx}_{\displaystyle =:\mbox{(IX)}/c }-\underbrace{c\mathcal{N}_{\infty}^{p}\int_{\mathcal B_{\rs}}\left|\frac{\tau_{h}\ti{u}}{\snr{h}^{\frac{\theta}{\sigma-1+p}}}\right|^{\sigma-1+p}\dx}_{\displaystyle =:c\mbox{(VIII)}_{1} }\,.
$$  
Using \rif{cometo3} in \rif{9.1} yields
\eqn{cometo}
$$
  \mbox{(IX)}+\mbox{(V)} \leq c \mbox{(VI)}_1+c   \mbox{(VIII)}_1 
+c |\mbox{(II)}+\mbox{(IV)}| $$ 
with $c\equiv c(n,p,q,\alpha, \ti{L},t)$ and we proceed estimating the terms on the right-hand side. We define, for $x \in \mathcal B_{2-\snr{h}}$
\eqn{defiG} 
$$
 \hat{\mathcal{G}}(h)\equiv \hat{\mathcal{G}}(x,h):=\int_{0}^{1}\partial_{z} \ti{F}(x+\ttl  h,D\ti{u}(x+\ttl  h))\d\ttl \,.
$$ 
As a consequence of \rif{1d}$_1$ it follows
\eqn{gghh}
$$
 \snr{\hat{\mathcal{G}}(x,h)} \lesssim_{n,p,q,\ti{L}} \int_{0}^{1}\left(\mathcal{D}(x,\ttl  h )^{(q-1)/2}+
1\right)\d\ttl$$
where $\mathcal{D}(\cdot)$ is defined in \rif{eett}. Then we have, using \eqref{af} (recall that $d=2$)
    \begin{flalign}
    \snr{\mbox{(II)}+\mbox{(IV)}}
    &=\frac{d}{\snr{h}^{1+\vartheta\sigma+\alpha}}\left|\int_{\mathcal B_{1}} \tau_{h}[\partial_{z} \ti{F}(\cdot,D\ti{u})]\cdot \snr{\tau_{h}\ti{u}}^{\sigma-1}\tau_{h}\ti{u} \eta D\eta\dx\right|\nonumber \\
    &=\frac{d}{|h|^{\vartheta\sigma+\alpha}}\left|\int_{\mathcal B_{1}} \hat{\mathcal{G}}(h)\cdot \partial_{h/\snr{h}}\left(\snr{\tau_{h}\ti{u}}^{\sigma-1}\tau_{h}\ti{u}\eta D\eta\right)\dx\right| \notag\\
    &\leq  \underbrace{\frac{d\sigma }{\snr{h}^{\vartheta\sigma+\alpha}}\int_{\mathcal B_{1}}\snr{\hat{\mathcal{G}}(h)}\eta\snr{D\eta}\snr{\tau_{h}\ti{u}}^{\sigma-1}\snr{\tau_{h}D\ti{u}}\dx}_{\displaystyle =: \mbox{(X)}} \nonumber \\
    &\quad + \underbrace{\frac{d}{\snr{h}^{\vartheta\sigma+\alpha}}\int_{\mathcal B_{1}}\snr{\hat{\mathcal{G}}(h)}\left(\snr{D\eta}^{2}+\eta\snr{D^{2}\eta}\right)\snr{\tau_{h}\ti{u}}^{\sigma}\dx}_{\displaystyle =: \mbox{(XI)}}\,.\label{rep00}
        \end{flalign}
In the application of Lemma \ref{perparti} in the last display, we have used \rif{riscalamento}$_2$ and that  $\eta\le \mathds{1}_{\mathcal B_{\rs}}$, see \rif{vettori}. For the next computations, note that Jensen's inequality and Fubini's Theorem imply
\begin{flalign}
\notag\int_{\mathcal B_{\rs}}\left[\int_{0}^{1}\left( \mathcal{D}(\ttl h)^{m}+1\right)\d\ttl\right]^{\texttt{v}} \dx& \leq 2^{\texttt{v}-1}  \int_{\mathcal B_{\rs}}\int_{0}^{1}\left( \mathcal{D}(\ttl h)^{\texttt{v} m}+1\right)\d\ttl \dx\\
& \leq c(\texttt{v}, m)
\int_{\mathcal B_{\rrs}} H^{\texttt{v} m}\dx \,, \qquad \mbox{ for every $\texttt{v}\geq 1$ and $m \geq 0$} \,.\label{rep2}
\end{flalign}  
Using first \rif{gghh} in \rif{rep00} and then \rif{rep2} repeatedly, we find
            \begin{flalign*}
    \mbox{(X)} &\le\eps\mbox{(V)}+c_{\eps}\int_{\mathcal B_{1}}\left(\int_{0}^{1}\left(\mathcal{D}(\ttl h)^{q-1}+1\right)\d\ttl\right)\, \mathcal{D}(h)^{\frac{2-p}{2}}\left|\frac{\tau_{h}\ti{u}}{\snr{h}^{\frac{\vartheta\sigma+\alpha-1}{\sigma-1}}}\right|^{\sigma-1}\snr{D\eta}^{2}\dx\\& \le\eps\mbox{(V)}+c_{\eps}\mathcal{N}_{\infty}^{2}\int_{\mathcal B_{\rs}}\left[\int_{0}^{1}\left(\mathcal{D}(\ttl h)^{q-p/2}+1\right)\d\ttl+ \mathcal{D}( h)^{q-p/2}\right]\left|\frac{\tau_{h}\ti{u}}{\snr{h}^{\frac{\vartheta\sigma+\alpha-1}{\sigma-1}}}\right|^{\sigma-1}\dx\nonumber \\ &
\le \eps\mbox{(V)}+\underbrace{c_{\eps}\mathcal{N}_{\infty}^{2}\left(\int_{\mathcal B_{\rs}}\left|\frac{\tau_{h}\ti{u}}{\snr{h}^{\frac{\vartheta\sigma+\alpha-1}{\sigma-1}}}\right|^{\frac{t(\sigma-1)}{t-2q+p}}\dx\right)^{\frac{t-2q+p}{t}}\left(\int_{\mathcal B_{\rrs}}H^{t/2}\dx\right)^{\frac{2q-p}{t}}}_{\displaystyle =: c_{\eps}\mbox{(VIII)}_{2}}
            \end{flalign*}
for every $\eps \in (0,1)$ and $c_{\eps}$ depends on $n,p,q,\alpha,\ti{L},t,\eps$. Note that we have used Young's inequality with conjugate exponents 
$$
\left(\frac{2q-p}{2-p}, \frac{2q-p}{2(q-1)}\right)\,.
$$
By similar means we also find
       \begin{flalign*}
    \mbox{(XI)} &  \leq c  \mathcal{N}_{\infty}^{2}\int_{\mathcal B_{\rs}}\snr{\hat{\mathcal{G}}(h)}\left|\frac{\tau_{h}\ti{u}}{\snr{h}^{\frac{\vartheta\sigma+\alpha}{\sigma}}}\right|^{\sigma}\dx \\
    & \leq c\mathcal{N}_{\infty}^{2}\int_{\mathcal B_{\rs}}\left(\int_{0}^{1}\left(\mathcal{D}(\ttl h)^{(q-1)/2}+1\right)\d\ttl\right) \left|\frac{\tau_{h}\ti{u}}{\snr{h}^{\frac{\vartheta\sigma+\alpha}{\sigma}}}\right|^{\sigma}\dx\\
    & \leq \underbrace{c \mathcal{N}_{\infty}^{2}\left(\int_{\mathcal B_{\rs}}\left|\frac{\tau_{h}\ti{u}}{\snr{h}^{\frac{\vartheta\sigma+\alpha}{\sigma}}}\right|^{\frac{t\sigma}{t-q+1}}\dx\right)^{\frac{t-q+1}{t}}\left(\int_{\mathcal B_{\rrs}}H^{t/2}\dx\right)^{\frac{q-1}{t}}}_{\displaystyle =: c \mbox{(VIII)}_{3}}\,.
    \end{flalign*}
We therefore conclude with
$$
 \snr{\mbox{(II)}+\mbox{(IV)}} \leq  c\eps\mbox{(V)}+c_{\eps}\mbox{(VIII)}_{2}+c\mbox{(VIII)}_{3} \,.
$$
Using this last inequality in \rif{cometo}, choosing $\eps \equiv \eps(n,p,q,\alpha,\ti{L},t)$ small enough in order to reabsorb $c\eps \mbox{(V)}$, we conclude with 
    \eqn{cometo2}
$$
  \mbox{(IX)} \leq c\mbox{(VI)}_{1}  + c   \mbox{(VIII)}_{1}  + c   \mbox{(VIII)}_{2}  + c   \mbox{(VIII)}_{3} \,.
$$ 
The terms in the right-hand side of \rif{cometo2} will be estimated via Lemma \ref{ls} and for this, recalling \rif{ts.1}, we note that 
\eqn{foll2so}
$$
\left\{
\begin{array}{c}
\displaystyle 
\,  \max\left\{ \frac{\theta}{\sigma-1+p}, \frac{\vartheta\sigma+\alpha-1}{\sigma-1},\frac{\vartheta\sigma+\alpha}{\sigma}\right\}<1\\ [15pt]\displaystyle
\, \min\left\{ \frac{\theta}{\sigma-1+p}, \frac{\vartheta\sigma+\alpha-1}{\sigma-1},\frac{\vartheta\sigma+\alpha}{\sigma}\right\}>0 \\ [15pt]\displaystyle
\, \max\left\{\theta, \frac{t(\vartheta\sigma+\alpha-1)}{t-2q+p},\frac{t(\vartheta\sigma+\alpha)}{t-q+1}\right\}\le t \\ [15pt]\displaystyle
\, \min\left\{\theta, \frac{t(\vartheta\sigma+\alpha-1)}{t-2q+p},\frac{t(\vartheta\sigma+\alpha)}{t-q+1}\right\}> 1\,.
\end{array}
\right.
$$
By \rif{foll2so} we apply \eqref{fs.0} with $s\equiv \theta/(\sigma-1+p)$ and $m\equiv \theta$ to get
    \begin{flalign*}
\mbox{(VIII)}_{1}&\le\frac{c\mathcal{N}_{\infty}^{p}A^{\sigma-1+p-\theta}}{(\rrs-\rs)^{\theta}}\int_{\mathcal B_{\rrs}}\snr{\ti{u}}^{\theta}\dx +c\mathcal{N}_{\infty}^{p}A^{\sigma-1+p-\theta} \int_{\mathcal B_{\rrs}}\snr{D\ti{u}}^{\theta}\dx \nonumber \\
&\le c\mathcal{N}_{\infty}^{p}A^{\sigma-1+p-\theta}\left[\frac{A^{\theta}}{(\rrs-\rs)^{\theta}}+\left(\int_{\mathcal B_{\rrs}}H^{t/2}\dx\right)^{\theta/t}\right] \\
& \leq \frac{c\mathcal{N}_{\infty}^{2}K^{1+\vartheta\sigma+\alpha}A^{\sigma-1+p}}{(\tau_2-\tau_1)^{1+\vartheta\sigma+\alpha}}\int_{\mathcal B_{\rrs}}H^{t/2}\dx,
    \end{flalign*} 
    for $c\equiv c(n,p,q,\alpha,\ti{L},t)$. Note that we have used that $\theta < 1+\vartheta\sigma+\alpha$ as here it is $p<2$ and that in the last line we also employed \rif{iraggiK}. For $\mbox{(VIII)}_{2}$ it is $s\equiv (\vartheta\sigma+\alpha-1)/(\sigma-1)\in (0,1)$ and $m\equiv t(\vartheta\sigma+\alpha-1)/(t-2q+p)>1$ so that \rif{fs.0} gives 
    \begin{flalign*}
    \mbox{(VIII)}_{2} &\le\frac{c\mathcal{N}_{\infty}^{2}A^{\sigma(1-\vartheta)-\alpha}}{(\rrs-\rs)^{\vartheta\sigma+\alpha-1}}\left(\int_{\mathcal B_{\rrs}}\snr{\ti{u}}^{\frac{t(\vartheta\sigma+\alpha-1)}{t-2q+p}}\dx\right)^{\frac{t-2q+p}{t}}\left(\int_{\mathcal B_{\rrs}}H^{t/2}\dx\right)^{\frac{2q-p}{t}}\nonumber \\
    &\quad +c\mathcal{N}_{\infty}^{2}A^{\sigma(1-\vartheta)-\alpha}\left(\int_{\mathcal B_{\rrs}}\snr{D\ti{u}}^{\frac{t(\vartheta\sigma+\alpha-1)}{t-2q+p}}\dx\right)^{\frac{t-2q+p}{t}}\left(\int_{\mathcal B_{\rrs}}H^{t/2}\dx\right)^{\frac{2q-p}{t}}\nonumber \\
&\le \frac{c\mathcal{N}_{\infty}^{2}A^{\sigma-1}}{(\rrs-\rs)^{\vartheta\sigma+\alpha-1}}\left(\int_{\mathcal B_{\rrs}}H^{t/2}\dx\right)^{\frac{2q-p}{t}}\nonumber+c\mathcal{N}_{\infty}^{2}A^{\sigma(1-\vartheta)-\alpha}\int_{\mathcal B_{\rrs}}H^{t/2}\dx
 \\ 
&\le \frac{c\mathcal{N}_{\infty}^{2}K^{\vartheta\sigma+\alpha-1}A^{\sigma-1}}{(\tau_2-\tau_1)^{\vartheta\sigma+\alpha-1}}\int_{\mathcal B_{\rrs}}H^{t/2}\dx
    \end{flalign*}
    for $c\equiv c(n,p,q,\alpha,\ti{L},t)$. For $ \mbox{(VIII)}_{3}$ we take $s \equiv (\vartheta\sigma+\alpha)/\sigma$ and $m \equiv t(\vartheta\sigma+\alpha)/(t-q+1)$, obtaining 
    \begin{flalign*}
   \mbox{(VIII)}_{3}&\le \frac{c\mathcal{N}_{\infty}^{2}A^{\sigma(1-\vartheta)-\alpha}}{(\rrs-\rs)^{\vartheta\sigma+\alpha}}\left(\int_{\mathcal B_{\rrs}}\snr{\ti{u}}^{\frac{t(\vartheta\sigma+\alpha)}{t-q+1}}\dx\right)^{\frac{t-q+1}{t}}\left(\int_{\mathcal B_{\rrs}}H^{t/2}\dx\right)^{\frac{q-1}{t}}\nonumber \\
    &\quad +c\mathcal{N}_{\infty}^{2}A^{\sigma(1-\vartheta)-\alpha}\left(\int_{\mathcal B_{\rrs}}\snr{D\ti{u}}^{\frac{t(\vartheta\sigma+\alpha)}{t-q+1}}\dx\right)^{\frac{t-q+1}{t}}\left(\int_{\mathcal B_{\rrs}}H^{t/2}\dx\right)^{\frac{q-1}{t}}\nonumber \\
    &\le \frac{c\mathcal{N}_{\infty}^{2}A^{\sigma}}{(\rrs-\rs)^{\vartheta\sigma+\alpha}}\left(\int_{\mathcal B_{\rrs}}H^{t/2}\dx\right)^{\frac{q-1}{t}}\nonumber+c\mathcal{N}_{\infty}^{2}A^{\sigma(1-\vartheta)-\alpha}\int_{\mathcal B_{\rrs}}H^{t/2}\dx
 \\
&\le \frac{c\mathcal{N}_{\infty}^{2}K^{\vartheta\sigma+\alpha}A^{\sigma}}{(\tau_2-\tau_1)^{\vartheta\sigma+\alpha}}\int_{\mathcal B_{\rrs}}H^{t/2}\dx
    \end{flalign*}
    with $c\equiv c(n,p,q,\alpha,\ti{L},t)$. Using the estimates in the last three displays in \rif{cometo2}, recalling the definitions of $\mbox{(VI)}_{1}$ in \rif{t1.1}, and of $\mbox{(IX)}$ in \rif{cometo3} and finally \rif{ss12}-\rif{ss122}, we arrive at
    \eqn{es.2}
    $$ 
\int_{\mathcal  B_{\ri}}\left|\frac{\tau_{h}^{2}\ti{u}}{\snr{h}^{\frac{\theta+p}{\sigma-1+p}}}\right|^{\sigma-1+p}\dx\le \frac{c\mathcal{N}_{\infty}^{2}K^{\sigma_{2}}A^{\sigma_{1}}}{(\tau_2-\tau_1)^{\sigma_{2}}}\int_{\mathcal B_{\rrs}}H^{t/2}\dx
    $$
    for $c\equiv c(n,p,q,\alpha,\ti{L},t)$.

    \subsection{Proof of Proposition \ref{tfs} concluded}\label{ultima} Estimates \eqref{es.1} and \eqref{es.2} unify in
        \eqn{es}
        $$
   \sup_{0<\snr{h}<\mathcal{h}_{0}} \int_{\mathcal  B_{\ri}}\left|\frac{\tau_{h}^{2}\ti{u}}{\snr{h}^{\beta_{0}(t)}}\right|^{\sigma(t)-1+p}\dx\le \frac{c\mathcal{N}_{\infty}^{d}K^{\sigma_{2}(t)}A^{\sigma_{1}(t)}}{(\tau_2-\tau_1)^{\sigma_{2}(t)}}\int_{\mathcal B_{\rrs}}(\snr{D\ti{u}}^{2}+1)^{t/2}\dx\,,
    $$
    where 
        $$
   \beta_{0}(t):=\left\{
\begin{array}{c}
\displaystyle 
\, \frac{1+\vartheta(t)\sigma(t)+\alpha+p}{\sigma(t)-1+p}\qquad \mbox{if} \ \ 2\leq p \\[16pt]\displaystyle
\, \frac{\theta(t)+p}{\sigma(t)-1+p}\qquad \mbox{if} \ \ 1<p< 2\,,
\end{array}
\right.
    $$ $c\equiv c(n,p,q,\alpha,\ti{L},t)$ and $\sigma, \vartheta, \theta, \sigma_1, \sigma_2$ are in \rif{ts}-\rif{ss122}, $t\geq t_*$ and $t_*$ as in \rif{defitante}$_2$. Note that here we prefer to emphasize the dependence on $t$ of the various quantities. By using the definitions in \rif{defitante}-\rif{ts.1}, and recalling how $\gamma$ has been defined in \rif{pqb}, 
we check that
\eqn{cecchi}
$$
1< \beta_{0}(t) <2\quad \mbox{holds for every $t\geq t_*$}\,.$$ 
Note that $ \beta_{0}(t) \to 1$ as $t\to \infty$. We also define the positive number $\mathfrak{g}$ via
$$0 \stackrel{\eqref{pqb},\eqref{illambda}}{<}\mathfrak{g}:= p+\gamma-q-\kk \Longrightarrow \sigma(t)-1+p=t+2\mathfrak{g}>t\,.
$$  By \rif{iraggiK}-\rif{vettori} we have  $ \rri  + 7 \mathcal h_0<  \ri$ so that \rif{cecchi} allows to use Lemma \ref{bllema} with $m\equiv t+2\mathfrak{g}$, $s \equiv \beta_0(t)$ and this yields
    \begin{eqnarray}
   \notag \nr{D\ti{u}}_{L^{t+2\mathfrak{g}}(\mathcal B_{\rri})}^{t+2\mathfrak{g}}
    &\stackleq{immersione2}& c\sup_{0<\snr{h}<\mathcal{h}_{0}}\left\|\frac{\tau_{h}^{2}\ti{u}}{\snr{h}^{\beta_{0}(t)}}\right\|_{L^{t+2\mathfrak{g}}(\mathcal  B_{\ri})}^{t+2\mathfrak{g}} +\frac{c}{\mathcal{h}_{0}^{\beta_{0}(t)(t+2\mathfrak{g})}}\nr{\ti{u}}_{L^{t+2\mathfrak{g}}(\mathcal  B_{\ri})}^{t+2\mathfrak{g}}\nonumber \\
    &\stackleq{es} & \frac{cK^{\sigma_{2}(t)+d}A^{\sigma_{1}(t)}}{(\tau_{2}-\tau_{1})^{\sigma_{2}(t)+d}}\left(\nr{D\ti{u}}_{L^{t}(\mathcal B_{\rrs})}^{t}+1\right)+\frac{cK^{\beta_{0}(t)(t+2\mathfrak{g})}A^{t+2\mathfrak{g}}}{(\tau_{2}-\tau_{1})^{\beta_{0}(t)(t+2\mathfrak{g})}} \label{numerology}
    \end{eqnarray}
      where $c\equiv c (n,p,q,\alpha, \ti{L},t)$. Note that we have used \rif{eett}$_4$ too. 
Setting 
    \eqn{byseting}
      $$
m_{t}:=\max\left\{\sigma_{1}(t),\sigma_{2}(t)+d, \beta_{0}(t)(t+2\mathfrak{g})\right\} \,.
  $$
  With easy manipulations, we conclude with 
  \eqn{1-0}
  $$
\nr{D\ti{u}}_{L^{t+2\mathfrak{g}}(\mathcal B_{\rri})}^{t+2\mathfrak{g}}  +1\le\frac{c_tK^{m_{t}}A^{m_{t}}\left(\nr{D\ti{u}}_{L^{t}(\mathcal B_{\rrs})}^{t}+1\right)}{(\tau_{2}-\tau_{1})^{m_{t}}}
  $$ 
  where $c_t$ depends on $n,p,q,\alpha,\ti{L},t$. We next introduce the sequence of numbers $\{t_{i}\}_{i\in \mathbb{N}\cup \{0\}}$ as
  \eqn{induttiva}
  $$  
  \begin{cases}
  \,   t_{i+1}:=t_{i}+2\mathfrak{g}\,, \quad  i \geq 0\vspace{0.5mm}\\
  \, t_{0}:=t_{*}
  \end{cases} \quad \Longleftrightarrow \quad  t_{i}=t_{*}+2i\mathfrak{g}\,, \qquad \forall \ i\geq 0\,.
  $$
 Recall that $t_*$ is defined in \rif{defitante}$_2$. Moreover, we take the sequence of shrinking balls \linebreak $\{\mathcal B_{\rr_{i}}\}_{i\in \N\cup\{0\}}$ where
  $$\rr_{i}:=\tau_{1}+\frac{\tau_{2}-\tau_{1}}{2^{i+1}}\,, \quad i \geq 0\,.$$
  In particular we have $\rr_{0} = \tau_{1}+ (\tau_{2}-\tau_{1})/2$. Recalling \rif{byseting}, we  apply \rif{1-0} with the choice - in \rif{raggimain} -  given by 
$$\begin{cases}
\, 
\rri\equiv \rri(i):= \rr_{i+1}, \quad \rrs\equiv \rrs(i)\equiv \rr_{i}, \quad
K\equiv K_{i}:=2^{i+1}, \quad  t\equiv t_{i}\vspace{1mm}\\
\,  m_{i}:=\max\left\{\sigma_{1}(t_{i}), \sigma_{2}(t_{i})+d,\beta_{0}(t_{i}) (t_{i}+2\mathfrak{g})\right\}\end{cases}$$
for integers $i\geq 0$. Therefore \rif{1-0} becomes
\eqn{iterante}
$$
   \mathcal D_{t_{i+1}}(\mathcal B_{\rr_{i+1}}) \leq \frac{c_{i}A^{m_{i}}\mathcal D_{t_{i}}(\mathcal B_{\rr_{i}})}{(\tau_{2}-\tau_{1})^{m_{i}}},\qquad \forall \ i \geq 0\,,
$$
with $c_{i}\equiv c_{i}(n,p,q,\alpha,\ti{L},i)$, where
    $$
    \mathcal D_m(\mathcal B_{\tau}) :=  \nr{D\ti{u}}_{L^{m}(\mathcal B_{\tau})}^{m} +1\,,\qquad 0 < \tau \leq 1\,, \ \  m \geq 1\,.
    $$
Iterating \rif{iterante} yields 
    \begin{flalign}\label{14}
   \mathcal D_{t_{i+1}}(\mathcal B_{\rr_{i+1}})\le \left(\prod_{j=0}^{i}\frac{c_{j}A^{m_{j}}}{(\tau_{2}-\tau_{1})^{m_{j}}}   \right) \mathcal D_{t_{0}}(\mathcal B_{\rr_{0}}),\qquad \forall \ i \geq 0,    \end{flalign}
    with $c_{j}\equiv c_{j}(n,p,q,\alpha,\ti{L},j)$, $j\in \{0,\cdots,i\}$. Definition \rif{induttiva} implies $t_{i+1}>t_{0}=t_{*}>p$, therefore we can use the interpolation inequality
    \eqn{inter}
    $$
\nr{D\ti{u}}_{L^{t_{0}}(\mathcal B_{\rr_{0}})}\le  \nr{D\ti{u}}_{L^{t_{i+1}}(\mathcal B_{\rr_{0}})}^{\lambda_{i+1}}\nr{D\ti{u}}_{L^{p}(\mathcal B_{\rr_{0}})}^{1-\lambda_{i+1}}\,,
$$
where $\lambda_{i+1}\in (0,1)$ is such that
$$ 
\frac{1}{t_{0}}=\frac{\lambda_{i+1}}{t_{i+1}}+\frac{1-\lambda_{i+1}}{p} \ \Longrightarrow \ \lambda_{i+1}=\frac{t_{i+1}(t_{0}-p)}{t_{0}(t_{i+1}-p)}\,.
$$
Inequality \rif{inter} and $\rr_{0}\le \tau_{2}\leq r_{2}$ easily imply
$$
   \mathcal D_{t_{0}}(\mathcal B_{\rr_{0}}) \leq    2\mathcal D_{t_{i+1}}(\mathcal B_{\tau_{2}})^{\frac{t_0\lambda_{i+1}}{t_{i+1}}}
      \mathcal D_{p}(\mathcal B_{r_{2}})^{\frac{t_0(1-\lambda_{i+1})}{p}}\,.
$$
Using the content of the last two displays in \rif{14}, 
and recalling that $\rr_{i+1}\ge \tau_{1}$, we obtain
$$
\mathcal D_{t_{i+1}}(\mathcal B_{\tau_{1}})\le 2 \left(\prod_{j=0}^{i}\frac{c_{j}A^{m_{j}}}{(\tau_{2}-\tau_{1})^{m_{j}}}\right) \mathcal D_{t_{i+1}}(\mathcal B_{\tau_{2}})^{\frac{t_{0}-p}{t_{i+1}-p}}
      \mathcal D_{p}(\mathcal B_{r_{2}})^{\frac{t_{i+1}-t_{0}}{t_{i+1}-p}}\,.
$$
Young's inequality now yields
\begin{flalign*}
\mathcal D_{t_{i+1}}(\mathcal B_{\tau_{1}}) & \le \frac{1}{2}\mathcal D_{t_{i+1}}(\mathcal B_{\tau_{2}})+2^{\frac{t_{i+1}-p}{t_{i+1}-t_{0}}}\left(\prod_{j=0}^{i}\frac{c_{j}A^{m_{j}}}{(\tau_{2}-\tau_{1})^{m_{j}}}\right)^{\frac{t_{i+1}-p}{t_{i+1}-t_{0}}}  \mathcal D_{p}(\mathcal B_{\tau_{2}})\\
& \le \frac{1}{2}\mathcal D_{t_{i+1}}(\mathcal B_{\tau_{2}})+ \frac{cA^{\mf{b}_i}\mathcal D_{p}(\mathcal B_{r_{2}})}{(\tau_{2}-\tau_{1})^{\mf{b}_i}}\,,
\end{flalign*}
where $c\equiv c(n,p,q,\alpha,\ti{L},i)$ and $\mf{b}_i\equiv \mf{b}_i(n,p,q,\alpha,i)$. 
Lemma \ref{iterlem} applied in the above inequality with the choice $\hhh(\tau)\equiv \mathcal D_{t_{i+1}}(\mathcal B_{\tau})$, finally yields
\eqn{iii}
$$
\mathcal D_{t_{i+1}}(\mathcal B_{r_{1}})\le \frac{cA^{\mf{b}_{i}}\mathcal D_{p}(\mathcal B_{r_{2}})}{(r_{2}-r_{1})^{\mf{b}_{i}}}
$$
for a different $c\equiv c(n,p,q,\alpha,\ti{L},i)$. This leads to \rif{15} taking $i\equiv i(\qqq)$ large enough to have $t_{i+1} \geq\qqq$ in \rif{iii}. The proof of Proposition \ref{tfs} and therefore of Proposition \ref{scat} are complete. 
\vspace{2mm}

Following the proofs of Propositions \ref{scat}-\ref{tfs} and rescaling in \rif{15}, we easily get 
 \begin{proposition}\label{scatdopo}
Let $u\in W^{1,q}_{\loc}(\Omega)\cap L^{\infty}_{\loc}(\Omega)$ be a minimizer of the functional  $\mathcal F$ in  \eqref{ggg2}, under assumptions \eqref{pqbb} and \eqref{assfr} with $p\leq n$. Then 
\eqn{hhhbb}
$$
\int_{B_{\rr/2}}\snr{Du}^{\qqq}\dx \le  c\left(\frac{\nr{u}_{L^{\infty}(\rm B)}}{\rr}+1\right)^{\mf{b}_{\qqq}}\left(\int_{B_{\rr}}\snr{Du}^{p}\dx +\rr^n\right)
$$
holds for every ball $B_{\rr}\Subset \Omega$, $\rr\leq 1$, where $c\equiv c(n,p,q,\alpha,\ti{L},\qqq)\geq 1$ and $\mf{b}_{\qqq}\equiv \mf{b}_{\qqq}(n,p,q,\alpha, \qqq)\geq 1$
\end{proposition}
\section{$\App{\rm B}$ and proof of Theorem \ref{main3}}\label{ptm}
\subsection{Approximation scheme on balls ${\rm B}\Subset \Omega$}\label{appisec}
Here we devise a procedure, that we call $\App{\rm B}$, aimed at approximating a minimizer $u\in W^{1,p}_{\loc}(\Omega)$ of the relaxed functional $\bar{\mathcal F}$, with minimizers of integral functionals as in \rif{ggg2}, with smooth integrands $F(\cdot)$ satisfying \rif{assfr}, on a fixed ball ${\rm B}\Subset \Omega$. It enhances a few techniques from \cite{piovra}. Needless to say, we recall that whenever we are going to $\App{\rm B}$, we shall always assume that $F(\cdot)$ verifies \rif{assf} and \rif{assf2} as described in Section \ref{addhol}; in particular, $F(\cdot)$ is continuous. In the following we denote by $\eps\equiv \{\eps_k\}$ and $\delta\equiv \{\delta_m\}$ two strictly decreasing sequences of positive numbers such that $(0,1)\ni \eps, \delta \to 0$ and $\delta \leq \dist({\rm B}, \partial \Omega)/2$; we shall sometimes extract subsequences still denoting $\eps, \delta$ the corresponding outcome. Moreover, $\texttt{o}(\varepsilon,{\rm B})$ and $\texttt{o}_{\varepsilon}(\delta,{\rm B})$ will denote two general quantities, depending also on ${\rm B}$, such that $\texttt{o}(\varepsilon,{\rm B})\to 0$ when $\eps \to 0$ and $\texttt{o}_{\varepsilon}(\delta,{\rm B})\to 0$ when $\delta \to 0$, respectively. The second convergence actually occurs for each fixed $\eps$,  but not necessarily uniformly with respect to $\eps$. The agreement is that $\texttt{o}(\varepsilon,{\rm B})$ is independent of $\delta$, while, as done for the generic constants denoted by  $c$, the quantities $\texttt{o}(\varepsilon,{\rm B}), \texttt{o}_{\varepsilon}(\delta,{\rm B})$ might change in different occurrences, but still keeping the same meaning in terms of $\eps, \delta$. Observe that, when appearing on the right-hand side of an estimate, we shall assume that $\texttt{o}(\varepsilon,{\rm B}), \texttt{o}_{\varepsilon}(\delta,{\rm B})$ are non-negative, as, in fact, taking their absolute value does not alter their infinitesimal nature. By Lemma \ref{bf.0} we pick a sequence $\{\ti{u}_{\varepsilon}\}\subset W^{1,q}({\rm B})$ such that
\eqn{x.1.1}
$$ 
\begin{cases}
\, \ti{u}_{\eps}\rightharpoonup u \ \ \mbox{weakly in} \ \ W^{1,p}({\rm B})\vspace{1mm}\\
\,  \mathcal{F}(\tilde u_{\eps},{\rm B})=\bar{\mathcal F}(u,{\rm B}) + \texttt{o}(\varepsilon,{\rm B})\,.
\end{cases}
$$
Moreover, in the case $u \in L^{\infty}({\rm B})$, by Lemma \ref{bf.1} we can also assume that 
\eqn{x.1.1b}
$$
\nr{\ti{u}_{\eps}}_{L^{\infty}({\rm B})}\le (1+\rr)\nr{u}_{L^{\infty}({\rm B})}+\varrho\,,
$$
where $\varrho$ denotes the radius of $B$. 
We fix a standard non-negative, radially symmetric mollifier $\phi\in C^{\infty}_{0}(\BB)$, $ \nr{\phi}_{L^{1}(\mathbb{R}^n)}=1$, $\mathcal B_{3/4} \subset \supp\,  \phi$ and such that $\phi(\cdot)>0$ on $\mathcal B_{3/4}$, and define
\eqn{approssima}
$$
\begin{cases}
\, \displaystyle F_{\delta}(x,z):=
\int_{\mathcal B_1} \int_{\mathcal B_1}F(x+\delta \lambda ,z+\delta y)\phi(\lambda)\phi(y)\d\lambda\dy\vspace{2.6mm}\\
\, F_{\eps, \delta}(x,z):=F_{\delta}(x,z)+\sigma_{\varepsilon} (\snr{z}^{2}+\mu_{\delta}^{2})^{q/2}\vspace{2.6mm}\\
\displaystyle \, \mu_{\delta}:=\mu+\delta\in (0,2], \quad \sigma_{\varepsilon}:= \frac{1}{1+1/\varepsilon+\nr{D\ti{u}_{\varepsilon}}_{L^{q}({\rm B})}^{2q}}\,.
\end{cases}
$$ 
By \rif{assf} and \rif{assf2}, and by construction, the integrands $F_{\eps, \delta}(\cdot)$ are $C^{\infty}$-regular, satisfy \eqref{assfr} on $\Omega\equiv \rm B$ with $\mu$ replaced by $\mu_\delta$ (the original $\mu$ being the one from initial assumptions \rif{assf}) a new constant $\ti{L}\geq L$, depending only on $n,p,q,L$ but independent of $\eps, \delta$, and with  $L_{*}\approx_{q} \sigma_{\varepsilon}^{-1}$. See \cite[Section 12]{ciccio} and \cite[Section 4.5]{onthe} and related references for details. By \rif{assf2} and  \rif{1d}$_1$ it easily follows that 
\eqn{lippi00}
$$|F_{ \delta}(x,z)
-F(x,z)|\lesssim_{n,q,L} \delta^{\alpha} (\snr{z}+1)^{q}$$
holds for every $(x,z)\in {\rm B}\times \er^n$. 
We then define 
$u_{\eps,\delta}\in \ti{u}_{\varepsilon}+W^{1,q}_{0}({\rm B})$ as the unique solution to
\eqn{pde0}
$$
u_{\eps,\delta}\mapsto \min_{w \in \ti{u}_{\varepsilon}+W^{1,q}_{0}({\rm B})} \mathcal{F}_{\eps,\delta}(w,{\rm B}),\qquad \mathcal{F}_{\eps,\delta}(w,{\rm B}):=\int_{{\rm B}}F_{\eps, \delta}(x,Dw)\dx\,.
$$
Using the very definition of $\sigma_{\varepsilon}$ and \rif{lippi00}, and recalling that $D\ti{u}_{\eps} \in L^{q}({\rm B})$, we have
\eqn{c.2}
$$
\begin{cases}
\, \displaystyle \sigma_{\varepsilon}\int_{{\rm B}}(\snr{D\ti{u}_{\varepsilon}}^{2}+\mu_{\delta}^{2})^{q/2}\dx \leq \sigma_{\varepsilon}\int_{{\rm B}}(\snr{D\ti{u}_{\varepsilon}}^{2}+4)^{q/2}\dx=\texttt{o}(\varepsilon,{\rm B})\vspace{2.6mm}\\
\displaystyle
 \, \nr{F_{\delta}(\cdot,D\ti{u}_{\varepsilon})-F(\cdot,D\ti{u}_{\varepsilon})}_{L^1({\rm B})}=\texttt{o}_{\varepsilon}(\delta,{\rm B})\,.
\end{cases}
$$
The information in the last display and \rif{x.1.1}$_2$ imply
\eqn{implyimply}
$$
\mathcal{F}_{\eps,\delta}(\ti{u}_{\varepsilon},{\rm B})=\bar{\mathcal{F}}(u,{\rm B})+\texttt{o}(\varepsilon,{\rm B})+\texttt{o}_{\varepsilon}(\delta,{\rm B})\,.
$$
The minimality of $u_{\eps,\delta}$ and \rif{assfr}$_1$ in turn give
\eqn{implyimply2}
 $$   c^{-1}\nr{Du_{\eps,\delta}}_{L^{p}({\rm B})}^{p}+c_{*}^{-1}\nr{Du_{\eps,\delta}}_{L^{q}({\rm B})}^q
  \leq \mathcal{F}_{\eps,\delta}(u_{\eps,\delta},{\rm B})\leq \mathcal{F}_{\eps,\delta}(\ti{u}_{\eps},{\rm B})\,,
$$
where $c\equiv c(n,p,q,L)$ and $c_{*}^{-1}\approx_{q}\sigma_{\varepsilon}$. 
We therefore conclude with 
\eqn{c.4}
 $$    \mathcal{F}_{\eps,\delta}(u_{\eps,\delta},{\rm B})+\nr{Du_{\eps,\delta}}_{L^{p}({\rm B})}^{p}+\sigma_{\varepsilon}\nr{Du_{\eps,\delta}}_{L^{q}({\rm B})}^q
\lesssim_{n,p,q,L}  \bar{\mathcal{F}}(u,{\rm B})+\texttt{o}(\varepsilon,{\rm B})+\texttt{o}_{\varepsilon}(\delta,{\rm B})
$$ 
and from now on we can assume that all the quantities denoted by $\texttt{o}(\varepsilon,{\rm B})$ and $\texttt{o}_{\varepsilon}(\delta,{\rm B})$ are non-negative. 
By \rif{c.4} and $ \ti{u}_{\eps} \in W^{1,q}({\rm B})$, for each $\eps >0$, the sequence $\{Du_{\eps,\delta}\}_{\delta}$ is bounded in $W^{1,q}({\rm B})$ (recall that $c_{*}$ is independent of $\delta$).  Therefore, up to not relabelled subsequences in $\delta$, for each $\eps>0$ we can assume that there exists $u_{\varepsilon}\in \ti{u}_{\varepsilon}+W^{1,q}_{0}({\rm B})$ such that $u_{\eps,\delta}\rightharpoonup u_{\varepsilon}$, weakly in $W^{1,q}({\rm B})$ as $\delta \to 0$, so that, letting $\delta \to 0$ in \rif{c.4} by weak lower semicontinuity we obtain 
\eqn{ggii2-1}
$$
\ \displaystyle   \nr{Du_{\eps}}_{L^{p}({\rm B})}^{p}
   \lesssim \, \bar{\mathcal{F}}(u,{\rm B})+\texttt{o}(\varepsilon,{\rm B}) \,.
   $$
 Note that, for every fixed  $\eps$, we are determining a different subsequence of $\delta$. 
By \rif{x.1.1}$_1$ and \rif{ggii2-1}, eventually passing to a not relabelled subsequence, we can now assume that there exists $\mathfrak{u} \in u+W^{1,p}_{0}({\rm B}) $ such that $u_{\eps}\rightharpoonup \mathfrak{u}$, weakly in $W^{1,p}({\rm B})$. 
Next, we have
\eqn{c.10}
$$
\nr{F(\cdot,Du_{\eps,\delta})-F_{\delta}(\cdot,Du_{\eps,\delta})}_{L^{1}( {\rm B})}= \texttt{o}_{\varepsilon}(\delta,{\rm B})\,.
$$
Indeed \rif{lippi00} and \rif{c.4} imply
\begin{flalign*}
\nr{F(\cdot,Du_{\eps,\delta})-F_{\delta}(\cdot,Du_{\eps,\delta})}_{L^{1}( {\rm B})}
& \lesssim  \delta^{\alpha} \left(\|Du_{\eps, \delta}\|_{L^{q}({\rm B})}^{q}+1\right)\\
& \lesssim \frac{ \delta^{\alpha} }{\sigma_\eps}\left[\bar{\mathcal{F}}(u,{\rm B})+1+\texttt{o}(\varepsilon,{\rm B})+\texttt{o}_{\varepsilon}(\delta,{\rm B})
\right] = \texttt{o}_{\varepsilon}(\delta,{\rm B})\,.
\end{flalign*}
We now show that $u=\mathfrak{u}$. 
Using standard weak lower semicontinuity theorems (see for instance \cite[Section 4.3]{giu}) we gain, for each fixed $\eps$
\begin{eqnarray*}
\mathcal{F}(u_{\varepsilon},{\rm B})&\le&\liminf_{\delta}\mathcal{F}(u_{\eps,\delta},{\rm B})\nonumber \\
&\le& \liminf_{\delta}\left(\mathcal{F}_{\eps,\delta}(u_{\eps,\delta},{\rm B})+\nr{F(\cdot,Du_{\eps,\delta})-F_{\delta}(\cdot,Du_{\eps,\delta})}_{L^{1}({\rm B})}\right)\nonumber \\
&\stackrel{\eqref{pde0},\eqref{c.10}}{\le}&\liminf_{\delta}\left(\mathcal{F}_{\eps,\delta}(\ti{u}_{\varepsilon},{\rm B})+\texttt{o}_{\varepsilon}(\delta,{\rm B})\right)\nonumber 
\end{eqnarray*}
so that \rif{implyimply} implies
$
\mathcal{F}(u_{\varepsilon},{\rm B}) \leq \bar{\mathcal{F}}(u,{\rm B})+\texttt{o}(\varepsilon,{\rm B}).
$
Recalling that $u_{\eps} \in W^{1,q}({\rm B})$ for every $\eps$, by the very definition in \rif{LSM} and again weak lower semicontinuity, we gain 
$\bar{\mathcal{F}}(\mathfrak{u},{\rm B})\leq \liminf_{\eps} \mathcal{F}(u_{\varepsilon},{\rm B}) \leq \bar{\mathcal{F}}(u,{\rm B})$. 
On the other hand the minimality of $u$ gives $\bar{\mathcal{F}}(u,{\rm B})\leq \bar{\mathcal{F}}(\mathfrak{u},{\rm B})$ so that $\bar{\mathcal{F}}(u,{\rm B})= \bar{\mathcal{F}}(\mathfrak{u},{\rm B})$. The strict convexity of the LSM-relaxation stated in Lemma \ref{bf.0} and the fact $u-\mathfrak{u} \in W^{1,p}_0(\rm B)$, now imply $u=\mathfrak{u}$.

\subsection{Proof of Theorem \ref{main3}} We use $\App{B_{\rr}}$ where $B_{\rr}\Subset \Omega$, $\rr\leq 1$ as in the statement of Theorem \ref{main3}. Applying Lemma \ref{massimino00} below to $u_{\eps,\delta}$, and using \rif{massiminobis} in conjunction with \rif{x.1.1b}, we come up with
\eqn{massimino2}
$$
\nr{u_{\eps, \delta}}_{L^{\infty}(B_{\rr})}\lesssim_{n,p,q,L} \nr{u}_{L^{\infty}(B_{\rr})}+\rr\,.
$$
Using \rif{hhhbb}, see Remark \ref{scemino} below, we gain 
\eqn{appina}
$$
\nr{Du_{\eps,\delta}}_{L^\qqq(B_{\rr/2})}^{\qqq}  \le 
c \left(\frac{\nr{u_{\eps,\delta}}_{L^{\infty}(B_{\rr})}}{\rr}+1\right)^{\mf{b}_{\qqq}}\left(\nr{Du_{\eps,\delta}}_{L^{p}(B_{\rr})}^{p}+\rr^n\right)
$$
so that, by means of \rif{c.4} and \rif{massimino2} we conclude with 
\eqn{ggii1}
$$
\nr{Du_{\eps,\delta}}_{L^\qqq(B_{\rr/2})}^{\qqq}  \le c \left(\frac{\nr{u}_{L^{\infty}(B_{\rr})}}{\rr} +1\right)^{\mf{b}_{\qqq}}\left(\bar{\mathcal{F}}(u,B_{\rr})+\rr^n+\texttt{o}(\varepsilon,B_{\rr})+\texttt{o}_{\varepsilon}(\delta,B_{\rr})\right)
$$
with $c\equiv c(n,p,q,\alpha,L,\qqq)\geq 1$ and $\mf{b}_{\qqq}\equiv \mf{b}_{\qqq}(n,p,q,\alpha,\qqq)\geq 1$. 
Letting first $\delta \to 0$ and then $\eps \to 0$ in \rif{ggii1}, and using weak lower semicontinuity, and yet  finally recalling that $u-(u)_{B_\rr}$ is still a local minimizer, we obtain \rif{m.4} and the proof of Theorem \ref{main3} is complete. 
\begin{remark}\label{scemino}{\em 
When deriving \rif{appina} we have in fact used  \rif{hhhbb} with $B_{\rr}$ replaced by $B_{\sigma \rr}$ for $\sigma <1 $, and then we have let $\sigma \to 1$. }
\end{remark}
\begin{lemma}[Maximum principle]\label{massimino00} Let $u \in W^{1,q}({\rm B})$ be such that
\eqn{massimino0}
$$
u\mapsto \min_{w\in u^*+W^{1,q}_{0}({\rm B})} \mathcal{F}(w,{\rm B}),\qquad u^*\in W^{1,q}({\rm B})\cap L^{\infty}({\rm B})\,,
$$
where ${\rm B}\subset \er^n$ is a ball and $\mathcal F$ is the functional defined in \eqref{ggg2}. Assume \eqref{assfr}$_{1}$ with $0 \leq \mu \leq 2$. Then
\eqn{massimino}
$$
\begin{cases}
\nr{u_+}_{L^\infty(\rm B)}  \leq \max\left\{\nr{u^*_+}_{L^\infty(\rm B)},  c\mu|{\rm B}|^{1/n} \right\}\vspace{2.5mm}\\ 
\nr{ (-u)_+ }_{L^\infty(\rm B)}  \leq \max\left\{\nr{(-u^*)_+}_{L^\infty(\rm B)},  c\mu|{\rm B}|^{1/n} \right\}
\end{cases}
$$
hold with $c\equiv c (n,p,q,\ti{L})$. In particular, it follows 
\eqn{massiminobis}
$$
\nr{u}_{L^\infty(\rm B)}  \leq \max\left\{\nr{u^*}_{L^\infty(\rm B)}, c\mu|{\rm B}|^{1/n} \right\}\,.
$$
\end{lemma}
\begin{proof} This is actually a variant of the classical maximum principle. Since we did not find an explicit reference, we rapidly include a proof. Let us prove the first inequality in \rif{massimino}. Scaling as in \rif{riscalamento} we can assume  ${\rm B}\equiv \mathcal B_1$. We take $\kk_\infty \geq  \nr{u^*_+}_{L^\infty(\mathcal B_1)}$ to be determined later, define $\kk_i:= \kk_{\infty}(2-2^{-i+1})\geq \kk_{\infty}=\kk_1$ for integers $i\geq 1$ and $A_{i} := \{u >\kk_i\}$. Testing minimality with $\min\{u,\kk_{i}\}\in u^*+W^{1,q}_0(\mathcal B_1)$ we obtain 
$\mathcal{F}(u,A_{i}) \leq \mathcal{F}(\min\{u,\kk_i\},A_{i})$ that, via 
\eqref{assfr}$_{1}$, yields
\eqn{caciotta}
$$
\int_{\mathcal B_1} |D(u-\kk_{i})_+|^p \dx \leq (2^{q-p}+1)\ti{L}^2 \mu^p |A_{i}|=: c_0^p|A_{i}|\,.
$$
Note that, when $\mu =0$, \rif{caciotta} with $i=1$ and $\kk_\infty \equiv   \nr{u^*_+}_{L^\infty(\mathcal B_1)}$ implies the first inequality in \rif{massimino} just using Poincarè's inequality as $(u-\kk_{i})_+\in W^{1,p}_0(\mathcal B_1)$ for every $i$. We can therefore assume that $\mu>0$. We proceed using a global variant of De Giorgi's technique as in \cite[page 318] {dibe2}. Sobolev embedding gives 
$
\kk_{\infty}^p2^{-pi}|A_{i+1}|
\leq c_1^p\nr{D(u-\kk_i)_+}_{L^{p}(\mathcal B_1)}^p \snr{A_i}^{p\sigma }
$
where $c_1\geq 1$ and $\sigma >0$ depend on $n,p$. Using \rif{caciotta} we arrive at $\snr{A_{i+1}}\leq 2^{pi}[c_1c_0\kk_{\infty}^{-1}]^p\snr{A_{i}}^{1+p\sigma}$ for every $i$. Using \cite[Lemma 15.1]{dibe2} we have that 
$\snr{A_{1}}\leq 2^{-1/(p\sigma^2)}(c_1c_0)^{-1/\sigma}\kk_{\infty}^{1/\sigma}$ implies $\snr{A_{i}}\to 0$ as $i\to \infty$ so that $u_+\leq 2 \kk_{\infty}$. Estimating $\snr{A_1}\leq \snr{\mathcal B_1}$, this last condition is satisfied provided $\kk_{\infty}\geq 2^{1/(p\sigma)} c_1c_0 \snr{\mathcal B_1}^{\sigma}\approx \mu$ so that the first inequality in \rif{massimino} follows up to the initial scaling argument. The second inequality in \rif{massimino} follows from the first one observing that $-u$ solves \rif{massimino0} with $-u^*$ replacing $u^*$ and the integrand $F(x, -z)$, that still satisfies \eqref{assfr}$_{1}$, replacing  $F(x,z)$. 
\end{proof}
\section{$C^{0,1}$-estimates with improved exponents (after Bella \& Sch\"affner)}\label{bellasec}  Let ${\rm B}\subset \mathbb{R}^{n}$ be a ball and $F_{0}\colon\mathbb{R}^{n}\to \mathbb{R}$ be an autonomous integrand satisfying $\eqref{assfr}_{1,2,3}$ (no $x$-dependence here) with $p,q$ satisfying 
\eqn{pq0}
$$
\frac qp < 1 +  \frac{2}{n-1}\,.
$$
Note that, according to \eqref{assfr}, here $\mu$ is positive. With $\ti{u}$ being a fixed, general $W^{1,\infty}({\rm B})$-regular function, we define $v\in \ti{u}+W^{1,q}_{0}({\rm B})$ by solving the Dirichlet problem
\eqn{pd0}
$$
v\mapsto \min_{w\in \ti{u}+W^{1,q}_{0}({\rm B})} \int_{{\rm B}}F_{0}(Dw)\dx\,,
$$
that admits a unique solution by strict convexity and coercivity. Standard regularity theory yields that $v\in W^{1,\infty}_{\loc}({\rm B})\cap W^{2,2}_{\loc}({\rm B})$ and $E_{\mu}(Dv)\in W^{1,2}_{\loc}({\rm B})$; see \cite[Section 4]{BM}, \cite[Section 5.2]{piovra} and related references. Moreover, \cite[Lemma 5.1]{piovra}, see also \cite{BM},  
yields that whenever $B_{\rr}\Subset {\rm B}$ is a ball, the (renormalized) Caccioppoli type inequality
\eqn{c0}
$$
\int_{B_{\rr/2}}\snr{D(E_{\mu}(Dv)-\kk)_{+}}^{2}\dx\le \frac{c\mf{M}^{q-p}}{\rr^{2}}\int_{B_{\rr}}(E_{\mu}(Dv)-\kk)_{+}^{2}\dx\,,
$$
holds with $c\equiv c(n,p,q,\ti{L})$, whenever $\mf{M}\ge \max \{\nr{Dv}_{L^{\infty}(B_{\rr})}, 1 \}$ and $\kk\ge 0$. Recall that the function $E_{\mu}(\cdot)$ has been introduced in \rif{defiE}. We define the gap exponent $\mf{s}$ via
\eqn{sss}
$$
\mf{s}:=\begin{cases}
\displaystyle 
\, \frac{2q}{(n+1)p-(n-1)q}\qquad \mbox{if either $n\ge 4 $ or $n=2$}\vspace{3mm}\\
\displaystyle
\, \mbox{any number} > \frac{q}{2p-q}\qquad \mbox{if $n=3$ and $q>p$} \vspace{3mm}
\\ 
\displaystyle
\, 1 \qquad \mbox{if $n=3$ and $q=p$}\,.
\end{cases}
$$
Notice that $\mf{s}\geq 1$ and $\mf{s}=1$ iff $p=q$.
Recalling \eqref{pq0}, we can use \cite[Theorem 5.1]{BS1} when $n\geq 3$ and \cite[Theorem 1]{schag} when $n=2$, that yield 
\eqn{lippi}
$$
\nr{Dv}_{L^{\infty}({\rm B}/2)} \le  c\left(\mint_{{\rm B}}F_{0}(Dv)\dx+1\right)^{\mf{s}/q}
$$
with $c\equiv c(n,p,q,\ti{L})$. 
Comments on \rif{sss} are in Section \ref{proof6} below. 
Using minimality of $v$ and then \eqref{assfr}$_{1}$, the last inequality implies
\begin{flalign}
 \notag \nr{Dv}_{L^{\infty}({\rm B}/2)} & \leq  c\left(\mint_{{\rm B}}F_{0}(D\ti{u})\dx+1\right)^{\mf{s}/q}\\
 &  \leq c\left(\mint_{{\rm B}}\snr{D\ti{u}}^{q}\dx+1\right)^{\mf{s}/q}\le c\nr{D\ti{u}}_{L^{\infty}({\rm B})}^{\mf{s}}+ c \label{lip0}
\end{flalign}
again for $c\equiv c(n,p,q,\ti{L})$. It  follows that 
\eqn{lip0.0}
$$
M\ge \max\left\{\nr{D\ti{u}}_{L^{\infty}({\rm B})},1\right\}\Longrightarrow \nr{Dv}_{L^{\infty}({\rm B/2})}\leq cM^{\mf{s}}\,,
$$
with $c\equiv c(n,p,q,\ti{L})$. In particular, \eqref{lip0.0} allows to choose $\mf{M}=M^{\mf{s}}$ in \eqref{c0} thereby getting
\eqn{lip000pre}
$$
\int_{B_{\rr/2}}\snr{D(E_{\mu}(Dv)-\kk)_{+}}^{2}\dx\le \frac{cM^{\mf{s}(q-p)}}{\rr^{2}}\int_{B_{\rr}}(E_{\mu}(Dv)-\kk)_{+}^{2}\dx
$$
provided $B_{\rr}\subseteq {\rm B}/2$. 
\begin{remark}\label{discuss}{\em The reader might wonder why the case $n=3$ is special in \rif{sss}. This distinction arises from the differing nature of the proofs for \rif{lippi} in dimensions $n> 2$ and $n=2$, as established in \cite{BS1} and \cite{schag}, respectively. For $n>2$, the proof of \rif{lippi} is based on a differentiation of the Euler-Lagrange equation of the functional with a De Giorgi's type iteration in $W^{1,2}$ involving truncations of $\snr{Du}^p$ (see also \cite{BM}). Notably, Sobolev embedding theorem is employed on surfaces of balls rather than on solid balls, which allows for an improved estimate on the ratio $q/p$, effectively gaining an additional dimension compared to earlier results in the literature \cite{BM, ma2}. Indeed, in De Giorgi's iteration, applying the Sobolev embedding theorem on surfaces leads to a different critical dimension for the embedding, determined by imposing $n-1=2$ instead of $n=2$. In this case the imbedding is in a sense suboptimal as it does lead directly to $L^\infty$ but to any finite power. This results in a slight increase in the exponent $\mf{s}$ when $n=3$. However, this still enables the authors to improve  the existing bounds. Further details on such an aspect are reported in Section \ref{proof6}. In contrast, the case $n=2$, addressed in \cite{schag}, relies on a different underlying principle. Specifically, by projecting the analysis onto lines and using a simpler one-dimensional embedding argument, a corresponding improvement is achieved via maximum principle rather than De Giorgi's iteration. As an ultimate outcome, no increase in the exponent $\mf{s}$ occurs for $n=2$. We shall apply \rif{lippi} in Section \ref{proofmain} below, when proving Theorem \ref{main}, and we shall indeed treat the cases $n\not =3$ and $n=3$ separately, in Sections \ref{proof3} and \ref{proof4}, respectively. We also note that, nevertheless, the difference in the proofs of the two cases is essentially minimal}.
\end{remark}

\section{Proof of Theorem \ref{main}}\label{proofmain}
The proof goes along Sections \ref{proof1}-\ref{proof7}.  The focal point of Theorem \ref{main}, as usual in nonuniformly elliptic problems, is the Lipschitz estimate of Proposition \ref{lipt} below. This is stated for minima of uniformly elliptic integrals with standard $q$-growth and eventually made effective via the approximation $\App{\rm B}$ in Section \ref{proof7}. The key tool is the fractional Caccioppoli type inequality in Section \ref{proof2}. At this stage the span of different parameters \rif{bkm} necessary to make \rif{28} produce useful estimates is allowed by the higher integrability result of Proposition \ref{scat}. Finally, a nonlinear potential theoretic version of De Giorgi's iteration in Section \ref{proof3} completes the picture. Once a priori Lipschitz estimates are at hand, gradient H\"older continuity for some exponent $\alpha_*$ as in \rif{m.2} can be reached using the methods in \cite{piovra}. Finally, in the non-degenerate case $\mu>0$, the full exponent $\alpha$ as in \rif{m.3} is achieved appealing to Proposition \ref{piovrapert}; this is done in Section \ref{proof7}.
\subsection{Uniform Lipschitz}\label{proof1} That is
\begin{proposition}\label{lipt}
Let $u\in W^{1,q}_{\loc}(\Omega)$ be a minimizer of the functional  $\mathcal F$ in  \eqref{ggg2}, under assumptions \trif{pq} and \eqref{assfr}. Then
 \eqn{hhhlbissa}
$$
\nr{Du}_{L^{\infty}(B_{\rr/2})} \le c\left(\mint_{B_{\rr}}\snr{Du}^{p}\dx+1\right)^{\mf{b}}
$$
holds for every ball $B_{\rr}\Subset \Omega$, $\rr \leq 1$, where $c\equiv c(n,p,q,\alpha,\ti{L})\geq 1$ and $\mf{b}\equiv \mf{b}(n,p,q,\alpha)\geq 1$.
\end{proposition}
Before going to the proof of Proposition \ref{lipt}, let us remark that 
we can confine ourselves to prove that \rif{hhhlbissa} holds in an (only apparently) weaker form, i.e., we prove that 
 \eqn{hhhl}
$$
\nr{Du}_{L^{\infty}(B_{\rr/4})} \le c\left(\mint_{B_{2\rr}}\snr{Du}^{p}\dx+1\right)^{\mf{b}}
$$
holds for every ball $B_{2\rr}\Subset \Omega$, $\rr \leq 1$, where $c\equiv c(n,p,q,\alpha,\ti{L})\geq 1$ and $\mf{b}\equiv \mf{b}(n,p,q,\alpha)\geq 1$.
Indeed, \rif{hhhlbissa} follows from \rif{hhhl} by a standard covering argument, eventually enlarging the constant $c$, but with the same exponent $\mf{b}$. Therefore, for the rest of the proof of Proposition \ref{lipt} we concentrate on \rif{hhhl}. 

\subsection{A Fractional (renormalized) Caccioppoli inequality} \label{proof2} This is in\begin{lemma}\label{fracacci}
Let $u\in W^{1,q}_{\loc}(\Omega)$ be a minimizer of the functional  $\mathcal F$ in  \eqref{ggg2}, under assumptions \eqref{assfr} and \trif{pq0}. Then for all balls $B_{\rr}\equiv B_{\rr}(x_{0}) \Subset \Omega$, $\rr\leq 1$ 
\begin{flalign}
& \nonumber \rr^{\beta-n/2} [(E_{\mu}(Du)-\kk)_{+}]_{\beta,2;B_{\rr/2}}+ \left(\mint_{B_{\rr/2}}(E_{\mu}(Du)-\kk)_{+}^{2\chi}\dx\right)^{\frac{1}{2\chi}} \\
&\notag \qquad  \qquad \le \frac{cM^{\mf{s}(q-p)/2}}{(\tilde \alpha(\kkk, \delta)-\beta)^{1/2}}\left(\mint_{B_{\rr}}(E_{\mu}(Du)-\kk)_{+}^{2}\dx\right)^{1/2}\\ & \qquad  \qquad  \qquad \qquad  +\frac{cM^{\mf{s}(q-p)/2}}{(\tilde \alpha(\kkk, \delta)-\beta)^{1/2}}\rr^{\alpha\kkk}\left(\mint_{B_{\rr}}(\snr{Du}+1)^{\mf{p}}\dx\right)^{1/2}\label{28}
\end{flalign}
holds for every $\kk \geq 0$  provided 
\eqn{MMM}
$$M\ge \max\left\{\nr{Du}_{L^{\infty}(B_{\rr})},1\right\}\,,$$ where $\mf{s}$  is in \eqref{sss} and $\beta, \chi, \kkk, \delta,  \mf{p}$ are such that
\eqn{bkm}
$$
\begin{cases}
\displaystyle  \, \chi \equiv \chi(\beta)  :=\frac{n}{n-2\beta} \,, \quad 0 < \beta<\tilde \alpha(\kkk, \delta), \quad \tilde \alpha(\kkk, \delta):=\frac{\alpha\kkk-p\delta}{1+\alpha\kkk-p\delta}\vspace{2mm}\\
\displaystyle \,  0 <  \kkk < 1\,, \quad 0 < \delta <\alpha\kkk/p\vspace{2mm}\\
\displaystyle \, \mf{p}:=\mf{m}+p\mf{s}, \quad \mf{m}:=2+2\textnormal{\texttt{b}}_{\kkk}+\frac{n\mf{s}}{2\delta},\quad \textnormal{\texttt{b}}_{\kkk}:= \max\left\{\frac{2q-p}{2},\frac{p(q-1)}{2(p-1)},\frac{p(2-\kkk)}{2(1-\kkk)}\right\}\,.
\end{cases}
$$
The constant $c$ appearing in \eqref{28} depends on $n,p,q,\alpha,\ti{L}, \beta$. 
\end{lemma}
\begin{proof} We rescale $u$ in $\mathcal B_1$ by considering $\ti{u}\in C^{1,\gamma_{0}}(\mathcal B_{1})$ defined in \rif{riscalamento}. This minimizes the functional 
$$
W^{1,q}(\mathcal B_1)\ni w\mapsto  \int_{\mathcal B_1}\ti{F}(x,Dw)\dx
$$
with $\ti{F}(\cdot)$ as in \eqref{funr}$_2$; recall that $\ti{F}(\cdot)$ satisfies \rif{assfr}$_{1,2,3}$ and \rif{riscalone} and observe that \rif{MMM} implies
\eqn{MMManc}
$$M\ge \max\left\{\nr{D\tilde u}_{L^{\infty}(\mathcal B_1)},1\right\}\,.$$
{\em Step 1: Estimates on balls.} To proceed, fix $\beta_{0}\in (0,1)$, $h\in \mathbb{R}^{n}$ such that 
\eqn{sogliah}
$$0< \snr{h} < \frac1{2^{8/\beta_{0}}}$$ and, with $x_{\rm c}\in \mathcal B_1$ being such that $\snr{x_{\rm c}} \leq  1/2+ 2|h|^{\beta_{0}}$, we consider $B_{h}:=B_{\snr{h}^{\beta_{0}}}(x_{\rm c})$, so that  $8B_{h}\Subset \mathcal{B}_{1}$. We then define $v\equiv v(B_{h})\in W^{1,q}(8B_{h})$ as the solution to the Dirichlet problem
\begin{flalign}\label{pd}
 \ti{u}+W^{1,q}_{0}(8B_{h})\ni v\mapsto \min_{w\in  \ti{u}+W^{1,q}_{0}(8B_{h})} \int_{8B_{h}} \ti{F}_{0}(Dw)  \dx
\end{flalign}
where $ \ti{F}_{0}(z):=\ti{F}(x_{\rm c},z)$, which clearly verifies $\eqref{assfr}_{1,2,3}$ (with no $x$-dependence). 
By the assumed \eqref{pq0} the content of Section \ref{bellasec} is available so that $
v $ satisfies \eqref{lip0}-\eqref{lip000pre} with $ {\rm B}\equiv 8B_{h}$ and with the choice of $\ti{u}$ made in \rif{riscalamento}. Specifically, \eqref{lip000pre} reads as
\eqn{lip000}
$$
\int_{2B_{h}}\snr{D(E_{\mu}(Dv)-\kk)_{+}}^{2}\dx\le \frac{cM^{\mf{s}(q-p)}}{\snr{h}^{2\beta_{0}}}\int_{4B_{h}}(E_{\mu}(Dv)-\kk)_{+}^{2}\dx\,,
$$
where $c\equiv c(n,p,q,\ti{L})$. Noting that $\mf{m}\geq 2q-p+2> q$, we can further refine \eqref{lip0} via Jensen's inequality as follows:
$$
 \nr{Dv}_{L^{\infty}(4B_{h})}\stackrel{\eqref{lip0}}{\le}  c\left(\mint_{8B_{h}}\snr{D\ti{u}}^{q}\dx+1\right)^{\mf{s}/q}
\leq  c |h|^{-n\beta_{0}\mf{s}/\mf{m}}\left(\int_{8B_{h}}\snr{D\ti{u}}^{\mf{m}}\dx\right)^{\mf{s}/\mf{m}}+c$$
and therefore, as it is also $\mf{m}>n\mf{s}/(2\delta)$, we conclude with
\eqn{22}
$$
 \nr{Dv}_{L^{\infty}(4B_{h})} \leq c\snr{h}^{-2\beta_{0}\delta}\nr{D\ti{u}}_{L^{\mf{m}}(\mathcal{B}_{1})}^{\mf{s}}+c
$$
with $c\equiv c(n,p,q,\ti{L})$. 
By minimality $v$ solves the Euler-Lagrange equation
\eqn{euleri}
$$\int_{8B_{h}} \partial_{z} \ti{F}_{0}(Dv)\cdot  D\varphi  \dx=0 \qquad \mbox{for all} \ \ \varphi\in W^{1,q}_{0}(8B_{h})\,.
$$
As $\varphi\equiv \ti{u}-v\in W^{1,q}_{0}(8B_{h})$ is admissible in both \eqref{els} and \eqref{euleri}, setting 
$$\begin{cases}
\, \mathcal{V}_{\mu}:=\snr{V_{\mu}(D\ti{u})-V_{\mu}(Dv)} \vspace{1.5mm}\\
\,    H_{\mu}:=H_{\mu}(D\tilde u)=\snr{D\ti{u}}^{2}+\mu^2 \vspace{1.5mm}\\
\,  H:=H_{1}(D\tilde u)=\snr{D\ti{u}}^{2}+1\vspace{1.5mm}\\
\, \mathcal{E}_{\kk} :=\snr{(E_{\mu}(D\ti{u})-\kk)_{+}-(E_{\mu}(Dv)-\kk)_{+}}, \quad \kappa \geq 0\,,
\end{cases}$$ we have
\begin{eqnarray}
\int_{8B_{h}}\mathcal{V}_{\mu}^{2}\dx&\stackrel{\eqref{1d}_{2}}{\le}&c\int_{8B_{h}}(\partial_{z} \ti{F}_{0}(D\ti{u})-\partial_{z} \ti{F}_{0}(Dv))\cdot (D\ti{u}-Dv)\dx\nonumber \\
&\stackrel{\eqref{els},\eqref{euleri}}{=}&c\int_{8B_{h}}(\partial_{z} \ti{F}(x_{\rm c},D\ti{u})-\partial_{z} \ti{F}(x,D\ti{u}))\cdot (D\ti{u}-Dv)\dx\nonumber \\
&\stackrel{\eqref{riscalone}}{\le}&c\rr^{\alpha}\snr{h}^{\alpha\beta_{0}}\int_{8B_{h}}(H^{(q-1)/2}_{\mu}+H^{(p-1)/2}_{\mu})\snr{D\ti{u}-Dv}\dx:=\mbox{(I)}.\label{abacab}
\end{eqnarray}
If $p\ge 2$, by Young's inequality we get
\begin{eqnarray*} 
\mbox{(I)}&=&c\rr^{\alpha}\snr{h}^{\alpha\beta_{0}}\int_{8B_{h}}(H^{(q-1)/2}_{\mu}+H^{(p-1)/2}_{\mu})(\snr{D\ti{u}}^{2}+\snr{Dv}^{2}+\mu^{2})^{\pm (p-2)/4}\snr{D\ti{u}-Dv}\dx\nonumber \\
&\stackrel{\eqref{Vm}_{1}}{\le}&\frac{1}{2}\int_{8B_{h}}\mathcal{V}_{\mu}^{2}\dx+c\rr^{2\alpha}\snr{h}^{2\alpha\beta_{0}}\int_{8B_{h}}(H^{q-1}_{\mu}+H^{p-1}_{\mu})(\snr{D\ti{u}}^{2}+\snr{Dv}^{2}+\mu^{2})^{1-p/2}\dx\nonumber \\
&\stackrel{p\ge 2}{\le}&\frac{1}{2}\int_{8B_{h}}\mathcal{V}_{\mu}^{2}\dx+c\rr^{2\alpha}\snr{h}^{2\alpha\beta_{0}}\int_{8B_{h}}(H^{q-1}_{\mu}+H^{p-1}_{\mu})H_{\mu}^{1-p/2}\dx\nonumber \\
&\le&\frac{1}{2}\int_{8B_{h}}\mathcal{V}_{\mu}^{2}\dx+c\rr^{2\alpha}\snr{h}^{2\alpha\beta_{0}}\int_{8B_{h}}H^{\frac{2q-p}{2}}\dx\,.
\end{eqnarray*}
When $1<p<2$ we instead have
\begin{eqnarray*}
\mbox{(I)}&\leq &c\rr^{\alpha}\snr{h}^{\alpha\beta_{0}}\int_{8B_{h}}H^{(q-1)/2}\snr{D\ti{u}-Dv}\dx\\
&\stackrel{\eqref{Vm}_{2}}{\le}&c\rr^{\alpha}\snr{h}^{\alpha\beta_{0}}\int_{8B_{h}}H^{(q-1)/2}\mathcal{V}_{\mu}^{2/p}\dx+c\rr^{\alpha}\snr{h}^{\alpha\beta_{0}}\int_{8B_{h}}H^{q/2-p/4}\mathcal{V}_{\mu}\dx\nonumber \\
&\le&\frac{1}{2}\int_{8B_{h}}\mathcal{V}_{\mu}^{2}\dx+c\rr^{\frac{p\alpha}{p-1}}\snr{h}^{\frac{p\alpha\beta_{0}}{p-1}}\int_{8B_{h}}H^{\frac{p(q-1)}{2(p-1)}}\dx\nonumber +c\rr^{2\alpha}\snr{h}^{2\alpha\beta_{0}}\int_{8B_{h}} H^{\frac{2q-p}{2}}\dx\nonumber\\
&\stackrel{1<p<2}{\le}&\frac{1}{2}\int_{8B_{h}}\mathcal{V}_{\mu}^{2}\dx+c\rr^{2\alpha}\snr{h}^{2\alpha\beta_{0}}\int_{8B_{h}} H^{\frac{p(q-1)}{2(p-1)}}\dx\,.
\end{eqnarray*}
In both the last two displays it is $c\equiv c(n,p,q,\ti{L})$; using their content in \rif{abacab} and reabsorbing the $\mathcal{V}_{\mu}^{2}$-term we obtain, in the full range $p>1$
\begin{flalign}\label{23}
\int_{8B_{h}}\mathcal{V}_{\mu}^{2}\dx\le c\rr^{2\alpha}\snr{h}^{2\alpha\beta_{0}}\int_{8B_{h}}\left(H^{\frac{2q-p}{2}}+H^{\frac{p(q-1)}{2(p-1)}}\right)\dx=:c\rr^{2\alpha}\snr{h}^{2\alpha\beta_{0}}\textnormal{\texttt{H}}\,,
\end{flalign}
again with $c\equiv c(n,p,q,\ti{L})$. Using \rif{diffH}, we have that 
\eqn{24}
$$
\mathcal{E}_{\kk}^{2}  \leq \snr{E_{\mu}(D\ti{u})-E_{\mu}(Dv)}^{2}\lesssim_{p}(\snr{D\ti{u}}^{2}+\snr{Dv}^{2}+\mu^{2})^{p/2}\mathcal{V}_{\mu}^{2}
$$
holds for every $\kk\ge 0$ (recall that $t \mapsto (t-\kk)_+$ is $1$-Lipschitz) and then, with $\kkk \in (0,1)$ as in \rif{bkm}, and via repeated use of \rif{elmy0}, we estimate as follows:
\begin{eqnarray*}
&& \int_{4B_{h}}\mathcal{E}_{\kk}^{2}\dx\\ &&\qquad\stackrel{\eqref{24}}{\le}c\int_{4B_{h}}H^{p/2}\mathcal{V}_{\mu}^{2(1-\kkk)}\mathcal{V}_{\mu}^{2\kkk}\dx+c\nr{Dv}_{L^{\infty}(4B_{h})}^{p}\int_{4B_{h}}\mathcal{V}_{\mu}^{2}\dx\nonumber \\
&& \qquad\stackrel{\eqref{22}}{\le}c\int_{4B_{h}}H^{p/2}\left(\snr{V_{\mu}(D\ti{u})}^{2(1-\kkk)}+\snr{V_{\mu}(Dv)}^{2(1-\kkk)}\right)\mathcal{V}_{\mu}^{2\kkk}\dx\nonumber \\
&&\qquad\qquad+c\left(\snr{h}^{-2\beta_{0}p\delta }\nr{D\ti{u}}_{L^{\mf{m}}(\mathcal{B}_{1})}^{p\mf{s}}+1\right)\int_{4B_{h}}\mathcal{V}_{\mu}^{2}\dx\nonumber \\
&&\qquad\stackrel{\eqref{23}}{\le}c\int_{4B_{h}} H^{p(2-\kkk)/2}\mathcal{V}_{\mu}^{2\kkk}\dx+c\left(\nr{Dv}_{L^{\infty}(4B_{h})}^{p(1-\kkk)}+1\right)\int_{4B_{h}}H^{p/2}\mathcal{V}_{\mu}^{2\kkk}\dx\nonumber \\
&&\qquad\qquad+c\rr^{2\alpha}\snr{h}^{2\beta_{0}(\alpha-p\delta)}\left(\nr{D\ti{u}}_{L^{\mf{m}}(\mathcal{B}_{1})}^{p\mf{s}}+1\right)\textnormal{\texttt{H}}\nonumber \\
&&\qquad\stackrel{\eqref{22}}{\le}c\left(\int_{4B_{h}}H^{\frac{p(2-\kkk)}{2(1-\kkk)}}\dx\right)^{1-\kkk}\left(\int_{4B_{h}}\mathcal{V}_{\mu}^{2}\dx\right)^{\kkk}\nonumber \\
&&\qquad\qquad+c\snr{h}^{-2\beta_{0}p(1-\kkk)\delta}\left(\nr{D\ti{u}}_{L^{\mf{m}}(\mathcal{B}_{1})}^{p\mf{s}(1-\kkk)}+1\right)\left(\int_{4B_{h}}H^{\frac{p}{2(1-\kkk)}}\dx\right)^{1-\kkk} \left(\int_{4B_{h}}\mathcal{V}_{\mu}^{2}\dx\right)^{\kkk}\nonumber \\
&&\qquad\qquad+c\rr^{2\alpha}\snr{h}^{2\beta_{0}(\alpha-p\delta)}\left(\nr{D\ti{u}}_{L^{\mf{m}}(\mathcal{B}_{1})}^{p\mf{s}}+1\right)\textnormal{\texttt{H}}\nonumber \\
&&\qquad\stackrel{\eqref{23}}{\le}c\rr^{2\alpha\kkk}\snr{h}^{2\alpha\beta_{0} \kkk}\left(\int_{4B_{h}}H^{\textnormal{\texttt{b}}_{\kkk}}\dx\right)^{1-\kkk}\textnormal{\texttt{H}}^{\kkk}\nonumber\\
&&\qquad\qquad+c\rr^{2\alpha\kkk}\snr{h}^{2\beta_{0}[\alpha\kkk-p(1-\kkk)\delta]}\left(\nr{D\ti{u}}_{L^{\mf{m}}(\mathcal{B}_{1})}^{p\mf{s}(1-\kkk)}+1\right)\left(\int_{4B_{h}}H^{\textnormal{\texttt{b}}_{\kkk}}\dx\right)^{1-\kkk}\textnormal{\texttt{H}}^{\kkk}\nonumber \\
&&\qquad\qquad+c\rr^{2\alpha}\snr{h}^{2\beta_{0}(\alpha-p\delta)}\left(\nr{D\ti{u}}_{L^{\mf{m}}(\mathcal{B}_{1})}^{p\mf{s}}+1\right)\textnormal{\texttt{H}}\,.
\end{eqnarray*}
We have used that $H\geq 1$. Recalling the definition of $\textnormal{\texttt{H}}$ in \rif{23}, and that of $\texttt{b}_{\kkk}$ in \rif{bkm}, observe that 
$$
\textnormal{\texttt{H}}\leq 2\int_{8B_{h}}H^{\textnormal{\texttt{b}}_{\kkk}}\dx\Longrightarrow
\left(\int_{4B_{h}}H^{\textnormal{\texttt{b}}_{\kkk}}\dx\right)^{1-\kkk}\textnormal{\texttt{H}}^{\kkk}\leq 2\int_{8B_{h}}H^{\textnormal{\texttt{b}}_{\kkk}}\dx
$$
 and therefore we conclude with 
\eqn{25}
$$\int_{4B_{h}}\mathcal{E}_{\kk}^{2}\dx\le c\snr{h}^{2\beta_{0}(\alpha\kkk-p\delta)}\left(\nr{D\ti{u}}_{L^{\mf{m}}(\mathcal{B}_{1})}^{p\mf{s}}+1\right)\rr^{2\alpha\kkk}\int_{8B_{h}}H^{\textnormal{\texttt{b}}_{\kkk}}\dx\,,
$$
where $c\equiv c(n,p,q,\ti{L})$. Note that for any $f\in L^{2}(2B_{h})$, since $|h|\leq |h|^{\beta_{0}}$ implies
\eqn{biglie}
$$
h + B_h \equiv h+ B_{\snr{h}^{\beta_{0}}}(x_{\rm c}) \subset  B_{\snr{h}^{\beta_{0}}+|h|}(x_{\rm c})\subset B_{2\snr{h}^{\beta_{0}}}(x_{\rm c})\equiv  2 B_{h}\,,
$$
then, recalling \rif{gh} it follows that
\eqn{motley}
$$
\begin{cases}
\, \displaystyle \nr{f(\cdot+h)}_{L^2(B_{h})} \leq \nr{f}_{L^2(2B_{h})}, \ \  \nr{\tau_h f}_{L^2(B_{h})}  \leq 2  \nr{f}_{L^2(2B_{h})}\vspace{1.2mm}\\
\, \displaystyle \nr{\tau_h f}_{L^2(B_{h})} \leq |h|  \nr{Df}_{L^2(2B_{h})}\,.
\end{cases}
$$
Then we have 
\begin{eqnarray*}
&&\int_{B_{h}}\snr{\tau_{h}(E_{\mu}(D\ti{u})-\kk)_{+}}^{2}\dx\\
&& \qquad \ \, \,   \le c\int_{B_{h}}\snr{\tau_{h}(E_{\mu}(Dv)-\kk)_{+}}^{2}\dx \\
&& \qquad \qquad +c \int_{B_{h}}\snr{\tau_{h}((E_{\mu}(D\ti{u})-\kk)_{+}-(E_{\mu}(Dv)-\kk)_{+})}^{2}\dx\nonumber \\
&& \qquad \stackleq{motley} c \int_{B_{h}}\snr{\tau_{h}(E_{\mu}(Dv)-\kk)_{+}}^{2}\dx+c\int_{2B_{h}}\mathcal{E}_{\kk}^{2}\dx\nonumber \\
&&\qquad \stackleq{motley}c \snr{h}^{2}\int_{2B_{h}}\snr{D(E_{\mu}(Dv)-\kk)_{+}}^{2}\dx+c\int_{2B_{h}}\mathcal{E}_{\kk}^{2}\dx\nonumber \\
&&\qquad \stackrel{\eqref{25}}{\le}c\snr{h}^{2}\int_{2B_{h}}\snr{D(E_{\mu}(Dv)-\kk)_{+}}^{2}\dx\nonumber \\
&&\qquad \qquad +c\snr{h}^{2\beta_{0}(\alpha\kkk-p\delta)}\left(\nr{D\ti{u}}_{L^{\mf{m}}(\mathcal{B}_{1})}^{p\mf{s}}+1\right)\rr^{2\alpha\kkk}\int_{8B_{h}}H^{\textnormal{\texttt{b}}_{\kkk}}\dx\nonumber \\
&&\qquad \stackrel{\eqref{lip000}}{\le}c\snr{h}^{2(1-\beta_{0})}M^{\mf{s}(q-p)}\int_{4B_{h}}(E_{\mu}(Dv)-\kk)_{+}^{2}\dx\nonumber \\
&&\qquad \qquad +c\snr{h}^{2\beta_{0}(\alpha\kkk-p\delta)}\left(\nr{D\ti{u}}_{L^{\mf{m}}(\mathcal{B}_{1})}^{p\mf{s}}+1\right)\rr^{2\alpha\kkk}\int_{8B_{h}}H^{\textnormal{\texttt{b}}_{\kkk}}\dx\nonumber \\
&&\qquad \ \, \,   \le c\snr{h}^{2(1-\beta_{0})}M^{\mf{s}(q-p)}\int_{4B_{h}}(E_{\mu}(D\ti{u})-\kk)_{+}^{2}\dx\\ 
&& \qquad\qquad +c\snr{h}^{2(1-\beta_{0})}M^{\mf{s}(q-p)}\int_{4B_{h}}\mathcal{E}_{\kk}^{2}\dx\nonumber \\
&&\qquad \qquad +c\snr{h}^{2\beta_{0}(\alpha\kkk-p\delta)}\left(\nr{D\ti{u}}_{L^{\mf{m}}(\mathcal{B}_{1})}^{p\mf{s}}+1\right)\rr^{2\alpha\kkk}\int_{8B_{h}}H^{\textnormal{\texttt{b}}_{\kkk}}\dx\nonumber \\
&&\qquad \stackrel{\eqref{25}}{\le} c\snr{h}^{2(1-\beta_{0})}M^{\mf{s}(q-p)}\int_{8B_{h}}(E_{\mu}(D\ti{u})-\kk)_{+}^{2}\nonumber \\
&&\qquad \qquad +c\snr{h}^{2\beta_{0}(\alpha\kkk-p\delta)}M^{\mf{s}(q-p)}\left(\nr{D\ti{u}}_{L^{\mf{m}}(\mathcal{B}_{1})}^{p\mf{s}}+1\right) \rr^{2\alpha\kkk}\int_{8B_{h}}H^{\textnormal{\texttt{b}}_{\kkk}}\dx,
\end{eqnarray*}
for $c\equiv c(n,p,q,\ti{L})$. We optimize the choice of $\beta_{0}\in (0,1)$ by choosing
\eqn{sceltachoose}
$$
1-\beta_{0}=\beta_{0}(\alpha\kkk-p\delta) \Longleftrightarrow \beta_{0}=\frac{1}{1+\alpha\kkk-p\delta}
$$
so that we finally arrive at
\begin{flalign}\label{26.1}
\int_{B_{h}}\snr{\tau_{h}(E_{\mu}(D\ti{u})-\kk)_{+}}^{2}\dx&\le c\snr{h}^{2\tilde \alpha(\kkk, \delta)}M^{\mf{s}(q-p)}\int_{8B_{h}}(E_{\mu}(D\ti{u})-\kk)_{+}^{2}\dx \nonumber \\
&\quad +c\snr{h}^{2\tilde \alpha(\kkk, \delta)}M^{\mf{s}(q-p)}\left(\nr{D\ti{u}}_{L^{\mf{m}}(\mathcal{B}_{1})}^{p\mf{s}}+1\right)\rr^{2\alpha\kkk}\int_{8B_{h}}H^{\textnormal{\texttt{b}}_{\kkk}}\dx\,,
\end{flalign}
where it is again $c\equiv c(n,p,q,\ti{L})$ and $\tilde \alpha(\kkk, \delta)$ is defined in \rif{bkm}. The choice in \rif{sceltachoose} eventually determines the upper bound for $\snr{h}$ as dictated by \rif{sogliah}.  

{\em Step 2: Covering and conclusion}. We glue estimates \eqref{26.1} via a covering argument.  
To this aim, we take a standard lattice $\mathcal L_{\snr{h}^{\beta_{0}}/\sqrt{n}}$ of open, disjoint hypercubes $$\{Q_{\snr{h}^{\beta_{0}}/\sqrt{n}}(y)\}_{y\in (2\snr{h}^{\beta_{0}}/\sqrt{n})\mathbb Z^n}$$ as in \rif{lattice}. From this lattice we select those hypercubes centred at points $$\{x_{\gamma}\}_{\gamma\le \mathfrak{n}}\subset (2\snr{h}^{\beta_{0}}/\sqrt{n})\mathbb Z^n\,, \quad \mbox{such that} \ \snr{x_\gamma} \leq  1/2+ 2|h|^{\beta_{0}}$$  (with $\mathfrak n\approx_{n} \snr{h}^{-n\beta_{0}}$ being the number of such points) and we determine the corresponding family $\{Q(\gamma)\equiv Q_{\snr{h}^{\beta_{0}}/\sqrt{n}}(x_\gamma)\}_{\gamma \leq  \mathfrak{n} }$. This family covers $\mathcal B_{1/2}$ up to a zero measure set (the sides of the cubes). Indeed note that if $\snr{x} >  1/2+ 2|h|^{\beta_{0}}$, then $Q_{\snr{h}^{\beta_{0}}/\sqrt{n}}(x) \cap \mathcal B_{1/2}=\emptyset$ as in fact $B_{\snr{h}^{\beta_0}}(x)\cap \mathcal B_{1/2}=\emptyset$ and $Q_{\snr{h}^{\beta_{0}}/\sqrt{n}}(x)  \subset B_{\snr{h}^{\beta_0}}(x)$.  All in all, we have 
\begin{flalign}\label{11.1}
\left| \  \mathcal B_{1/2}\setminus \bigcup_{\gamma\le \mathfrak{n}}Q(\gamma) \ \right|=0,\quad Q(\gamma_{1})\cap Q(\gamma_{2})=\emptyset \ \Longleftrightarrow \ \gamma_{1}\not =\gamma_{2}\,.
\end{flalign}
The family $\{Q(\gamma)\}_{\gamma \leq  \mathfrak{n} }$ can be realized as the family of inner cubes of the family of  balls  $\{B(\gamma)\equiv  B_{\snr{h}^{\beta_{0}}}(x_{\gamma}) \}_{\gamma\le \mathfrak{n}}$, i.e., $Q(\gamma)\equiv Q_{\textnormal{inn}}(B(\gamma))$ in the sense of \rif{inner} for every $\gamma \leq \mathfrak n$, and each ball $B(\gamma)$ is of the type considered in \rif{26.1}. By construction $
8B(\gamma)\Subset \mathcal B_1$ holds for every $\gamma$, and each of the dilated 
balls $8B(\gamma)$ intersects the similarly dilated ones fewer than $\mathfrak{c}_n$ times, which is a number depending only on $n$ (uniform finite intersection property). This can easily be checked considering the corresponding family of outer hypercubes $\{Q_{\snr{h}^{\beta_{0}}}(x_\gamma)\}_{\gamma \leq  \mathfrak{n} }$ and observing that it has the same property  and that $B(\gamma)\subset Q_{\snr{h}^{\beta_{0}}}(x_\gamma)$. The above properties imply that 
\eqn{sommamis}
$$
\sum_{\gamma\le \mf{n}}\lambda(8B(\gamma)) \leq \mathfrak{c}_n
\lambda(\mathcal B_{1})
$$
holds for every Borel measure $\lambda(\cdot)$ defined on $\mathcal B_{1}$. Needless to say, by \rif{11.1} also $\{B(\gamma)\}_{\gamma \leq  \mathfrak{n} }$ is a (measure) covering of $ \mathcal B_{1/2}$, i.e., 
\eqn{11.11}
$$
\left| \  \mathcal B_{1/2}\setminus \bigcup_{\gamma\le \mathfrak{n}}B(\gamma) \ \right|=0\,.
$$ 
We then rewrite estimate \eqref{26.1} on balls $B_{h}\equiv B(\gamma)$, and sum up over $\gamma$, in order to get
\begin{eqnarray*}
&& \int_{\mathcal B_{1/2}}\snr{\tau_{h}(E_{\mu}(D\ti{u})-\kk)_{+}}^{2}\dx\\
&&\qquad \stackrel{\eqref{11.11}}{\le}\sum_{\gamma\le \mf{n}}\int_{B(\gamma)}\snr{\tau_{h}(E_{\mu}(D\ti{u})-\kk)_{+}}^{2}\dx\nonumber \\
&& \qquad \stackrel{\eqref{26.1}}{\le} c\snr{h}^{2\tilde \alpha(\kkk, \delta)}M^{\mf{s}(q-p)}\sum_{\gamma\le \mf{n}}\int_{8B(\gamma)}(E_{\mu}(D\ti{u})-\kk)_{+}^{2}\dx\nonumber \\
&&\qquad \qquad \ \ +c\snr{h}^{2\tilde \alpha(\kkk, \delta)}M^{\mf{s}(q-p)}\left(\nr{D\ti{u}}_{L^{\mf{m}}(\mathcal{B}_{1})}^{p\mf{s}}+1\right) \rr^{2\alpha\kkk}\sum_{\gamma\le \mf{n}}\int_{8B(\gamma)}H^{\textnormal{\texttt{b}}_{\kkk}}\dx\nonumber \\
&&\qquad \stackrel{\eqref{sommamis}}{\le}c\snr{h}^{2\tilde \alpha(\kkk, \delta)}M^{\mf{s}(q-p)}\int_{\mathcal{B}_{1}}(E_{\mu}(D\ti{u})-\kk)_{+}^{2}\dx\nonumber \\
&&\qquad \qquad \ \ +c\snr{h}^{2\tilde \alpha(\kkk, \delta)}M^{\mf{s}(q-p)}\left(\nr{D\ti{u}}_{L^{\mf{m}}(\mathcal{B}_{1})}^{p\mf{s}}+1\right)\rr^{2\alpha\kkk}\int_{\mathcal{B}_{1}}H^{\textnormal{\texttt{b}}_{\kkk}}\dx\,.
\end{eqnarray*}
Note that \eqref{bkm} yields
\begin{flalign*}
\nr{D\ti{u}}_{L^{\mf{m}}(\mathcal{B}_{1})}^{p\mf{s}}\int_{\mathcal{B}_{1}}H^{\textnormal{\texttt{b}}_{\kkk}}\dx \leq c\left(\int_{\mathcal{B}_{1}}H^{(\mf{m}+p\mf{s})/2}\dx\right)^{\frac{p\mf{s}}{\mf{m}+p\mf{s}}} \int_{\mathcal{B}_{1}}H^{\mf{m}/2}\dx 
\leq  c\int_{\mathcal{B}_{1}}H^{\mf{p}/2}\dx
\end{flalign*}
and here $c\equiv c (n)$. 
Combining the content of the last two displays yields
\begin{flalign}
\notag \int_{\mathcal B_{1/2}}\snr{\tau_{h}(E_{\mu}(D\ti{u})-\kk)_{+}}^{2}\dx & \leq \bar{c}\snr{h}^{2\tilde \alpha(\kkk, \delta)}M^{\mf{s}(q-p)}\int_{\mathcal{B}_{1}}(E_{\mu}(D\ti{u})-\kk)_{+}^{2}\dx \nonumber \\
&\qquad +\bar{c}\snr{h}^{2\tilde \alpha(\kkk, \delta)}M^{\mf{s}(q-p)}\rr^{2\alpha\kkk}\int_{\mathcal{B}_{1}}H^{\mf{p}/2}\dx \label{26.2}
\end{flalign}
 and $\bar{c}\equiv \bar{c}(n,p,q,\ti{L})\geq 1$. This holds whenever $h \in \er^n$ satisfies \rif{sogliah} and $\beta_{0}$ is defined in \rif{sceltachoose},and in particular whenever $h$ satisfies 
 $$
 \snr{h} < \frac 1{2^{16}}  \left(\, \leq \frac{1}{2^{8(1+\alpha\kkk-p\delta)}} \right)\,.
 $$
Inequality \rif{26.2} implies \rif{cru1} with $m\equiv 2$, $s\equiv \tilde \alpha(\kkk, \delta)$ and $w \equiv (E_{\mu}(D\ti{u})-\kk)_{+}/\mathcal H_\eps$, where 
$$
\mathcal H_\eps \equiv \sqrt{\bar{c}}M^{\mf{s}(q-p)/2}\left(\nr{(E_{\mu}(D\ti{u})-\kk)_{+}}_{L^{2}(\mathcal{B}_{1})}+\rr^{\alpha\kkk}\nr{H^{\mf{p}/2}}_{L^{1}(\mathcal{B}_{1})}^{1/2}\right)+\eps\,,
$$
for every $\eps\in (0,1)$. Applying Lemma \ref{l4} we obtain $$(E_{\mu}(D\ti{u})-\kk)_{+}\in W^{\beta,2}(\mathcal  B_{1/2}) \quad \mbox{whenever $0< \beta <\tilde \alpha(\kkk, \delta)$}\,.$$
Scaling back \rif{cru2} from $w$ to $(E_{\mu}(D\ti{u})-\kk)_{+}$, and eventually letting $\eps \to 0$, we conclude with 
\begin{flalign}
\notag \nr{(E_{\mu}(D\ti{u})-\kk)_{+}}_{W^{\beta,2}(\mathcal B_{1/2})}&\leq \frac{c(M^{\mf{s}(q-p)/2}+1)}{(\tilde \alpha(\kkk, \delta)-\beta)^{1/2}}\nr{(E_{\mu}(D\ti{u})-\kk)_{+}}_{L^{2}(\mathcal{B}_{1})}\nonumber \\
&\quad +\frac{cM^{\mf{s}(q-p)/2}}{(\tilde \alpha(\kkk, \delta)-\beta)^{1/2}}\rr^{\alpha\kkk}\nr{\snr{D\ti{u}}+1}_{L^{\mf{p}}(\mathcal{B}_{1})}^{\mf{p}/2}\,,\label{fracci1}
\end{flalign}
where $\chi$ and $\beta$ obey \rif{bkm} and $c\equiv c(n,p,q,\alpha,\ti{L})$. 
We recall that if $w\in W^{\beta,2}(\mathcal B_{1/2})$ with $\beta \in (0,1)$, then fractional Sobolev embedding holds in the form
$$ 
\| w \|_{L^{\frac{2n}{n-2\beta}}(\mathcal B_{1/2})}\lesssim_{n,\beta}\| w \|_{W^{\beta,2}(\mathcal B_{1/2})}\,,
$$
see for instance \cite[Section 6]{guide}. This and \rif{fracci1} then provide
\begin{flalign}
\notag \nr{(E_{\mu}(D\ti{u})-\kk)_{+}}_{L^{2\chi}(
\mathcal B_{1/2})}&\leq \frac{c(M^{\mf{s}(q-p)/2}+1)}{(\tilde \alpha(\kkk, \delta)-\beta)^{1/2}}\nr{(E_{\mu}(D\ti{u})-\kk)_{+}}_{L^{2}(\mathcal{B}_{1})}\nonumber \\
&\quad +\frac{cM^{\mf{s}(q-p)/2}}{(\tilde \alpha(\kkk, \delta)-\beta)^{1/2}}\rr^{\alpha\kkk}\nr{\snr{D\ti{u}}+1}_{L^{\mf{p}}(\mathcal{B}_{1})}^{\mf{p}/2},\label{fracci2}
\end{flalign}
where $\chi$ and $\beta$ obey \rif{bkm}$_1$ and $c\equiv c(n,p,q,\alpha,\ti{L},\beta)$. Summing \rif{fracci1} and \rif{fracci2}, recalling that $M\geq 1$ so that $M^{\mf{s}(q-p)/2}+1 \leq 2M^{\mf{s}(q-p)/2}$, and finally scaling back to $u$, we obtain \eqref{28} and the proof of Lemma \ref{fracacci} is complete. \end{proof}
\begin{remark}[Tracking parameters]\label{sceltine} {\em All conditions in \rif{bkm} can be fulfilled as follows, starting from $\kkk, \delta$ and $\beta$. Once $\kkk \in (0,1)$ is fixed, we can choose $\delta$ in the range $0 < \delta <\alpha\kkk/p$, and finally, this allows to choose $\beta$ in the range $0<\beta<\tilde \alpha(\kkk, \delta)$. The (large) exponent $\ppp$ is then determined after $\kkk$ and $\delta$ have been fixed, via the intermediate parameters $\mf{m}, \textnormal{\texttt{b}}_{\kkk}$.}
\end{remark}

\subsection{Proof of \rif{hhhl} when $n \geq 4$ or $n= 2$}\label{proof3} There is no loss of generality in assuming 
\eqn{recalla}
$$0< \alpha <1\,.$$ Indeed, \rif{pq} involves an open bound on $q/p$, so that the case $\alpha=1$ is automatically implied  by the one in \rif{recalla} with $\alpha$ sufficiently close to $1$ (i.e., if $q/p< 1+ 1/n$, then we can find $\alpha <1$ such that  \rif{pq} holds). In this respect also note that, since all the a priori estimates we are going to derive are on balls $B_{\rr}$ with $\rr \leq 2$, then  \rif{assfr}$_4$ for $\alpha=1$ implies the same condition (with $\ti{L}$ replaced by $4^{1-\alpha}\ti{L}$) for any $\alpha<1$.  

We start applying Lemma \ref{fracacci} for general parameters $\kkk, \delta$ and $\beta$ (and therefore $\ppp$) as in \rif{bkm} to be chosen later;  see also Remark  \ref{sceltine}. Note that there is no loss of generality in assuming $\nr{Du}_{L^{\infty}(B_{\rr/4})}\ge 1$, otherwise  \rif{hhhl} follows trivially. We take a ball $B_{2\rr}\Subset \Omega$ as in \rif{hhhl}, and concentric balls $B_{\rr/4}\Subset  B_{\tau_{1}}\Subset   B_{\tau_{2}}\Subset B_{\rr/2}$. We set $r_{0}:=(\tau_{2}-\tau_{1})/8$ so that, for any arbitrary point $x_{0}\in B_{\tau_1}$, it follows $B_{2r_{0}}(x_{0})\Subset B_{\tau_{2}}$. By means of \eqref{28}, that we use with $M:= \nr{Du}_{L^{\infty}(B_{\tau_{2}})}\ge 1$, we are able to apply Lemma \ref{revlem} on $B_{r_{0}}(x_{0})$ with $\kk_{0}=0$, $M_{0}\equiv M_{*}:= M^{\mf{s}(q-p)/2}$, $f:= (\snr{Du}+1)^{\mf{p}}$, $\sigma:= \alpha\kkk$, $\vartheta:= 1/2$, $\chi:=\chi(\beta)>1$ to deduce that
\begin{flalign}\label{29}
\notag    E_{\mu}(Du(x_{0}))&\le c\nr{Du}_{L^{\infty}(B_{\tau_{2}})}^{\mf{t}p}\left(\mint_{B_{r_{0}}(x_{0})}E_{\mu}(Du)^{2}\dx\right)^{1/2}\\
    & \quad +c\nr{Du}_{L^{\infty}( B_{\tau_{2}})}^{\mf{t}p}\mathbf{P}^{1/2}_{\alpha\kkk}\left((\snr{Du}+1)^\mf{p};x_{0},2r_{0}\right),
\end{flalign}
where $c\equiv c(n,p,q,\alpha,\ti{L}, \kkk, \delta, \beta)$ and 
\eqn{ext} 
$$
\mf{t}:=\frac{\mf{s}}{2}\frac{\chi(\beta)}{\chi(\beta)-1}\left(\frac qp-1\right)=\frac{\sss n}{4\beta}\left(\frac qp-1\right)\,.
$$
As $x_{0}\in B_{\tau_{1}}$ is arbitrary, noting that $\snr{Du}^p \leq p \snr{E_{\mu}(Du)} + \mu^p$, a few standard manipulations lead to 
\begin{flalign}\label{30.10}
\nr{Du}_{L^{\infty}(B_{\tau_{1}})}&\le\frac{c\nr{Du}_{L^{\infty}(B_{\tau_{2}})}^{\mf{t}}}{(\tau_{2}-\tau_{1})^{\frac{n}{2p}}}\left(\int_{B_{\rr/2}}(\snr{Du}^2+1)^{p}\dx\right)^{\frac{1}{2p}}\nonumber \\
&\quad +c\nr{Du}_{L^{\infty}(B_{\tau_{2}})}^{\mf{t}}\left\|\mathbf{P}^{1/2}_{\alpha\kkk}\left((\snr{Du}+1)^\mf{p};\cdot,2r_{0}\right)\right\|_{L^{\infty}(B_{\tau_{1}})}^{1/p}+c\,,
\end{flalign}
with $c\equiv c(n,p,q,\alpha,\ti{L}, \kkk,\delta, \beta)$. Note that, in order to apply Lemma \ref{revlem}, we have used that all points are Lebesgue points of $E_{\mu}(Du)$ by \eqref{areg}. 
Since $n>2\alpha\kkk$, Lemma \ref{crit} applies and gives
\eqn{30.2}
$$
\left\|\mathbf{P}^{1/2}_{\alpha\kkk}\left((\snr{Du}+1)^\mf{p};\cdot,2r_{0}\right)\right\|_{L^{\infty}(B_{\tau_{1}})}\le c\nr{\snr{Du}+1}_{L^{\mf{q}}(B_{\rr/2})}^{\mf{p}/2}\qquad \mbox{for all} \ \ \mf{q}>\frac{n\mf{p}}{2\alpha\kkk}\,,
$$
with $c\equiv c(n,\alpha,\kkk, \mf{q})$, and therefore from now on we fix 
\eqn{ququ0}
$$\mf{q}\equiv \mf{q}(m,p,q,\alpha,\delta, \kkk):=\frac{n\mf{p}}{\alpha\kkk}\,. $$ 
Merging \rif{30.10} and \rif{30.2}, we arrive at
\begin{flalign}\label{30.1}
\nr{Du}_{L^{\infty}(B_{\tau_{1}})}&\le\frac{c\nr{Du}_{L^{\infty}(B_{\tau_{2}})}^{\mf{t}}}{(\tau_{2}-\tau_{1})^{\frac{n}{2p}}}\left(\int_{B_{\rr/2}}(\snr{Du}^2+1)^{p}\dx\right)^{\frac{1}{2p}}\nonumber \\
&\quad +c\nr{Du}_{L^{\infty}(B_{\tau_{2}})}^{\mf{t}}\nr{\snr{Du}+1}_{L^{\mf{q}}(B_{\rr/2})}^{\frac{\mf{p}}{2p}}+c
\end{flalign} 
with $c\equiv c(n,p,q,\alpha,\ti{L}, \kkk,\delta, \beta)$. The constant $c$ in \rif{30.1} tends to infinity as $\beta\to \tilde \alpha(\kkk, \delta)$, this behaviour originally coming from the factor $1/(\tilde \alpha(\kkk, \delta)-\beta)^{1/2}$ appearing in \rif{28}. Keeping Remark \ref{sceltine} in mind, we now choose the parameters $\kkk, \delta$ and $\beta$. By the very definition of $\tilde \alpha(\kkk, \delta)$ in \rif{bkm} we have 
\eqn{limiti}
$$ \lim_{(\kkk,\delta) \to (1,0)} \tilde \alpha(\kkk, \delta) = \frac{\alpha}{1+\alpha} \stackrel{\rif{recalla}}{>} \frac \alpha 2\,.
$$
We can therefore find 
 \eqn{cappone}
 $$\kkk, \delta\equiv \kkk, \delta(p, \alpha)\in (0,1) \quad \mbox{such that} \quad  
 \frac{\alpha}{2 \tilde \alpha(\kkk, \delta)}<1\,.$$ 
This allows to find $\beta\equiv \beta(p, \alpha)<\tilde  \alpha(\kkk, \delta)$ as in \rif{bkm}$_1$, close enough to $\tilde  \alpha(\kkk, \delta)$, in order to meet
\eqn{bbb}
$$
\frac{\alpha}{2\beta}<1\,.
$$
As also noted in Remark \ref{sceltine}, the choice of $\kkk$ and $\delta$ also determines the numbers 
\eqn{pupupu}
$$\mf{p}, \qqq\equiv \mf{p}, \qqq(n,p,q,\alpha)$$
 from \rif{bkm}$_3$ and \rif{ququ0}. Such choices finally fix the constant $c\equiv c(n,p,q,\alpha,\ti{L})$ in \rif{30.1}. Observe that up to \rif{pupupu} this proof works in the full range $n\geq 2$. Next, recalling \rif{sss} and \rif{ext}, we rewrite $\mf{t}$ as
\eqn{ttt}
$$
\mf{t}=\frac{q/p}{n+1-(n-1)q/p}\cdot \frac{n(q/p-1)}{2\beta}\,.$$
Observe that
\eqn{obs1}
$$
\eqref{pq} \ \Longrightarrow \frac qp<1+\frac{1}{n} \ \Longleftrightarrow \frac{q/p}{n+1-(n-1)q/p}<1
$$
and
\eqn{obs2}
$$
\ \frac{n(q/p-1)}{2\beta}\stackrel{\eqref{pq}}{<}\frac{\alpha}{2\beta}\stackrel{\eqref{bbb}}{<} 1\,.
$$
This means that $\mf{t}<1$, so we can use Young's inequality in \eqref{30.1} concluding with
\begin{flalign}
\notag \nr{Du}_{L^{\infty}(B_{\tau_{1}})} &\le\frac{1}{2}\nr{Du}_{L^{\infty}(B_{\tau_{2}})}
+\frac{c}{(\tau_{2}-\tau_{1})^{\frac{n}{2p(1-\mf{t})}}}\left(\int_{B_{\rr/2}}(\snr{Du}^2+1)^{p}\dx\right)^{\frac{1}{2p(1-\mf{t})}}\\& \qquad  +c\nr{Du}_{L^{\mf{q}}(B_{\rr/2})}^{\frac{\mf{p}}{2p(1-\mf{t})}}+c\label{comedopo}
\end{flalign}
where $c\equiv c(n,p,q,\alpha,\ti{L})$ (note that Young's inequality is actually not needed when $q=p$, as then $\mf{t}=0$). Lemma \ref{iterlem} applied with $r_1\equiv \rr/4$, $r_2\equiv \rr/2$ and $\hhh(\tau)\equiv \nr{Du}_{L^{\infty}(B_{\tau})}$, which is always finite thanks to \rif{areg}, now yields 
\eqn{after}
$$
\nr{Du}_{L^{\infty}(B_{\rr/4})}\le c\left(\mint_{B_{\rr/2}}\snr{Du}^{2p}\dx\right)^{\frac{1}{2p(1-\mf{t})}}+c\nr{Du}_{L^{\mf{q}}(B_{\rr/2})}^{\frac{\mf{p}}{2p(1-\mf{t})}}+c\,, 
$$
with $c\equiv c(n,p,q,\alpha,\tilde{L})$, where $\mf{p}, \qqq$ in \rif{pupupu} and $\mf{t}$ in \rif{ext} are all functions of $n,p,q,\alpha$. The right-hand side in the above display can be then estimated using Proposition \ref{scat}, eventually leading to \rif{hhhl} for a suitable (large) exponent $\mf{b}\equiv \mf{b}(n,p,q,\alpha)$, after a few standard manipulations. 
\subsection{Proof of \rif{hhhl} when $n=3$}\label{proof4} The proof is essentially the same of Section \ref{proof3} and it is exactly the same when $q=p$. Indeed, when $n=3$ and $q>p$, recalling \rif{sss} and \rif{ext}, we just have to replace \rif{ttt} by
\eqn{ttteps}
$$
\mf{t}=\left(\frac{q/p}{n+1-(n-1)q/p}+\eps\right)\cdot \frac{n(q/p-1)}{2\beta}=
\left(\frac{q/p}{4-2q/p}+\eps\right)\cdot \frac{3(q/p-1)}{2\beta}$$
for every $\eps > 0$. As  
$$\frac{q/p}{4-2q/p} <1 \Longleftrightarrow \frac qp<1+\frac 1n=\frac 43\,,$$ 
we can pick $\eps>0$ such that
$$
\frac{q/p}{4-2q/p}+\eps<1
$$
so that, as in \rif{obs1}-\rif{obs2}, again we conclude with $\mf{t}<1$ and the rest goes as for the previous case.

\subsection{Two-steps improvement of the gap bounds}\label{proof6} Here we comment on the basic ingredients allowing to upgrade the previously known bound \rif{vecchiobound} from \cite{piovra} to the sharp one \rif{pq}. The first, and actually most important step is to improve \rif{vecchiobound} in something of the form
\eqn{vecchiobound2}
$$
\frac qp   <1+ s \frac{\alpha}{n}\,, \qquad \quad  s \in (0,1)
$$
thereby getting a linear rather than quadratic decay of the gap bound in terms of $\alpha/n$. This the {\em structural} improvement. The crucial ingredient in order to achieve \rif{vecchiobound2} is the higher integrability result of Proposition \ref{scat}. Once this result is secured, in order to arrive at \rif{pq} we build on the methods of \cite{piovra}, that necessitate a different use of certain a priori estimates for minima of autonomous functionals;  see, in particular, the treatment to get \rif{25}.  Anyway, reaching the sharp bound in \rif{pq} still requires a last ingredient, namely, \rif{lippi} with the specific exponent $\mf{s}$ in \rif{sss}. This is the {\em fine} improvement. Estimate \rif{lippi} is due to Bella \& Schäffner \cite{BS01, BS0, BS1, schag} and improves a classical one of Marcellini  \cite{BM, piovra, ma2, masurvey}; moreover, it requires the gap bound in \rif{pq0}, which is weaker than the one required by Marcellini. Specifically, Bella \& Schäffner's results allow to use the exponent $\mf{s}$ in \rif{sss}, while the exponent previously used in the literature was the slightly larger
\eqn{sssbis} 
$$
\mf{s}_{\textnormal{old}}\equiv \begin{cases}
\displaystyle 
\, \frac{2q}{(n+2)p-nq}\quad \mbox{if} \ \ n\ge 3\vspace{3mm}  \\
\displaystyle
\, \mbox{any number} > \frac{q}{2p-q}\qquad \mbox{if $n=2$ and $q>p$}
\vspace{3mm}\\ 
\displaystyle
\, 1 \qquad \mbox{if $n=2$ and $q=p$}\,.
\end{cases} 
$$
Employing $\mf{s}_{\textnormal{old}}$ instead of the one in \rif{sss} would lead to 
$
q/p   < 1+  \alpha/(n+1),
$
which is still of the type in \rif{vecchiobound2}, but not yet the sharp one in \rif{pq}.  
\subsection{Proof of Theorem \ref{main} completed}\label{proof7} 
\hfill \break 

{\em Proof of \eqref{m.1} and of the $\mathcal F$-minimality of $u$.} We employ $\App{B_{\rr}}$ from Section \ref{appisec}. Using \rif{hhhlbissa} (in a manner similar to that described in Remark \ref{scemino}) in conjunction to \rif{c.4}, we obtain
\eqn{ggii1-dopo}
$$
\nr{Du_{\eps,\delta}}_{L^{\infty}(B_{\rr/2})} \le c \rr^{-n\mf{b}}\left[\bar{\mathcal{F}}(u,B_{\rr})+\texttt{o}(\varepsilon,B_{\rr})+\texttt{o}_{\varepsilon}(\delta,B_{\rr})\right]^{\mf{b}}+c\,,
$$ where $c\equiv c(n,p,q,\alpha,L)\geq 1$, $\mf{b}\equiv \mf{b}(n,p,q,\alpha)\geq 1$, so that \rif{m.1} follows letting first $\delta \to 0$ and then $\eps \to 0$ in \rif{ggii1-dopo} \footnote{The full, standard argument is indeed as follows. Note that \rif{ggii1-dopo} trivially implies 
$$
\left(\mint_{B_{\rr/2}}\snr{Du_{\eps,\delta}}^{\qqq}\dx\right)^{1/\qqq} \le c \rr^{-n\mf{b}}\left[\bar{\mathcal{F}}(u,B_{\rr})+\texttt{o}(\varepsilon,B_{\rr})+\texttt{o}_{\varepsilon}(\delta,B_{\rr})\right]^{\mf{b}}+c\,,
$$
for every $\qqq>1$. Then we use weak lower semicontinuity of convex functionals to let first $\delta \to 0$ and then $\eps \to 0$. We finally let $\qqq\to \infty$ in the resulting inequality, and this yields \rif{m.1}.}. A standard covering argument then implies that $Du \in L^{\infty}_{\loc}(\Omega;\er^n)$. Let now ${\rm B}\Subset \Omega$ be a ball and $\varphi \in C^{\infty}_0({\rm B})$. For $t \in \er$ define the function $g(t):=\mathcal F(u+t\varphi,{\rm B} )$, which is differentiable at zero and finite by $Du \in L^{\infty}({\rm B};\er^n)$. Lemma \ref{bf.0} and the $\bar{\mathcal F}$-minimality of $u$ imply $g(0) =\mathcal F(u,{\rm B}) = \bar{\mathcal F}(u,{\rm B})\leq  \bar{\mathcal F}(u+t\varphi,{\rm B} )=\mathcal F(u+t\varphi,{\rm B} )=g(t)$. It follows $g'(0)=0$, that is, the Euler-Lagrange equation $\diver\, \partial_z F(x, Du)=0$ of the original functional $\mathcal F$ is satisfied in ${\rm B}\Subset \Omega$. Moreover, again $Du \in L^{\infty}({\rm B};\er^n)$ and a standard density argument imply that 
\eqn{elel3}
$$
\int_{\rm B}\partial_z F(x, Du)\cdot D\varphi\dx =0
$$
actually holds whenever $\varphi \in W^{1,1}_0({\rm B})$. Convexity of $z  \mapsto F(\cdot, z)$ then implies that $u$ is a minimizer of $\mathcal F$ in the sense of Definition \ref{defi-min}. Indeed  
$$
\mathcal F(w,{\rm B}) \geq \mathcal F(u,{\rm B})+  \int_{B}\partial_z F(x, Du)\cdot (Dw-Du)\dx \stackrel{\eqref{elel3}}{=}  \mathcal F(u,{\rm B})
$$
holds whenever $w \in u+W^{1,1}_0({\rm B})$\,. 

\vspace{2mm}

{\em Proof of \eqref{m.2}.} We take $B_{r}\equiv B_{r}(x_{\rm c})$, with  $x_{\rm c} \in B_{\theta \rr}$, and $0< r= (1-\theta)\rr/4$, so that $B_{r} \Subset B_{\rr}$. We use $\App{B_r}$ and, similarly to \rif{ggii1-dopo}, we get
\begin{flalign}
\notag\nr{Du_{\eps, \delta}}_{L^{\infty}(B_{\tau})}&\le cr^{-n\mf{b}}\left[\bar{\mathcal{F}}(u,B_{\rr})+1\right]^{\mf{b}}+ cr^{-n\mf{b}}[\texttt{o}(\varepsilon,B_{r}) + \texttt{o}_{\eps}(\delta,B_{r})]\label{stim1bis}
\\ & =:M+ cr^{-n\mf{b}}[\texttt{o}(\varepsilon,B_{r}) + \texttt{o}_{\eps}(\delta,B_{r})]\,,
\end{flalign}
where 
$B_{\tau}(x_{\rm c})\Subset B_{r}(x_{\rm c})$ are concentric, and $\tau \leq r/2$; note here that the minimizers $u_{\eps, \delta}$ are only defined on $B_{r}(x_{\rm c})$. 
We have used that $A \mapsto \bar{\mathcal{F}}(u,A)$ is non-decreasing set function in order to estimate $\bar{\mathcal{F}}(u,B_{r})\leq \bar{\mathcal{F}}(u,B_{\rr})$. Estimate \rif{stim1bis} is the precise analogue of \cite[(6.62)]{piovra} with $f\equiv \hhh\equiv 0$, with $B_{\rr}\equiv \Omega$ and $B_{\theta \rr}\equiv \Omega_0$ in the language of \cite{piovra}. From now on we can proceed until the end of \cite[Section 6.4]{piovra}. Specifically, also recalling \cite[Section 7.3]{piovra} and arguing as in \cite[pp. 1174-1175]{piovra}, we obtain the following analogue of \cite[(6.71)]{piovra}:
$$
 \mint_{B_s(x_{\rm c})}|Du-(Du)_{B_s(x_{\rm c})}|^p   \dx
\leq  c  s^{p\alpha_*} \,, \quad \mbox{for every $s \leq ((1-\theta)\rr/16)^{1+p/n}$}\,,$$
where $c, \alpha_*$ are as in \rif{m.2}. Being $x_{\rm c}\in B_{\theta\rr}$ arbitrary, the above inequality implies \rif{m.2} by the standard Campanato-Meyers integral characterization of H\"older continuity. Just note that, once the a priori estimate \rif{stim1bis} is available, in order to apply the arguments of \cite[Section 6.4]{piovra} and of \cite[Section 7.3]{piovra} it is sufficient to assume $q/p< 1+2/n$, which is implied by \rif{pq}, and that serves to gain\cite[(6.65)]{piovra} (which is in fact implied by \rif{lip0.0} and \rif{stim1bis} here). The more restrictive bounds on $q/p$ otherwise assumed in \cite{piovra} are not needed in \cite[Section 6.4]{piovra} once \rif{stim1bis} is available, as derived here by different means from those in \cite{piovra}.

\vspace{2mm}
{\em Proof of \eqref{m.3}.} 
With $\bar \theta := (1+\theta)/2\in (\theta, 1)$, estimates \rif{m.1}-\rif{m.2} imply $
\|Du\|_{L^{\infty}(B_{\bar \theta \rr})} +[Du]_{0, \alpha_*;B_{\bar \theta \rr}}  \leq M,$ 
where $M, \alpha_*$ depend on $n,p,q,\alpha,L, \theta, \rr$ and $ \bar{\mathcal{F}}(u,B_{\rr})$. We can now apply Proposition \ref{piovrapert} to $\diver\, \partial_z F(x, Du)=0$, with $\Omega \equiv B_{\bar \theta \rr}$ and $\Omega_0\equiv B_{\theta \rr}$, thereby concluding with \rif{m.3}, with the described dependence on the constants.

\begin{remark}\label{remarkfinale}{\em In the proof of \rif{m.3} we can replace the assumption of continuity of $\partial_{zz}F(\cdot)$ made in Theorem \ref{main}, with a weaker one, that is, exactly as in Proposition \ref{piovrapert}, we can assume that for every $H>0$ there exists a modulus of continuity $ \omega_{H}(\cdot)$ such that
\rif{modulus} holds with $A(\cdot)\equiv \partial_zF(\cdot)$ for every $x\in \Omega$. We also note that the proof of the $\mathcal F$-minimality of relaxed minimizers rests of the validity of the Euler-Lagrange equation \rif{elel3}, specifically for every test function $\varphi \in W^{1,1}_0(B)$. This result adds up to a large literature devoted to the validity of the Euler-Lagrange equation under non-standard growth conditions, see for instance \cite{menita} and related references.}
\end{remark} 
\begin{remark}\label{bere}{\em The Lavrentiev phenomenon and condition \rif{pq} are deeply intertwined, the connection passing through the LSM-relaxation from Definition \ref{LSMdef}. This approach, pioneered in \cite{ma0, madg, ma1}, is nowadays standard. Indeed, consider the so-called double phase functional \cite{CM}
\eqn{ildoppio}
$$
w \mapsto \mathcal D(w,  \mathcal B_1):=\int_{ \mathcal B_1} \left(\snr{Dw}^p+\aax(x)\snr{Dw}^q\right)\dx \,, \qquad 0\leq \aax(\cdot) \in C^{0, \alpha}(\Omega)\;.
$$ 
This is still uniformly elliptic since its ellipticity ratio is uniformly bounded and its integrand only satisfies a milder form of nonuniform ellipticity described in \cite[Section 2.2.4]{cimemin} and originally introduced in \cite{ciccio}. If $q/p\leq 1+\alpha/n$, then $ \mathcal D$ coincides with its LSM-relaxation $ \mathcal D=\tilde{\mathcal D}$ as commented in connection to \rif{implica-2}. Moreover, such a condition on $q/p$ guarantees the local H\"older gradient continuity of minima of $\mathcal D$ \cite{BCM, CM}. On the other hand, the existence of bounded, yet discontinuous minimizers that do not belong to $W^{1,q}_{\loc}$ it is shown in \cite{sharp,FMM} provided $p<n<n+\alpha <q$. This last condition always implies $q/p>1+\alpha/n$. Simultaneously, the Lavrentiev phenomenon \rif{l.0} appears. This also shows the optimality of Theorem \ref{main3} when $p\geq 2$. Note also that in the case $q/p >1+\alpha/n$ the LSM-relaxation $ \tilde{\mathcal D}$ stops being a variational integral and features a representation via potentially infinite singular measures \cite[Section 7]{MiMu} that makes it very different from the kind of functionals that we are considering here. Similar representations involving singular measures occur in the vectorial setting in the context of cavitation \cite{BFM, FMa}; see \cite{ABF} for the convex case. It is exactly the occurrence of singular parts in the measure representation to make the LSM-relaxation of a functional particularly suitable when modelling problems featuring singular solutions, as for instance in cavitation \cite{ma0, madg}. Of course, this is only possible when conditions such as \rif{pq} are violated. The borderline case of the condition \rif{pq}, namely $q/p=1+\alpha/n$, is not covered by Theorem \ref{main} and deserves a few words. Missing limit cases of the gap bounds is a typical outcome when establishing Hölder continuity of solutions to problems involving nonuniform ellipticity, as seen in \cite{BS0, BS01, BS1, ma1, ma2, masurvey}. This is essentially due to the level of generality of the problems considered. In certain special situations, additional structure properties allow to catch the borderline case. This is for instance the case of the functional $\mathcal D$ in \rif{ildoppio}, being ultimately a consequence of the fact that $\mathcal D$ is uniformly  
elliptic and this allows for some additional higher integrability in Gehring's style, which is otherwise unachievable in general cases as for instance considered in Theorem \ref{main}}. 
\end{remark}

\end{document}